\documentclass[final]{siamltex}
\usepackage{}
\usepackage{graphicx}
\usepackage{multirow,bigstrut}
\usepackage{amsmath,amsfonts,amssymb}
\usepackage{mathrsfs}
\usepackage{color}
\usepackage{upgreek}
\usepackage{bm}
\usepackage{verbatim}
\usepackage{mathtools}
\usepackage{calligra}
\usepackage[T1]{fontenc}
\usepackage{subfigure}
\usepackage{float}

\usepackage{verbatim}
\usepackage{exscale}
\usepackage{relsize}
\newtheorem{rem}{\hspace{1mm}Remark}[section]
\newtheorem{exam}{\hspace{1mm}Example}[section]
\newtheorem{case}{\hspace{1mm}Case}[section]

\usepackage{tikz}
\usetikzlibrary{shapes,arrows}

\begin{document}


\title{Efficient multiscale algorithms for simulating nonlocal optical response of metallic nanostructure arrays}
\thanks{{\small The work of Liqun Cao and Yongwei Zhang is supported by National Natural Science Foundation of China (grant 11971030, 11571353, 91330202) and (grant 11871441, 11671369), respectively.}}
\author{ Yongwei Zhang\thanks
    {School of Mathematics and Statistics, Zhengzhou University, Zhengzhou 450001, China
    (zhangyongwei@lsec.cc.ac.cn).}
\and Chupeng Ma\thanks{{
Institute for Applied Mathematics and Interdisciplinary Center for Scientific Computing, Heidelberg University, Im Neuenheimer Feld 205, Heidelberg 69120, Germany; ({\tt machupeng@lsec.cc.ac.cn}).}}
\and Liqun Cao\thanks
    {Corresponding author.
    LSEC\@, NCMIS\@, University of Chinese Academy of Sciences,
    Institute of Computational Mathematics and Scientific\slash Engineering Computing,
    Academy of Mathematics and Systems Science, Chinese Academy of Sciences, Beijing 100190, China
    (clq@lsec.cc.ac.cn).}
\and Dongyang Shi\thanks
    {School of Mathematics and Statistics, Zhengzhou University, Zhengzhou 450001, China
    (shi\_dy@zzu.edu.cn).}}
\maketitle
\begin{abstract}
In this paper, we consider numerical simulations of the nonlocal optical response of metallic nanostructure arrays inside a dielectric host, which is of particular interest to the nanoplasmonics community due to many unusual properties and potential applications. Mathematically, it is described by Maxwell's equations with discontinuous coefficients coupled with a set of Helmholtz-type equations defined only on the domains of metallic nanostructures. To solve this challenging problem, we develop an efficient multiscale method consisting of three steps. First, we extend     the system into the domain occupied by the dielectric medium in a novel way and result in a coupled system with rapidly oscillating coefficients. A rigorous analysis of the error between the solutions of the original system and the extended system is given. Second, we derive the homogenized system and define the multiscale approximate solution for the extended system by using the multiscale asymptotic method. Third, to fix the inaccuracy of the multiscale asymptotic method inside the metallic nanostructures, we solve the original system in each metallic nanostructure separately with boundary conditions given by the multiscale approximate solution. A fast algorithm based on the $LU$ decomposition is proposed for solving the resulting linear systems. By applying the multiscale method, we obtain the results that are in good agreement with those obtained by solving the original system directly at a much lower computational cost. Numerical examples are provided to validate the efficiency and accuracy of the proposed method.
\end{abstract}

\begin{keywords}
nonlocal optical response, metallic nanostructure arrays, multiscale asymptotic method, finite element method.
\end{keywords}

\begin{AMS}
65N30, 65N55, 65F10, 65Y05
\end{AMS}

\section{Introduction}\label{sec-1}
Metallic nanostructures have attracted great interest of engineers and scientists in the last two decades due to the phenomenon of localized surface plasmon resonance (LSPR) \cite{pitarke2006theory}. When a metallic nanoparticle is irradiated by light at optical frequencies, the light wave couples with the collective oscillation of conduction band electrons at the metal surface and excites LSPR, leading to many unusual and fascinating physical properties. It enables the confinement of light at the nanoscale, the enormous local fields enhancement of the incident light, and the squeezing of light beyond the diffraction limit. As a powerful platform for manipulation
of light at the nanoscale, metallic nanostructures have found a wide range of applications in various fields, such as near-field scanning microscopy, ultrasensitive sensing and detection, solar cells, and related studies have formed a fast-growing and highly interdisciplinary field -- Nanoplasmonics \cite{maier2007plasmonics}.

In order to understand and make use of the phenomenon of LSPR, an appropriate modeling for describing the interaction between light and metal nanostructures is required. The most commonly used model for describing the optical response of metals is the Drude model (later extended into the Lorentz--Drude model) \cite{drude1900elektronentheorie}. This is a classical model based on the kinetic theory of electrons in a metal which assumes that electrons in the metal do not interact with each other and move against a positively charged ionic background due to external fields. The Drude model has achieved a great success in explaining many physical properties of bulk metals. However, as the size of metals becomes much smaller than the wavelength of the incident light, nonlocal effects due to the Pauli exclusion principle become important and the Drude model fails to explain experimentally observable phenomena. To overcome this, some more sophisticated material models, such as a nonlinear hydrodynamic model which couples macroscopic Maxwell's equations for electromagnetic fields with hydrodynamic equations for free electrons in metals, are used \cite{eguiluz1975influence}. The linear-response approximation of the hydrodynamic model yields the nonlocal hydrodynamic Drude (NHD) model, which has been a popular model in the computational study of metallic nanostructures due to its ability to describe the nonlocal optical response of metal nanoparticles with a high accuracy and a low computational effort \cite{ciraci2013hydrodynamic,raza2015nonlocal,toscano2012modified}. Mathematically, the NHD model is a coupled system of Maxwell's equations for the electric field and a Helmholtz-type equation for the polarization current.

Much effort has been devoted to solving the NHD model numerically to simulate the nonlocal optical properties of metallic nanostructures. In \cite{hiremath2012numerical}, a Galerkin finite element method based on the N\'{e}d\'{e}lec elements is proposed to solve the NHD model to simulate complex-shaped nano-plasmonic scatters. In \cite{zheng2018boundary}, a computational scheme within the framework of a boundary integral equation and a method of moments algorithm is developed for the NHD model to predict the interaction of light with metallic nanoparticles. The well-posedness and the convergence of finite element method for the NHD model are rigorously proved in \cite{ma2019mathematical}.
The hybrid discontinuous Galerkin methods for the NHD model have also been considered \cite{li2017hybridizable,schmitt2016dgtd,schmitt2018simulation,vidal2018hybridizable}. For other more numerical methods for this coupled system, we refer the reader to \cite{eremin2019numerical,huang2016theoretical,zheng2019potential} and references therein.

It is worth noting that almost all the existing numerical studies on the NHD model simulate only one or several metallic nanostructures. Compared with individual and small clusters of nanoparticles, when arranged into multidimensional arrays, metallic nanoparticles strongly interact with each other such that the collective properties can be rationally designed by changing the inter-particle spacing, array period, and overall composition, providing unparalleled opportunities for realizing materials with interesting and unusual photonic and metamaterial properties \cite{ross2016optical,wang2018rich}. Indeed, metallic nanoparticle arrays have found applications in plasmonic solar cells \cite{mokkapati2009designing} and optical sensing \cite{enoch2004optical} and the literature of related studies is fast growing in recent years. However, due to the huge computational cost, numerical simulations of optical properties of metallic nanoparticle arrays are very limited, especially in the case of three-dimensional arrays. A simple and widely used method for describing the macroscopic electromagnetic properties of composite materials is given by the Maxwell--Garnett effective medium theory (EMT) \cite{choy2015effective} in which the effective permittivity is determined by the volume fraction of the inclusions, the permittivity of the inclusions, and the background material. However, in general, the Maxwell--Garnett formula is only valid at low volume fraction. In the case of plasmon resonance, it has been reported that the Maxwell--Garnett approach is correct only at volume fraction of the inclusions lower than $10^{-5}$ \cite{belyaev2018electrodynamic}. Furthermore, in the case considered in this paper, due to the spatially-dispersive permittivity of the metallic nanostructures, the application of the Maxwell--Garnett formula will lead to a much more complicated effective system. Therefore, in order to simulate the nonlocal optical response of realistic metallic nanoparticle arrays, some accurate and efficient methods beyond the Maxwell--Garnett formula are necessary.

In this paper, we consider numerical simulations of the nonlocal optical response of metallic nanostructure arrays embedded in a dielectric medium with the NHD model. Mathematically, it involves solving a coupled system of Maxwell's equations defined on the whole domain and a set of Helmholtz-type equations defined in each metallic nanostructure separately. It is a challenging task to solve this system numerically, mainly due to two reasons. First, the coefficients of Maxwell's equations are rapidly oscillating, whose numerical solution is computationally expensive (or even impossible) in practice with a standard method, especially when there are a huge number of metallic nanostructures. Second, in contrast to many other coupled systems, this system is coupled in a novel way in which an equation defined on the whole domain is coupled with another equation defined on many unconnected domains, making it difficult to apply the classical multiscale asymptotic method directly. In order to solve this system efficiently, some special strategies to deal with such a coupling are necessary.

In this paper, we develop an efficient multiscale method to solve this coupled system, which consists of three steps. First, we extend the equation satisfied by the polarization current into the domain occupied by the dielectric medium by setting some coefficients of the equation sufficiently large outside the metallic nanostructures. The error between the solutions of the original system and the extended system is rigorously proved. In this way, we obtain a coupled system defined on the whole domain of the scatter with rapidly oscillating coefficients. Second, we use the multiscale asymptotic method to derive the homogenized system and define the multiscale approximate solution for the extended system. The coefficients of the homogenized system are constant, making it much easier to solve. In view of the fact that the polarization current has little effect on the homogenized coupled system in practice, we propose to solve the homogenized Maxwell's equations rather than the homogenized coupled system to further reduce the computational burden. Numerical results show that the multiscale asymptotic method produces satisfactory results outside the metallic nanostructures while it fails to capture the oscillations of the electric field inside the metallic nanostructures. Third, in order to remedy the failure of the multiscale asymptotic method inside the metallic nanostructures, we solve the original system in each metallic nanostructure separately with boundary conditions given by the multiscale approximate solution. Since now the system is only solved in each metallic nanostructure with constant coefficients, the computational cost is much lower. In addition, by taking advantage of the periodic arrangement of metallic nanostructures, we propose a fast algorithm based on the $LU$ decomposition for solving the resulting linear systems to further reduce the computational cost and make the algorithm applicable to the system with a huge number of metallic nanostructures. By these three steps, we get the results agreeing well with those obtained from solving the original system directly with a much lower computational effort. The multiscale method developed in this paper might be able to be generalized to simulate the optical properties of composite materials with the non-dispersive media matrix and the dispersive media inclusions, such as the Drude, Lorentz--Drude and Debye media.

The rest of this paper is organized as follows. In section~\ref{sec-2}, we briefly introduce the NHD model and describe the problem considered in this paper. The extended system is given at the end of this section. In section~\ref{sec-3}, we derive the homogenized system and define the multiscale approximate solution for the extended system by using the technique of multiscale asymptotic expansion. In section~\ref{sec-4}, we describe the multiscale approaches and the associated numerical algorithms in detail. In section~\ref{sec-5}, we present some numerical experiments to demonstrate the accuracy and efficiency of our method.

\section{Problem Formulation}\label{sec-2}
Throughout this paper, we use $\mathbf{curl}$, $\mathbf{div}$, and $\mathbf{grad}$ to denote the curl, divergence, and gradient operators, respectively.

\subsection{Nonlocal hydrodynamic Drude model}
The NHD model is derived via the linearization of the hydrodynamic model \cite{ciraci2013hydrodynamic,raza2015nonlocal}. Within the hydrodynamic model, the free electrons in a metal are considered as a charged and compressible fluid which is described in terms of the charge density $n(\mathbf{x},t)$, the electron fluid velocity ${\bf v}(\mathbf{x},t)$, and the electron pressure $p(\mathbf{x},t)$. Under the influence of macroscopic electromagnetic fields, the motion of the electron fluid is determined by the Euler equation
 \begin{equation}\label{eq:1.1}
 {\displaystyle m_{e}(\partial_{t} + {\bf v}\cdot \mathbf{grad} +\gamma){\bf v} = -e({\bf E} + \mu_{0}{\bf v}\times {\bf H}) - \frac{\mathbf{grad} \;p}{n}},
\end{equation}
along with the continuity equation
\begin{equation}\label{eq:1.2}
{\displaystyle  \partial_{t}n+\mathbf{div}(n{\bf v})= 0 ,}
\end{equation}
where $-e$ is the electron charge, $m_{e}$ is the effective electron mass, and $\gamma>0$ is the damping constant. The term $-e({\bf E} + \mu_{0}{\bf v}\times {\bf H})$ represents the Lorentz force (assuming non-magnetic materials) in which the electric field $\mathbf{E}$ and the magnetic field $\mathbf{H}$ satisfy the macroscopic Maxwell's equations
\begin{equation}\label{eq:1.3}
\left\{
\begin{array}{lll}
{\displaystyle  \mathbf{curl} \;{\bf E} = -{\mu_{0}} \partial_{t}{\bf H},}\\[2mm]
 {\displaystyle \mathbf{curl} \;{\bf H} = \varepsilon_{0} \varepsilon_{\infty} \partial_{t}{\bf E} +{\bf J}. }
 \end{array}
 \right.
\end{equation}
Here $\mu_{0}$ and ${\varepsilon_{0}}$ are the magnetic permeability and electric permittivity of vacuum, respectively, and $\varepsilon_{\infty}$ is the relative electric permittivity of the metal that takes into account the polarization of bound electrons. The hydrodynamic equations (\ref{eq:1.1})-(\ref{eq:1.2}) are coupled with Maxwell's equations (\ref{eq:1.3}) via the polarization current density $\mathbf{J}$ of the free electrons:
\begin{equation}\label{eq:1.4}
\mathbf{J} = -en\mathbf{v}.
\end{equation}
The above equations (\ref{eq:1.1})-(\ref{eq:1.4}) constitute the hydrodynamic model containing nonlinear and nonlocal effects. In order to get a simplified system of equations more suitable for numerical computation, we linearize the hydrodynamic model in a perturbative manner. The physical fields are expanded into a non-oscillating term (e.g. the constant equilibrium electron density $n_{0}$) and a small first-order dynamic term as follows:
\begin{equation}\label{eq:1.5}
n({\bf x},t)\approx n_{0} + n_{1}({\bf x},t), \quad {\bf v}({\bf x},t) \approx {\bf v}_{0} + {\bf v}_{1}({\bf x},t).
\end{equation}
Similar expansions can be written for the electric and magnetic fields. Since in the absence of external fields ${\bf v} = {\bf v}_{0} = {\bf 0}$, the two nonlinear (higher-order) terms ${\bf v}\cdot {\bf grad}\,{\bf v}$ and ${\bf v}\times {\bf H}$ in (\ref{eq:1.1}) are dropped out due to linearization. By using the Thomas-Fermi theory of metals \cite{raza2015nonlocal}, the electron pressure $p(\mathbf{x},t)$ is given by
\begin{equation}\label{eq:1.5.0}
p(\mathbf{x},t)=(3\pi^{2})^{2/3} \frac{\hbar^{2}}{5m_{e}}n({\bf x},t)^{5/3}.
\end{equation}
In view of (\ref{eq:1.5}) and (\ref{eq:1.5.0}), the term $\displaystyle \frac{{\bf grad} \,p}{n}$ in (\ref{eq:1.1}) can be linearized as
\begin{equation}\label{eq:1.6}
\frac{\mathbf{grad} \; p}{n} \approx m_{e}\beta^{2}\frac{\mathbf{grad} \; n}{n_{0}},
\end{equation}
where $\displaystyle \beta^{2} = \frac{(3\pi)^{2/3}}{3}\frac{\hbar^{2}}{m_{e}^{2}}n_{0}^{2/3}$ is a parameter representing the nonlocality. Using the approximation (\ref{eq:1.6}) and neglecting the two nonlinear terms in (\ref{eq:1.1}) due to linearization, we arrive at the following linearized Euler equation
 \begin{equation}\label{eq:1.7}
\partial_{t}{\bf v}  = \frac{-e}{m_{e}}{\bf E}  - \gamma {\bf v} - \beta^{2}\frac{\mathbf{grad} \; n}{n_{0}}
\end{equation}
and the linearized continuity equation
\begin{equation}\label{eq:1.8}
\partial_{t}n+n_{0}\mathbf{div}\;{\bf v}= 0.
\end{equation}
Differentiating (\ref{eq:1.7}) with respect to time $t$, inserting the linearized current density ${\bf J} \approx -en_{0}{\bf v}$ and making use of (\ref{eq:1.8}), we come to
\begin{equation}\label{eq:1.9}
\partial_{tt}{\bf J} + \gamma\partial_{t}{\bf J} - \beta^{2}\mathbf{grad} \,(\mathbf{div}\,{\bf J}) - \omega_{p}^{2} \varepsilon_{0} \partial_{t} {\bf E} = 0,
\end{equation}
where $\omega_{p} = \sqrt{n_{0}e^{2}/(m_{e}\varepsilon_{0})}$ is the plasma frequency. Combining (\ref{eq:1.3}) and (\ref{eq:1.9}), we obtain the NHD model for metals
 \begin{equation}\label{eq:1.10}
\left\{
\begin{array}{@{}l@{}}
{\displaystyle  \mathbf{curl}\;{\bf E} = -\mu_{0}\partial_{t}{\bf H} ,}\\[2mm]
 {\displaystyle \mathbf{curl} \; {\bf H} = \varepsilon_{0}\varepsilon_{\infty}\partial_{t}{\bf E} +{\bf J} ,} \\[2mm]
 {\displaystyle \partial_{tt}{\bf J} + \gamma\partial_{t}{\bf J} - \beta^{2}\mathbf{grad} \,(\mathbf{div}\,{\bf J}) - \omega_{p}^{2} \varepsilon_{0} \partial_{t} {\bf E} = 0.}
\end{array}
\right.
\end{equation}
Replacing $\partial_{t}$ with $-{\rm i}\omega$ in (\ref{eq:1.10}) by Fourier transformation in the time domain, where ${\rm i}$ is the imaginary unit and $\omega$ is the angular frequency, and eliminating the magnetic field ${\bf H}$, we obtain the NHD model in the frequency domain
 \begin{equation}\label{eq:1.11}
\left\{
\begin{array}{@{}l@{}}
{\displaystyle \mathbf{curl}\;(\mu_{0}^{-1}\mathbf{curl}\;{\bf E}) - \varepsilon_{0}\varepsilon_{\infty}\omega^{2}{\bf E} = {\rm i}\omega {\bf J}, }\\[2mm]
 {\displaystyle \omega(\omega+{\rm i}\gamma){\bf J} + \beta^{2} \mathbf{grad} \,(\mathbf{div}\,{\bf J}) = {\rm i}\omega \omega^{2}_{p}\varepsilon_{0}{\bf E}.}
\end{array}
\right.
\end{equation}

\begin{rem}
By Fourier transformation in the space domain, we replace the operator ${\bf grad}\, {\bf div}$ with $-{\bf k}^{2}$ in the second equation of (\ref{eq:1.11}) and get
\begin{equation}
{\bf J}= \frac{{\rm i}\omega \omega^{2}_{p}\varepsilon_{0}\,{\bf E}}{{\omega(\omega+{\rm i}\gamma)-\beta^{2}{\bf k}^{2}}},
\end{equation}
which leads to the spatially-dispersive (relative) permittivity for the metal
\begin{equation}
\epsilon(\omega,{\bf k}) = \epsilon_{\infty}- \frac{\omega_{p}^{2}}{\omega(\omega+{\rm i}\gamma)-\beta^{2}{\bf k}^{2}}.
\end{equation}
The parameter $\beta$ represents the level of nonlocality. As $\beta\rightarrow 0$, the NHD model reduces to the classical local-response Drude model.
\end{rem}

\subsection{Electromagnetic scattering by metallic nanostructure arrays}
In this paper, we consider the problem as shown in Figure~2.1 where the metallic nanostructure arrays embedded in a dielectric medium are illuminated by an incident plane wave. The interaction of light with metallic nanostructure arrays is described by the NHD model
 \begin{equation}\label{eq:2.1}
\left\{
\begin{array}{@{}l@{}}
    {\displaystyle  \mathbf{curl}\, (\mu_{\eta}^{-1}\,({\bf x}) \mathbf{curl}\, {\bf E}_{\eta}) - \varepsilon_{\eta} ({\bf x}) \omega^{2}\, {\bf E}_{\eta} -{\mathrm {{i}}}\omega {\bf J}_{\eta} ={\bf 0}, \,\;\quad {\rm in} \;\; \Omega_{s},}\\[2mm]
  {\displaystyle \omega(\omega+{\mathrm {{i}}}\gamma){\bf J}_{\eta} + \beta^{2} \mathbf{grad}\,(\mathbf{div}\, {\bf J}_{\eta}) -{\rm i}\omega \omega^{2}_{p}\varepsilon_{0}{\bf E}_{\eta}={\bf 0}, \;\quad {\rm in}\;\; \Omega_{s,\eta}.}
\end{array}
\right.
\end{equation}
Here $\Omega_{s}\subset \mathbb{R}^{3}$ is a bounded Lipschitz domain occupied by the scatter (the metallic nanostructures and the dielectric medium). $\Omega_{s,\eta} =\cup_{k=1}^{N} \Omega_{\eta}^{k}$ denotes the domain occupied by the metallic nanostructures, where $\Omega_{\eta}^{k}$ is the domain occupied by the $k$-th nanostructure and $N$ is the number of metallic nanostructures. $\eta$ is the relative size of the periodic microstructure and we denote by $Y=(0,l_{1})\times(0,l_{2})\times(0,l_{3})$ the rescaled reference cell, where $l_{1}$,  $l_{2}$, and $l_{3}$ are positive constants. $\mu_{\eta}(\mathbf{x})$ and $\varepsilon_{\eta}(\mathbf{x})$ are the magnetic permeability and electric permittivity, respectively. We assume that $\mu_{\eta}(\mathbf{x}) = \mu(\mathbf{x}/\eta)$, $\varepsilon_{\eta}(\mathbf{x}) = \varepsilon(\mathbf{x}/\eta)$, where $\mu({\bf y})$ and $\varepsilon({\bf y})$ satisfy

$ (\mathbf{A}_1). $ $\mu({\bf y})$ and $\varepsilon({\bf y})$ are $Y$-periodic in $\mathbb{R}^{3}$, i.e.,
\begin{equation*}
\mu({\bf y} + kl_{i} \mathbf{e}_{i}) = \mu({\bf y}), \quad \varepsilon({\bf y} + kl_{i} \mathbf{e}_{i}) = \varepsilon({\bf y}),\quad \forall k\in \mathbb{Z},\quad i=1,2,3, \quad {\bf y}\in\mathbb{R}^{3},
\end{equation*}
and $\{\mathbf{e}_{1},\, \mathbf{e}_{2},\,\mathbf{e}_{3}\}$ is the canonical basis of $\mathbb{R}^{3}$.

\vspace{2mm}
$ (\mathbf{A}_2). $ There exists positive constants $\mu_{1}$, $\mu_{2}$, $\varepsilon_{1}$, and $\varepsilon_{2}$ such that $\mu({\bf y})$ and $\varepsilon({\bf y})$ satisfy
 \begin{equation*}
0<\mu_1\leq \mu({\bf y}) \leq \mu_2< \infty,\quad
0<\varepsilon_1\leq \varepsilon({\bf y})\leq \varepsilon_2< \infty ,\quad \forall {\bf y}\in Y.
\end{equation*}
\begin{rem}
In the simple case as illustrated in Fig~2.1, where there is only one metallic nanostructure inside the reference cell, $\mu({\bf y})$ and $\varepsilon({\bf y})$ are piecewise constant functions given by
\begin{gather*}
\mu({\bf y})
=\left\{
\begin{array}{lll}
 \mu_{m}, \,  \quad \mathbf{y} \in Y_{1}\\[3mm]
 \mu_{d}, \,  \,\,\quad \mathbf{y} \in Y_{2},
\end{array}
\right.
\quad \varepsilon({\bf y})
=\left\{
\begin{array}{l}
 \varepsilon_{m}, \,  \quad \mathbf{y}\in Y_{1}\\[3mm]
 \varepsilon_{d}, \,  \;\quad \mathbf{y} \in Y_{2},
\end{array}
\right.
\end{gather*}
where $Y_{1}\subset Y$ and $Y_{2}\subset Y$ are the domains occupied by the metallic nanostructure and the dielectric medium, respectively, $\mu_{m}\,$($\varepsilon_{m}$) and $\mu_{d}\,$($\varepsilon_{d}$) are the magnetic permeability (electric permittivity) of the metal and dielectric medium, respectively.
\end{rem}

\begin{figure}\label{fig:1-1}
\centering
\begin{tikzpicture}
\filldraw[fill=black!15!white,draw=black] (-1.0,-1.0)  rectangle (1.0,1.0);
\foreach \x in {-0.75,-0.25,0.25,0.75}
\foreach \y in {-0.75,-0.25,0.25,0.75}
{
  \filldraw[fill=red!40, draw=red!20!black] (\x,\y) circle (0.125cm);
}
\node at (0,3.0) {Plane wave};
\node at (0,2.5) {excitation};
\draw[thick][->](2,3)--(0.9, 1.35);
\draw[thick](1.0,14.2/6)--(1.8, 11/6);
\draw[thick](1.2,16/6)--(2.0, 12.8/6);
\draw[thick](1.1,15.1/6)--(1.9, 11.9/6);

\draw[dashed] (0,0) circle (2.1cm);

\filldraw[fill=black!15!white,draw=black] (3.5,-1.5)  rectangle (6.5,1.5);
\filldraw[fill=red!40, draw=red!20!black] (5,0) circle (0.75cm);
\end{tikzpicture}
\caption{Scattering setting: an incident electromagnetic wave is incident on metallic nanostructure arrays embedded in a dielectric medium. We assume that the dielectric medium is surrounded by free space. The rescaled reference cell is plotted on the right.}
\end{figure}
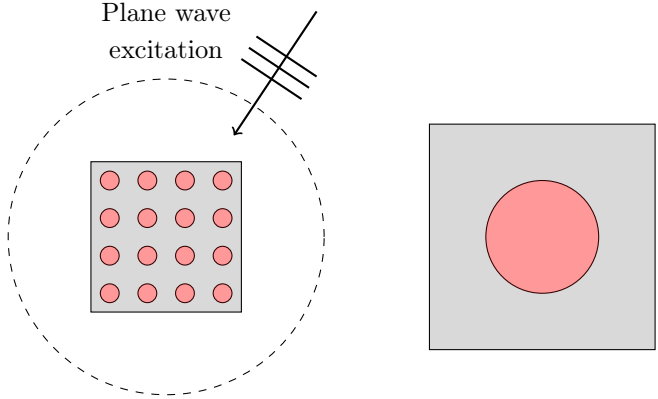

Note that the polarization current $\mathbf{J}_{\eta}$ only exists in the metallic nanostructures. For the well-posedness of the problem, we need to impose some appropriate boundary conditions on $\partial \Omega_{s,\eta}$ for $\mathbf{J}_{\eta}$. In this paper, we use the hard-wall boundary conditions
\begin{equation}\label{eq2.2}
    {\displaystyle {\bf n}\cdot {\bf J}_{\eta}= 0,\quad  {\rm on}\;\; \partial \Omega_{s,\eta},}
\end{equation}
which implies that the electrons are confined within the metal and spill-out of electrons outside the metal is neglected.

We assume that the scatter is surrounded by free space in which the (total) electric field satisfies the following Maxwell's equations
 \begin{equation}\label{eq2.3}
  {\displaystyle \mathbf{curl}\,(\mu_0^{-1}\,\mathbf{curl}\,\mathbf{E})-\varepsilon_0\omega^2\,\mathbf{E}=\mathbf{0}}
\end{equation}
and the scattered field $\mathbf{E}_{s} = \mathbf{E} - \mathbf{E}^{inc}$ satisfies the Silver-M\"{u}ller radiation condition
\begin{equation}\label{eq:2.4}
    {\displaystyle (\mu_0^{-1}\,\mathbf{curl}\, {\bf E}_s)\times \hat{\bf x} - {\mathrm{{i}}} \omega{\bf E}_s  =o(1/|{\mathbf{x}}|),\quad {\rm as}\quad |\mathbf{x}|\rightarrow \infty}.
\end{equation}
Here $\mathbf{E}^{inc}$ denotes the electric field of the incident wave. In addition, we need to impose the transmission conditions across the boundary $\partial\Omega_{s}$ for the electric field
\begin{equation}\label{eq:2.5}
\mathbf{n}\times \mathbf{E}_{\eta}=\mathbf{n}\times \mathbf{E},\quad
\mathbf{n}\times (\mu_{\eta}^{-1}\mathbf{curl}\,\mathbf{E}_{\eta})=\mathbf{n}\times (\mu_0^{-1}\mathbf{curl}\,\mathbf{E}).
\end{equation}

The scattering problem (\ref{eq:2.1})-(\ref{eq:2.5}) is defined on the whole space. To make this unbounded electromagnetic problem suitable for numerical computations, it is necessary to truncate the infinite domain. Different techniques can be used, such as absorbing boundary conditions, perfectly matched layer (PML), and boundary integral equation techniques. In this paper, we apply the following absorbing boundary condition on the boundary of a much larger domain $\Omega$ that contains the domain $\Omega_{s}$ of the scatter
 \begin{equation}\label{eq:2.6}
    {\displaystyle (\mu_0^{-1}\,\mathbf{curl}\, {\bf E}) \times {\bf n} - {\mathrm {{i}}}\omega({\bf n}\times{\bf E})\times{\bf n} ={\bf g},\quad {\rm on}\;\; \partial \Omega,}
\end{equation}
where ${\bf g}({\bf x})=(\mu_0^{-1}\,\mathbf{curl}\, {\bf E}^{inc})\times {\bf n} - {\mathrm {{i}}}\omega({\bf n}\times{\bf E}^{inc})\times{\bf n}$. (\ref{eq:2.6}) is the first-order approximation of the Silver-M\"uller radiation condition \cite{stupfel1995theoretical}. It is important to note that the multiscale method developed in this paper is independent of the choice of domain truncation techniques.

Combining (\ref{eq:2.1})-(\ref{eq:2.6}), Maxwell's equations coupled with the NHD model which describe electromagnetic scattering by metallic nanostructure arrays embedded in a dielectric host are given by
 \begin{equation}\label{eq:2.7}
\left\{
\begin{array}{@{}l@{}}
    {\displaystyle  \mathbf{curl}\, (\mu_{\eta}^{-1}\,({\bf x}) \mathbf{curl}\, {\bf E}_{\eta}) - \varepsilon_{\eta} ({\bf x}) \omega^{2}\, {\bf E}_{\eta} -{\mathrm {{i}}}\omega {\bf J}_{\eta} ={\bf 0}, \quad \;{\rm in} \;\; \Omega_{s},}\\[2mm]
  {\displaystyle \omega(\omega+{\mathrm {{i}}}\gamma){\bf J}_{\eta} + \beta^{2} \mathbf{grad}\,(\mathbf{div}\, {\bf J}_{\eta}) -{\rm i}\omega \omega^{2}_{p}\varepsilon_{0}{\bf E}_{\eta}={\bf 0},\;\quad {\rm in} \;\Omega_{s,\eta},}\\[2mm]
 {\displaystyle \mathbf{curl}\,(\mu_0^{-1}\,\mathbf{curl}\,\mathbf{E})-\varepsilon_0\omega^2\,\mathbf{E}=\mathbf{0}, \qquad \qquad \qquad\quad \; \;\quad {\rm in}\;\; \Omega/\Omega_{s}, } \\[2mm]
 {\displaystyle \mathbf{n}\times (\mu_{\eta}^{-1}\mathbf{curl}\,\mathbf{E}_{\eta})=\mathbf{n}\times (\mu_0^{-1}\mathbf{curl}\,\mathbf{E}),\qquad {\rm on}\; \partial \Omega_{s}, }\\[2mm]
 {\displaystyle \mathbf{n}\times \mathbf{E}_{\eta}=\mathbf{n}\times \mathbf{E},\qquad \qquad \quad \qquad \qquad \;\qquad {\rm on}\; \partial \Omega_{s}, }\\[2mm]
 {\displaystyle  {\bf n}\cdot {\bf J}_{\eta}= 0,\qquad \qquad \;\;\, \qquad \qquad \qquad \;\qquad\quad  {\rm on}\;\; \partial \Omega_{s,\eta}, }\\[2mm]
 {\displaystyle   (\mu_0^{-1}\,\mathbf{curl}\, {\bf E}) \times {\bf n} - {\mathrm {{i}}}\omega({\bf n}\times{\bf E})\times{\bf n} ={\bf g},\;\;\quad {\rm on}\;\; \partial \Omega.}
\end{array}
\right.
\end{equation}

Due to the rapidly oscillating coefficients of Maxwell's equations, it is usually very difficult to solve the coupled system (\ref{eq:2.7}) directly by a standard numerical method. In general, this difficulty can be overcome by using the homogenization and multiscale asymptotic methods. However, an interesting feature of this coupled system is that the equation satisfied by the polarization current is only defined in the metallic nanostructures while couples with Maxwell's equations defined on the whole domain, which prevents the application of the usual multiscale asymptotic method. To tackle this problem, we extend the equation satisfied by the polarization current outside the metallic nanostructures in a novel way.

First, we introduce the $Y$-periodic functions $\gamma_{\lambda}(\mathbf{y})$ and $\beta^{2}_{\lambda}(\mathbf{y})$ which are given by
\begin{gather}\label{eq:2.7.0}
\gamma_{\lambda}(\mathbf{y})
=\left\{
\begin{array}{lll}
 \gamma, \,\,  \quad \mathbf{y} \in Y_{1}\\[3mm]
 \lambda, \,  \,\,\quad \mathbf{y} \in Y_{2},
\end{array}
\right.
\quad \beta^{2}_{\lambda}(\mathbf{y})
=\left\{
\begin{array}{l}
 \beta^{2}, \,  \quad \mathbf{y}\in Y_{1}\\[3mm]
 \lambda, \, \,\, \;\quad \mathbf{y} \in Y_{2},
\end{array}
\right.
\end{gather}
in the reference cell, where $Y_{1}\subset Y$ and $Y_{2}\subset Y$ are the domains occupied by metallic nanostructures and the dielectric medium, respectively, and $\lambda >0 $ is the extension parameter which is assumed to be sufficiently large.

Next we set $\gamma_{\eta,\lambda}(\mathbf{x}) = \gamma_{\lambda}({\bf x}/\eta)$ and $\beta^{2}_{\eta,\lambda}(\mathbf{x}) = \beta^{2}_{\lambda}(\mathbf{x}/\eta)$, which satisfy
\begin{gather}\label{eq:2.8}
\gamma_{\eta,\lambda}(\mathbf{x})
=\left\{
\begin{array}{l}
 \gamma, \,  \quad \mathbf{x}\in\Omega_{s,\eta}\\[3mm]
 \lambda, \,  \quad \mathbf{x}\in \Omega_{s}/\Omega_{s,\eta},
\end{array}
\right.
\quad \beta^{2}_{\eta,\lambda}(\mathbf{x})
=\left\{
\begin{array}{l}
 \beta^{2}, \,  \quad \mathbf{x}\in\Omega_{s,\eta}\\[3mm]
 \lambda, \,  \;\,\quad\mathbf{x}\in\Omega_{s}/\Omega_{s,\eta},
\end{array}
\right.
\end{gather}
respectively.

Finally, we extend the equation satisfied by the polarization current into the whole of $\Omega_{s}$ by replacing the constants $\gamma$ and $\beta^{2}$ in the equation by $\gamma_{\eta,\lambda}(\mathbf{x})$ and $\beta^{2}_{\eta,\lambda}(\mathbf{x})$, respectively. We denote by $({\mathbf{J}}_{\eta,\lambda}, {\mathbf{E}}_{\eta,\lambda}, {\mathbf{E}}_{\lambda})$ the extended solution, which satisfies the extended system
 \begin{equation}\label{eq:2.9}
\left\{
\begin{array}{@{}l@{}}
    {\displaystyle  \mathbf{curl}\, (\mu_{\eta}^{-1} \mathbf{curl}\, {\bf E}_{\eta,\lambda}) - \varepsilon_{\eta} \omega^{2}\, {\bf E}_{\eta,\lambda} -{\mathrm {{i}}}\omega {\bf J}_{\eta,\lambda} ={\bf 0}, \qquad \quad \qquad \quad \;{\rm in} \;\; \Omega_{s},}\\[2mm]
  {\displaystyle \omega(\omega+{\mathrm {{i}}}\gamma_{\eta,\lambda}){\bf J}_{\eta,\lambda} + \mathbf{grad}\,(\beta^{2}_{\eta,\lambda}\mathbf{div}\, {\bf J}_{\eta,\lambda}) -{\rm i}\omega \omega^{2}_{p}\varepsilon_{0}{\bf E}_{\eta,\lambda}={\bf 0},\;\quad {\rm in} \;\;\Omega_{s},}\\[2mm]
 {\displaystyle \mathbf{curl}\,(\mu_0^{-1}\,\mathbf{curl}\,{\mathbf{E}}_{\lambda})-\varepsilon_0\omega^2\,{\mathbf{E}}_{\lambda}=\mathbf{0}, \qquad \qquad \quad \qquad \qquad\quad \; \;\quad {\rm in}\;\; \Omega/\Omega_{s}, } \\[2mm]
 {\displaystyle \mathbf{n}\times (\mu_{\eta}^{-1}\mathbf{curl}\,{\mathbf{E}}_{\eta,\lambda})=\mathbf{n}\times (\mu_0^{-1}\mathbf{curl}\,{\mathbf{E}}_{\lambda}),\,\qquad {\rm on}\; \partial \Omega_{s}, }\\[2mm]
 {\displaystyle \mathbf{n}\times {\mathbf{E}}_{\eta,\lambda}=\mathbf{n}\times {\mathbf{E}}_{\lambda},\qquad \qquad \quad  \qquad \qquad \;\qquad {\rm on}\; \partial \Omega_{s}, }\\[2mm]
 {\displaystyle  {\bf n}\cdot {\bf J}_{\eta,\lambda}= 0,\qquad \quad \qquad \;\, \qquad \qquad \qquad \;\qquad\quad  {\rm on}\;\,\partial \Omega_{s}, }\\[2mm]
 {\displaystyle   (\mu_0^{-1}\,\mathbf{curl}\, {\bf E}_{\lambda}) \times {\bf n} - {\mathrm {{i}}}\omega({\bf n}\times{\bf E}_{\lambda})\times{\bf n} ={\bf g},\;\;\;\quad {\rm on}\;\; \partial \Omega.}
\end{array}
\right.
\end{equation}

We have the following error estimate of the extension:
\begin{theorem}\label{thm:2.1}
Let $({\mathbf{J}}_{\eta}, {\mathbf{E}}_{\eta}, {\mathbf{E}})$ and $({\mathbf{J}}_{\eta,\lambda}, {\mathbf{E}}_{\eta,\lambda}, {\mathbf{E}}_{\lambda})$ be the solutions of the original system (\ref{eq:2.7}) and the extended system (\ref{eq:2.9}), respectively. Assume that $\mathbf{g}\in \mathbf{L}^{2}(\partial \Omega)$ and $\Omega$, $\Omega_s$, and $\Omega_{s,\eta}$ are bounded, simply-connected, Lipschitz domains in $\mathbb{R}^{3}$ with $\bar{\Omega}_{s}\subset \Omega$ and $\bar{\Omega}_{s,\eta}\subset \Omega_{s}$.  It holds for any $\lambda>0$,
\begin{equation}
\begin{array}{lll}
{\displaystyle \Vert {\mathbf{J}}_{\eta} - {\mathbf{J}}_{\eta,\lambda}\Vert_{\mathbf{H}(\mathbf{div};\Omega_{s,\eta})} +\Vert {\mathbf{E}}_{\eta} - {\mathbf{E}}_{\eta,\lambda}\Vert_{\mathbf{H}(\mathbf{curl};\Omega_{s})} }\\[2mm]
{\displaystyle \qquad \quad +\, \Vert {\mathbf{E}} - {\mathbf{E}}_{\lambda}\Vert_{\mathbf{H}(\mathbf{curl};\Omega/\Omega_{s})} \leq {C}/{\lambda^{\frac12}}, }\\[2mm]
\end{array}
\end{equation}
where $C$ is a positive constant independent of $\lambda$ but might depend on $\eta$.
\end{theorem}
\begin{rem}
Theorem~\ref{thm:2.1} indicates that the error between the original system and the extended system decreases with increasing extension parameter $\lambda$.
\end{rem}

The proof of Theorem~\ref{thm:2.1} is given in the Appendix.

\section{Homogenization and Multiscale asymptotic method}\label{sec-3}
In this section, we derive the homogenized system and define the multiscale approximate solutions for the extended system (\ref{eq:2.9}) by using the technique of multiscale asymptotic expansion. For brevity, we omit the dependence of the electric field and the polarization current on the parameter $\lambda$ and write $({\bf E}_{\eta,\lambda},{\bf J}_{\eta,\lambda})$ as $({\bf E}_{\eta},{\bf J}_{\eta})$ in what follows. We consider the homogenization of the following equations
 \begin{equation}\label{eq:2-6}
\left\{
\begin{array}{@{}l@{}}
    {\displaystyle  \mathbf{curl}\, (\mu_{\eta}^{-1}\mathbf{curl}\, {\bf E}_{\eta}) - \varepsilon_{\eta} \omega^{2} {\bf E}_{\eta} = {\mathrm {{i}}}\omega {\bf J}_{\eta},}\\[2mm]
    {\displaystyle \gamma^{\ast}_{\eta}{\bf J}_{\eta} + \mathbf{grad}\,(\beta^{\ast}_{\eta}\mathbf{div}\, {\bf J}_{\eta}) = {\rm i}\omega {\bf E}_{\eta},}
\end{array}
\right.
\end{equation}
where
\begin{equation}\label{eq:2-6-0}
\mu_{\eta} = \mu({\bf x} / \eta),\quad \varepsilon_{\eta} = \varepsilon({\bf x} / \eta), \quad \gamma^{\ast}_{\eta} = \frac{\omega(\omega+{\rm i}\gamma({\bf x} / \eta))}{\omega_{p}^{2}\varepsilon_{0}},\quad \beta^{\ast}_{\eta} = \frac{\beta^{2}({\bf x} / \eta)}{\omega_{p}^{2}\varepsilon_{0}},
\end{equation}
and $\mu({\bf y})$, $\varepsilon({\bf y})$, $\gamma({\bf y})$, $\beta({\bf y})$ are $Y$-periodic functions. We want to study the behaviour of $({\bf{E}}_{\eta},\bf J_{\eta})$ as $\eta\to 0$.

For the sake of simplicity, we introduce the operators
\begin{equation*}
\begin{array}{@{}c@{}}
    \mathcal{A}^{\eta}= {\mathbf{curl}}\, \mu_{\eta}^{-1}\,{\mathbf{curl}} - \varepsilon_{\eta} \omega^{2},\quad
    \mathcal{B}^{\eta}= {\mathbf{grad}}\, \beta^{\ast}_{\eta}\,{\mathbf{div}} +\gamma^{\ast}_{\eta}.
\end{array}
\end{equation*}
We look for a formal asymptotic expansion of the form
 \begin{equation}\label{eq:2-8-2}
\begin{array}{@{}l@{}}
{\bf{E}}_{\eta}={\bf{E}}_0({\bf x},{\bf x}/\eta)+
    \eta{\bf{E}}_1({\bf x},{\bf x}/\eta)+\eta^2{\bf{E}}_2({\bf x},{\bf x}/\eta)+...,\\[2mm]
{\bf{J}}_{\eta}={\bf{J}}_0({\bf x},{\bf x}/\eta)+
\eta{\bf{J}}_1({\bf x},{\bf x}/\eta)+\eta^2{\bf{J}}_2({\bf x},{\bf x}/\eta)+...,
\end{array}
\end{equation}
where for $k=0,1,\cdots$, ${\bf{E}}_k({\bf x},{\bf y})$ and ${\bf{J}}_k({\bf x},{\bf y})$ are functions of both variables ${\bf x}$ and ${\bf y}$ and $Y$-periodic with respect to ${\bf y}$.

Let ${\bm \Psi} = {\bm \Psi}({\bf x}, {\bf y})\in \mathbb{R}^{3}$ be a function depending on two variables of $\mathbb{R}^{3}$ and denote by ${\bm \Psi}_{\eta} = {\bm \Psi}({\bf x}, {\bf x}/\eta)$. By using the chain rule, we have
\begin{equation*}
\begin{array}{lll}
{\displaystyle \mathbf{curl}\,{\bm \Psi}_{\eta}({\bf x}) = \big[\mathbf{curl}_{\bf x}\,{\bm \Psi}+\frac{1}{\eta}\mathbf{curl}_{\bf y}\,{\bm \Psi}\big]({\bf x}, \frac{\bf x}{\eta}) }\\[2mm]
{\displaystyle \mathbf{div}\,{\bm \Psi}_{\eta}({\bf x}) = \big[\mathbf{div}_{\bf x}\,{\bm \Psi}+\frac{1}{\eta}\mathbf{div}_{\bf y}\,{\bm \Psi}\big]({\bf x}, \frac{\bf x}{\eta})}.
\end{array}
\end{equation*}
With this in mind, we can write $\mathcal{A}^{\eta}{\bm \Psi}_{\eta}$ and $\mathcal{B}^{\eta}{\bm \Psi}_{\eta}$ as follows:
\begin{equation}\label{eq:2-8-0-0}
\begin{array}{lll}
{\displaystyle \mathcal{A}^{\eta}{\bm \Psi}_{\eta}({\bf x})= \big[\big(\eta^{-2}\mathcal{A}_1+\eta^{-1}\mathcal{A}_2+\mathcal{A}_3\big){\bm \Psi}\big]({\bf x}, \frac{\bf x}{\eta}),}\\[3mm]
{\displaystyle \mathcal{B}^{\eta}{\bm \Psi}_{\eta}({\bf x})= \big[\big(\eta^{-2}\mathcal{B}_1+\eta^{-1}\mathcal{B}_2+\mathcal{B}_3\big){\bm \Psi}\big]({\bf x}, \frac{\bf x}{\eta}),}
\end{array}
\end{equation}
where
 \begin{equation}\label{eq:2-8-0}
 \left\{
\begin{array}{@{}l@{}}
    {\displaystyle \mathcal{A}_1=\mathbf{curl}_{\bf y}\,\mu^{-1}({\bf y})\,\mathbf{curl}_{\bf y}}\\[2mm]
    {\displaystyle \mathcal{A}_2=\mathbf{curl}_{\bf x}\,\mu^{-1}({\bf y})\,\mathbf{curl}_{\bf y}+\mathbf{curl}_{\bf y}\,\mu^{-1}({\bf y})\,\mathbf{curl}_{\bf x}}\\[2mm]
    {\displaystyle \mathcal{A}_3=\mathbf{curl}_{\bf x}\,\mu^{-1}({\bf y})\,\mathbf{curl}_{\bf x} - \varepsilon({\bf y})\omega^{2},}
\end{array}
\right.
\end{equation}
and
 \begin{equation}\label{eq:2-8-1}
\left\{
\begin{array}{@{}l@{}}
    {\displaystyle \mathcal{B}_1=\mathbf{grad}_{\bf y}\,\beta^{\ast}({\bf y})\,\mathbf{div}_{\bf y}}\\[2mm]
    {\displaystyle \mathcal{B}_2=\mathbf{grad}_{\bf x}\,\beta^{\ast}({\bf y})\,\mathbf{div}_{\bf y}+\mathbf{grad}_{\bf y}\,\beta^{\ast}({\bf y})\,\mathbf{div}_{\bf x}}\\[2mm]
    {\displaystyle \mathcal{B}_3=\mathbf{grad}_{\bf x}\,\beta^{\ast}({\bf y})\,\mathbf{div}_{\bf x} +\gamma^{\ast}({\bf y}).}
\end{array}
\right.
\end{equation}
Next substituting the expansions (\ref{eq:2-8-2}) into the equations (\ref{eq:2-6}), using (\ref{eq:2-8-0-0})-(\ref{eq:2-8-1}) and equating the power-like terms of $\eta$, we get the following equations
 \begin{equation}\label{eq:2-9-0}
\begin{array}{@{}l@{}}
    {\displaystyle \eta^{-2}:\qquad \mathcal{A}_1{\bf{E}}_0=\bf{0}}\\[2mm]
    {\displaystyle \eta^{-1}:\qquad\mathcal{A}_1{\bf{E}}_1+\mathcal{A}_2{\bf{E}}_0=
    \bf{0}}\\[2mm]
    {\displaystyle \;\,\eta^{0}:\qquad\mathcal{A}_1{\bf{E}}_2+
    \mathcal{A}_2{\bf{E}}_1+\mathcal{A}_3{\bf{E}}_0={\rm i} \omega{\bf{J}}_0,
    }
\end{array}
\end{equation}
and
 \begin{equation}\label{eq:2-9-1}
\begin{array}{@{}l@{}}
    {\displaystyle \eta^{-2}:\qquad\mathcal{B}_1{\bf{J}}_0=\bf{0}}\\[2mm]
    {\displaystyle \eta^{-1}:\qquad\mathcal{B}_1{\bf{J}}_1+\mathcal{B}_2{\bf{J}}_0=
    \bf{0}}\\[2mm]
    {\displaystyle \;\,\eta^{0}:\qquad\mathcal{B}_1{\bf{J}}_2+
    \mathcal{B}_2{\bf{J}}_1+\mathcal{B}_3{\bf{J}}_0={\rm i}\omega {\bf{E}}_0.
    }
\end{array}
\end{equation}
Since ${\bf x}$ and ${\bf y}$ are considered as independent variables, equations (\ref{eq:2-9-0})-(\ref{eq:2-9-1}) are PDEs in ${\bf y}$ while ${\bf x}$ plays the role of a parameter.

Note that
\begin{equation*}
\begin{array}{lll}
{\displaystyle (\mathcal{A}_1{\bf E}_0,{\bf{E}}_0)_{Y}=(\mu^{-1}({\bf y})\mathbf{curl}_{\bf y}\,{\bf{E}}_0,\mathbf{curl}_{\bf y}\,{\bf{E}}_0)_{Y},}\\[2mm]
{\displaystyle (\mathcal{B}_1{\bf J}_0,{\bf{J}}_0)_{Y}=(\beta^{\ast}({\bf y})\mathbf{div}_{\bf y}\,{\bf{J}}_0,\mathbf{div}_{\bf y}\,{\bf{J}}_0)_{Y}. }
\end{array}
\end{equation*}
Consequently, the first equations of (\ref{eq:2-9-0}) and (\ref{eq:2-9-1}) are equivalent to
\begin{equation}\label{eq:2-10-1}
    \begin{array}{@{}l@{}}
        {\displaystyle \mathbf{curl}_{\bf y}\,{\bf{E}}_0={\bf 0},\quad \mathbf{div}_{\bf y}\,{\bf J}_0}=\bf{0},
    \end{array}
\end{equation}
respectively. Using (\ref{eq:2-10-1}), we can reduce the second equations of (\ref{eq:2-9-0}) and (\ref{eq:2-9-1}) to
\begin{equation}\label{eq:2-10-2}
    \begin{array}{@{}l@{}}
        {\displaystyle \mathbf{curl}_{\bf y}\,(\mu^{-1}({\bf y})\,\mathbf{curl}_{\bf y}\,{\bf{E}}_1+\mu^{-1}({\bf y}) \mathbf{curl}_{\bf x}\,{\bf{E}}_0)={\bf 0},}\\[2mm]
        {\displaystyle \mathbf{grad}_{\bf y}\,(\beta^{\ast}({\bf y})\,\mathbf{div}_{\bf y}\,{\bf{J}}_1+\beta^{\ast}({\bf y})\mathbf{div}_{\bf{x}}\,{\bf{J}}_0)=\bf{0}.}
    \end{array}
\end{equation}
Applying $\mathbf{div}_{\bf y}$ and $\mathbf{curl}_{\bf y}$ to the third equations of (\ref{eq:2-9-0}) and (\ref{eq:2-9-1}), respectively, we have
\begin{equation}\label{eq:2-10-3}
    \begin{array}{@{}l@{}}
        {\displaystyle \mathbf{div}_{\bf y}\,\mathbf{curl}_{\bf x}(\mu^{-1}({\bf y})\,\mathbf{curl}_{\bf y}{\bf{E}}_1+\mu^{-1}({\bf y})\,\mathbf{curl}_{\bf x}{\bf{E}}_0) -\omega^2\mathbf{div}_{\bf y}(\varepsilon({\bf y}){\bf{E}}_0)}\\[2mm]
        {\displaystyle  \qquad ={\rm i}\omega\,\mathbf{div}_{\bf y}{\bf{J}}_0={\bf 0},}\\[2mm]

        {\displaystyle \mathbf{curl}_{\bf y}\,\mathbf{grad}_{\bf x}(\beta^{\ast}({\bf y})\,\mathbf{div}_{\bf y}{\bf{J}}_1+\beta^{\ast}({\bf y})\,\mathbf{div}_{\bf x}{\bf{J}}_0)+\mathbf{curl}_{\bf y}\big(\gamma^{\ast}({\bf y})\,{\bf{J}}_0 \big)}\\[2mm]
        {\displaystyle \qquad ={\rm i}\omega\,\mathbf{curl}_{\bf y}{\bf{E}}_0={\bf 0},}
    \end{array}
\end{equation}
where we have used (\ref{eq:2-10-1}). Since $\mathbf{div}_{\bf y}\,\mathbf{curl}_{\bf x}=-\mathbf{div}_{\bf{x}}\,\mathbf{curl}_{\bf y}$ and
$\mathbf{curl}_{\bf y}\,\mathbf{grad}_{\bf x}=-\mathbf{curl}_{\bf x}\,\mathbf{grad}_{\bf y}$, combining (\ref{eq:2-10-2}) and (\ref{eq:2-10-3}), we obtain
\begin{equation}\label{eq:2-10-4}
    \begin{array}{@{}l@{}}
       \mathbf{div}_{\bf y}(\varepsilon({\bf y}){\bf{E}}_0)=\mathbf{curl}_{\bf y}\big(\gamma^{\ast}({\bf y})\,{\bf{J}}_0\big)={\bf 0}.
    \end{array}
\end{equation}
Denote by
\begin{equation}\label{eq:2-10-5}
    \begin{array}{@{}l@{}}
  {\displaystyle       {\bf {w}}_1=\mu^{-1}({\bf y})\,\mathbf{curl}_{\bf y}\,{\bf{E}}_1+\mu^{-1}({\bf y}) \mathbf{curl}_{\bf x}\,{\bf{E}}_0,}\\[2mm]
  {\displaystyle {w}_2=\beta^{\ast}({\bf y})\,\mathbf{div}_{\bf y}\,{\bf{J}}_1+\beta^{\ast}({\bf y})\mathbf{div}_{\bf{x}}\,{\bf{J}}_0 },
    \end{array}
\end{equation}
and
\begin{equation}\label{eq:2-10-6}
    \begin{array}{@{}l@{}}
        \mathcal{M}_{Y}({\bf {w}}_1)=\tilde{{\bf {w}}}_1,\,\quad
        \mathcal{M}_{Y}({{w}}_2)=\tilde{{{w}}}_2,\\[2mm]
        \mathcal{M}_{Y}({\bf {E}}_0)=\tilde{{\bf {E}}}_0,\,\quad
        \mathcal{M}_{Y}({\bf {J}}_0)=\tilde{{\bf {J}}}_0,
    \end{array}
\end{equation}
where $\mathcal{M}_{Y}(f) = \frac{1}{|Y|} \int_{Y} f({\bf y})\,d{\bf y}$ denotes the mean value of $f$ over the reference cell $Y$.

It follows from (\ref{eq:2-10-1})-(\ref{eq:2-10-2}) and the definition of ${\bf {w}}_1,{{w}}_2$ that
\begin{equation}\label{eq:2-10-7}
    \begin{array}{@{}l@{}}
        \mathbf{curl}_{\bf y} \,{\bf {w}}_1={\bf 0},\, \quad \mathbf{grad}_{\bf y}\, {w}_2 ={\bf 0},\\[2mm]
        \mathbf{div}_{\bf y} \,(\mu({\bf y}) {\bf {w}}_1)= \mathbf{div}_{\bf y}\,\mathbf{curl}_{\bf x}\,{\bf{E}}_0= -\mathbf{div}_{\bf{x}}\,\mathbf{curl}_{\bf y}\,{\bf{E}}_0={\bf 0}.
    \end{array}
\end{equation}
Therefore, we see that $\mathbf{curl}_{\bf y}\,( {\bf {w}}_1-\tilde{{\bf w}}_1)={\bf 0}$ and $\mathcal{M}_{Y}({\bf {w}}_1-\tilde{\bf w}_1)={\bm 0}$, which imply that there exists a $Y$-periodic function ${\psi}({\bf x},{\bf y})$ such that
\begin{equation}\label{eq:2-10-8}
    \begin{array}{@{}l@{}}
        {\bf {w}}_1-\tilde{{\bf w}}_1=\mathbf{grad}_{\bf y}\, {\psi}.
    \end{array}
\end{equation}
Similarly, in view of (\ref{eq:2-10-1}), there exists $Y$-periodic functions ${\phi}({\bf x},{\bf y})$ and ${\bm \varphi}({\bf x},{\bf y})$ such that
\begin{equation}\label{eq:2-10-8-1}
    \begin{array}{@{}l@{}}
        {\bf {E}}_0-\tilde{{\bf E}}_0=\mathbf{grad}_{\bf y}\, {\phi},\quad
        {\bf {J}}_0-\tilde{{\bf J}}_0=\mathbf{curl}_{\bf y}\, {\bm \varphi}.
    \end{array}
\end{equation}
In addition, since ${\bf grad}_{\bf y}\,w_{2} = {\bf 0}$, we see that $w_{2} = \tilde{w}_{2}$. Substituting (\ref{eq:2-10-8}) into the last equation of (\ref{eq:2-10-7}) gives
\begin{equation}\label{eq:2-10-9}
    \begin{array}{@{}l@{}}
       {\displaystyle  \mathbf{div}_{\bf y}\,\big(\mu({\bf y})\,\mathbf{grad}_{\bf y}\,{\psi}\big)=- \mathbf{div}_{\bf y}\,(\mu({\bf y}) \tilde{{\bf w}}_1) = -\sum_{i=1}^{3}(\tilde{{\bf w}}_1)_{i}\, \mathbf{div}_{\bf y}\,(\mu({\bf y}) {\bf e}_{i})},
    \end{array}
\end{equation}
where $(\tilde{{\bf w}}_1)_{i}$ denotes the $i$-th component of $\tilde{{\bf w}}_1$ and $\{ {\bf e}_1, {\bf e}_2, {\bf e}_3\} = \{(1,0,0)^T,(0,1,0)^T, (0,0,1)^T\}$ is the canonical basis of $\mathbb{R}^{3}$.
Similarly, we deduce from (\ref{eq:2-10-8-1}) and (\ref{eq:2-10-4}) that
\begin{equation}\label{eq:2-10-9-1}
    \begin{array}{@{}l@{}}
   {\displaystyle      \mathbf{div}_{\bf y}\,\big(\varepsilon({\bf y})\,\mathbf{grad}_{\bf y}\,{ \phi}\big)=- \mathbf{div}_{\bf y}\,(\varepsilon({\bf y})\tilde{{\bf E}}_0) = -\sum_{i=1}^{3}(\tilde{{\bf E}}_0)_{i} \, \mathbf{div}_{\bf y}\,(\varepsilon({\bf y}) {\bf e}_{i}),}\\[4mm]
      {\displaystyle  \mathbf{curl}_{\bf y} \big(\gamma^{\ast}({\bf y})\mathbf{curl}_{\bf y}\,{\bm \varphi}\big) =- \mathbf{curl}_{\bf y}\,\big(\gamma^{\ast}({\bf y}) \tilde{{\bf J}}_0\big) = - \sum_{i=1}^{3} (\tilde{{\bf J}}_0)_{i} \, \mathbf{curl}_{\bf y}\,\big(\gamma^{\ast}({\bf y})\mathbf{e}_{i}\big).}
    \end{array}
\end{equation}
To proceed further, we introduce the scalar-valued cell functions $\theta_i^{\mu}({\bf y})$, $\theta_i^{\varepsilon}({\bf y})$, and the vector-valued cell functions $\Theta_i^{\gamma}({\bf y})$, $i=1,2,3$, which are the solutions of the following problems:
\begin{equation}\label{eq:2-10-10}
\left\{
    \begin{array}{@{}l@{}}
      {\displaystyle   \mathbf{div}_{\bf y}\,(\mu({\bf y})\,\mathbf{grad}_{\bf y}\,\theta_i^{\mu})=- \mathbf{div}_{\bf y}\,(\mu({\bf y}){\bf e}_i),\quad {\rm in} \; \;Y}\\[2mm]
      {\displaystyle \theta_i^{\mu}\quad Y\,{\rm -}\,{\rm periodic}, }
    \end{array}
    \right.
\end{equation}

\begin{equation}\label{eq:2-10-10-1}
\left\{
    \begin{array}{@{}l@{}}
      {\displaystyle   \mathbf{div}_{\bf y}\,(\varepsilon({\bf y})\,\mathbf{grad}_{\bf y}\,\theta_i^{\varepsilon})=- \mathbf{div}_{\bf y}\,(\varepsilon({\bf y}){\bf e}_i),\quad {\rm in} \; \;Y}\\[2mm]
      {\displaystyle \theta_i^{\varepsilon}\quad Y\,{\rm -}\,{\rm periodic}, }
    \end{array}
    \right.
\end{equation}
and
\begin{equation}\label{eq:2-10-10-2}
\left\{
    \begin{array}{@{}l@{}}
      {\displaystyle   \mathbf{curl}_{\bf y}\,(\gamma^{\ast}({\bf y})\,\mathbf{curl}_{\bf y}\,\Theta_i^{\gamma})=- \mathbf{curl}_{\bf y}\,(\gamma^{\ast}({\bf y}){\bf e}_i),\quad {\rm in} \; \;Y}\\[2mm]
      {\displaystyle \mathbf{div}_{\bf y}\,\Theta_i^{\gamma} = 0 ,\quad {\rm in} \; \;Y}\\[2mm]
      {\displaystyle \Theta_i^{\gamma}\quad Y\,{\rm -}\,{\rm periodic}, }
    \end{array}
    \right.
\end{equation}
respectively.

Next we define the vectors $\bm{\theta}^{\mu}=(\theta_1^{\mu},\theta_2^{\mu},\theta_3^{\mu})$, $\bm{\theta}^{\eta}=(\theta_1^{\eta},\theta_2^{\eta},\theta_3^{\eta})$, and the matrix $\bm{\Theta}^{\gamma}=\{\Theta_1^{\gamma},\Theta_2^{\gamma},\Theta_3^{\gamma}\}$. Applying the superposition principle to the equations (\ref{eq:2-10-9}) and (\ref{eq:2-10-9-1}), we can write $\psi$, $\phi$, and ${\bm \varphi}$ as
\begin{equation}\label{eq:2-10-11-0}
    \begin{array}{@{}l@{}}
     {\displaystyle    \psi = \sum_{i=1}^{3}(\tilde{{\bf w}}_1)_{i} \,\theta_i^{\mu} =\bm{\theta}^{\mu}\cdot \tilde{{\bf w}}_1,\quad \phi = \sum_{i=1}^{3}(\tilde{{\bf E}}_0)_{i} \,\theta_i^{\mu} =\bm{\theta}^{\varepsilon}\cdot \tilde{{\bf E}}_0},\\[2mm]
       {\displaystyle\qquad \qquad \quad  {\bm \varphi} = \sum_{i=1}^{3}(\tilde{{\bf J}}_0)_{i} \,\Theta_i^{\gamma} =\bm{\Theta}^{\gamma}\cdot \tilde{{\bf J}}_0,}
    \end{array}
\end{equation}
respectively. Combining (\ref{eq:2-10-8}), (\ref{eq:2-10-8-1}), and (\ref{eq:2-10-11-0}), we obtain
\begin{equation}\label{eq:2-10-11}
    \begin{array}{@{}l@{}}
        {\bf {w}}_1=(I+\mathbf{grad}_{\bf y}\,{\bm \theta^{\mu}})\tilde{{\bf w}}_1,\quad
        {\bf {E}}_0=(I+\mathbf{grad}_{\bf y}\,{\bm \theta}^{\varepsilon})\tilde{{\bf E}}_0,\\[3mm]
        {\bf \qquad \qquad \qquad {J}}_0=(I+\mathbf{curl}_{\bf y}\,{\bm \Theta^{\gamma}})\tilde{{\bf J}}_0.
    \end{array}
\end{equation}
Integrating the two equations in (\ref{eq:2-10-5}) over the reference cell $Y$ after multiplying them by $\mu({\bf y})$ and $1/\beta^{\ast}({\bf y})$, respectively, it follows that
\begin{equation}\label{eq:2-10-12}
    \begin{array}{@{}l@{}}
        \mathbf{curl}_{\bf x}\,\tilde{{\bf E}}_0 = \mathcal{M}_{Y}(\mu({\bf y}){\bf {w}}_1) = \mathcal{M}_{Y}\big(\mu({\bf y})(I+\mathbf{grad}_{\bf y}\,{\bm \theta}^{\mu}) \big)\tilde{{\bf w}}_1, \\[2mm]
        \mathbf{div}_{\bf{x}}\,\tilde{{\bf J}}_0 =\mathcal{M}_{Y}(\frac{1}{\beta^{\ast}({\bf y})} w_{2}) = \mathcal{M}_{Y}(\frac{1}{\beta^{\ast}({\bf y})})\tilde{w}_{2},
    \end{array}
\end{equation}
where we have used (\ref{eq:2-10-11}) and the fact that $w_{2} = \tilde{w}_{2}$. Similarly, applying $\mathcal{M}_{Y}$ to the third equations of (\ref{eq:2-9-0}) and (\ref{eq:2-9-1}), respectively, recalling the definition of ${\bf w}_{1}$ and $w_{2}$, and using (\ref{eq:2-10-11}), we arrive at
\begin{equation}\label{eq:2-10-13}
    \begin{array}{@{}l@{}}
        \mathbf{curl}_{\bf x} \,\tilde{\bf {w}}_1 -\omega^2 \mathcal{M}_{Y}\big(\varepsilon({\bf y})(I+\mathbf{grad}\,{\bm \theta^{\eta}})\big)\tilde{{\bf E}}_0={\rm i}\omega \tilde{{\bf J}}_0,\\[2mm]
        \mathbf{grad}_{\bf{x}} \,\tilde{{w}}_2+\mathcal{M}_{Y}\big(\gamma^{\ast}({\bf y})(I+\mathbf{curl}_{\bf y}\,{\bm \Theta^{\gamma}})\big)\tilde{{\bf J}}_0={\rm i}\omega\tilde{{\bf E}}_0.
    \end{array}
\end{equation}
Substituting (\ref{eq:2-10-12}) into (\ref{eq:2-10-13}), we obtain the homogenized system associated to (\ref{eq:2-6})
\begin{equation}\label{eq:2-10-14}
\left\{
    \begin{array}{@{}l@{}}
        \mathbf{curl}_{\bf x}\,\big(\widehat{\mu}^{-1}\,\mathbf{curl}_{\bf x}\,\tilde{{\bf E}}_0\big)  -\omega^2 \widehat{\varepsilon}\,\tilde{{\bf E}}_0={\rm i}\omega \tilde{{\bf J}}_0,\\[2mm]
        \mathbf{grad}_{\bf{x}}\, \big(\widehat{\beta^{\ast}}\, \mathbf{div}_{\bf{x}}\,\tilde{{\bf J}}_0\big)+\widehat{\gamma^{\ast}}\,\tilde{{\bf J}}_0={\rm i}\omega\tilde{{\bf E}}_0,
    \end{array}
    \right.
\end{equation}
where the homogenized coefficients are given by
\begin{equation}\label{eq:2-10-14-0}
    \begin{array}{@{}l@{}}
        \widehat{\mu}=\mathcal{M}_{Y}(\mu({\bf y})(I+\mathbf{grad}_{\bf y}\,{\bm \theta}^{\mu})),\quad \widehat{\varepsilon}=\mathcal{M}_{Y}(\varepsilon({\bf y})(I+\mathbf{grad}_{\bf y}\,{\bm \theta}^{\varepsilon})),\\[3mm]
         \widehat{\gamma^{\ast}}=\mathcal{M}_{Y}(\gamma^{\ast}({\bf y})(I+\mathbf{curl}_{\bf y}\,{\bm \Theta^{\gamma}})),\quad \widehat{\beta^{\ast}}= \mathcal{M}_{Y}(\frac{1}{\beta^{\ast}({\bf y})})^{-1},
    \end{array}
\end{equation}
respectively, and the cell functions $\bm{\theta}^{\mu}=(\theta_1^{\mu},\theta_2^{\mu},\theta_3^{\mu})$, $\bm{\theta}^{\eta}=(\theta_1^{\eta},\theta_2^{\eta},\theta_3^{\eta})$, and $\bm{\Theta}^{\gamma}=\{\Theta_1^{\gamma},\Theta_2^{\gamma},\Theta_3^{\gamma}\}$ are defined in (\ref{eq:2-10-10})-(\ref{eq:2-10-10-2}), respectively. Note that although $\mu({\bf y})$, $\varepsilon({\bf y})$, and $\gamma^{\ast}({\bf y})$ are scalar-valued functions, in general, the homogenized coefficients $\widehat{\mu}$, $\widehat{\varepsilon}$, and $\widehat{\gamma^{\ast}}$ are non-diagonal matrices.

In addition, we define the multiscale approximate solution for $({\bf{E}}_{\eta},\bf J_{\eta})$ as follows:
\begin{equation}\label{eq:2-10-15}
\left\{
    \begin{array}{@{}l@{}}
        \tilde{\bf {E}}_{0,\eta}=(I+\mathbf{grad}_{\bf y}\,{\bm \theta^{\varepsilon}})\,\tilde{{\bf E}}_0,\\[2mm]
        \tilde{\bf {J}}_{0,\eta}=(I+\mathbf{curl}_{\bf y}\,{\bm \Theta^{\gamma}})\,\tilde{{\bf J}}_0.
    \end{array}
   \right.
\end{equation}
It is worth pointing out that the homogenized system and the multiscale approximate solution are only defined in the domain of the scatter. The homogenized system (\ref{eq:2-10-14}) is coupled to Maxwell's equations defined in the exterior domain with constant coefficients and forms the homogenized coupled system
 \begin{equation}\label{eq:2-10-14-1}
\left\{
\begin{array}{@{}l@{}}
    {\displaystyle \mathbf{curl}\,\big(\widehat{\mu}^{-1}\,\mathbf{curl}\,\tilde{{\bf E}}_0\big)  -\omega^2 \widehat{\varepsilon}\,\tilde{{\bf E}}_0-{\rm i}\omega \tilde{{\bf J}}_0 = {\bf 0}, \;\quad \;{\rm in} \;\; \Omega_{s},}\\[2mm]
  {\displaystyle \mathbf{grad} \,\big( \widehat{\beta^{\ast}}\, \mathbf{div}\,\tilde{{\bf J}}_0\big)+\widehat{\gamma^{\ast}}\,\tilde{{\bf J}}_0-{\rm i}\omega\tilde{{\bf E}}_0 = {\bf 0},\qquad \quad {\rm in} \;\;\Omega_{s},}\\[2mm]
 {\displaystyle \mathbf{curl}\,(\mu_0^{-1}\,\mathbf{curl}\,\mathbf{E})-\varepsilon_0\omega^2\,\mathbf{E}=\mathbf{0},  \qquad \qquad \; \;\quad {\rm in}\;\; \Omega/\Omega_{s}, } \\[2mm]
 {\displaystyle \mathbf{n}\times (\widehat{\mu}^{-1}\mathbf{curl}\,\tilde{\mathbf{E}}_{0})=\mathbf{n}\times (\mu_0^{-1}\mathbf{curl}\,\mathbf{E}),\,\qquad {\rm on}\;\; \partial \Omega_{s}, }\\[2mm]
 {\displaystyle \mathbf{n}\times \tilde{\mathbf{E}}_{0}=\mathbf{n}\times \mathbf{E},\qquad \qquad \quad \qquad \qquad \;\qquad {\rm on}\; \;\partial \Omega_{s}, }\\[2mm]
 {\displaystyle  {\bf n}\cdot \tilde{\bf J}_{0}= 0,\qquad \qquad \;\;\, \qquad \qquad \qquad \;\qquad\quad  {\rm on}\;\; \partial \Omega_{s}, }\\[2mm]
 {\displaystyle   (\mu_0^{-1}\,\mathbf{curl}\, {\bf E}) \times {\bf n} - {\mathrm {{i}}}\omega({\bf n}\times{\bf E})\times{\bf n} ={\bf g},\;\;\quad {\rm on}\;\; \partial \Omega.}
\end{array}
\right.
\end{equation}
\begin{rem}
The coefficients of the homogenized system are constant in the domain of the scatter, making the numerical solution of the homogenized problem much easier than the original problem. Physically, the homogenization procedure is equivalent to replacing the heterogeneous scatter comprising of metallic nanostructures and the dielectric medium by an averaging homogeneous scatter as illustrated in Fig~3.1.
\end{rem}

For simplicity, we set
\begin{gather}\label{eq:2.16}
\mu_{\rm eff}
=\left\{
\begin{array}{l}
 \widehat{\mu}, \,  \;\;\quad \rm{in}\;\; \Omega_{s}\\[3mm]
 \mu_{0}, \,  \,\quad \rm{in}\;\; \Omega/\Omega_{s},
\end{array}
\right.
\quad \varepsilon_{\rm eff}
=\left\{
\begin{array}{l}
\widehat{\varepsilon}, \,  \;\,\;\quad \rm{in}\;\; \Omega_{s}\\[3mm]
 \varepsilon_{0}, \,  \;\,\quad \rm{in}\;\; \Omega/\Omega_{s},
\end{array}
\right.
\end{gather}
and denote by $\tilde{{\mathbf{E}}}_{0}$ the electric field on the whole domain $\Omega$. Then, the homogenized coupled system (\ref{eq:2-10-14-1}) can be rewritten as
 \begin{equation}\label{eq:2.18}
\left\{
\begin{array}{@{}l@{}}
    {\displaystyle \mathbf{curl}\,\big({\mu}_{\rm eff}^{-1}\,\mathbf{curl}\,\tilde{{\bf E}}_{0}\big)  -\omega^2{\varepsilon}_{\rm eff} \,\tilde{{\bf E}}_{0}-{\rm i}\omega \tilde{{\bf J}}_0 = {\bf 0}, \;\;\quad \;{\rm in} \;\; \Omega,}\\[3mm]
  {\displaystyle \mathbf{grad} \,\big( \widehat{\beta^{\ast}}\, \mathbf{div}\,\tilde{{\bf J}}_0\big)+\widehat{\gamma^{\ast}}\,\tilde{{\bf J}}_0-{\rm i}\omega \tilde{{\bf E}}_{0} = {\bf 0},\,\quad\qquad \quad {\rm in} \;\;\Omega_{s},}\\[2mm]
 {\displaystyle  {\bf n}\cdot \tilde{\bf J}_{0}= 0,\qquad \qquad \quad\;\, \qquad \qquad \qquad \;\qquad\quad  {\rm on}\;\; \partial \Omega_{s}, }\\[2mm]
 {\displaystyle   (\mu_{\rm eff}^{-1}\,\mathbf{curl}\, \tilde{\bf E}_{0}) \times {\bf n} - {\mathrm {{i}}}\omega({\bf n}\times \tilde{\bf E}_{0})\times{\bf n} ={\bf g},\;\;\quad {\rm on}\;\; \partial \Omega.}
\end{array}
\right.
\end{equation}

\begin{figure}\label{fig:3-1}
\centering
\begin{tikzpicture}
\filldraw[fill=black!15!white,draw=black] (-1.0,-1.0)  rectangle (1.0,1.0);
\foreach \x in {-0.75,-0.25,0.25,0.75}
\foreach \y in {-0.75,-0.25,0.25,0.75}
{
  \filldraw[fill=red!40, draw=red!20!black] (\x,\y) circle (0.125cm);
}
\node at (0,3.0) {Plane wave};
\node at (0,2.5) {excitation};
\draw[thick][->](2,3)--(0.9, 1.35);
\draw[thick](1.0,14.2/6)--(1.8, 11/6);
\draw[thick](1.2,16/6)--(2.0, 12.8/6);
\draw[thick](1.1,15.1/6)--(1.9, 11.9/6);

\draw[dashed] (0,0) circle (2.1cm);

\draw[thick][->](2.2,-1.0)--(4.3, -1.0);
\node[scale=0.7] at (3.3,-0.7) {Homogenization};

\filldraw[fill=red!10!white,draw=black] (5.5,-1.0)  rectangle (7.5,1.0);
\foreach \x in {-0.75,-0.25,0.25,0.75}
\foreach \y in {-0.75,-0.25,0.25,0.75}
{
  \filldraw[fill=red!40, draw=red!20!black] (\x,\y) circle (0.125cm);
}
\node at (6.5,3.0) {Plane wave};
\node at (6.5,2.5) {excitation};
\draw[thick][->](8.5,3)--(7.4, 1.35);
\draw[thick](7.5,14.2/6)--(8.3, 11/6);
\draw[thick](7.7,16/6)--(8.5, 12.8/6);
\draw[thick](7.6,15.1/6)--(8.4, 11.9/6);

\draw[dashed] (6.5,0) circle (2.1cm);

\end{tikzpicture}
\caption{Left: the original problem; right: the homogenized problem.}
\end{figure}
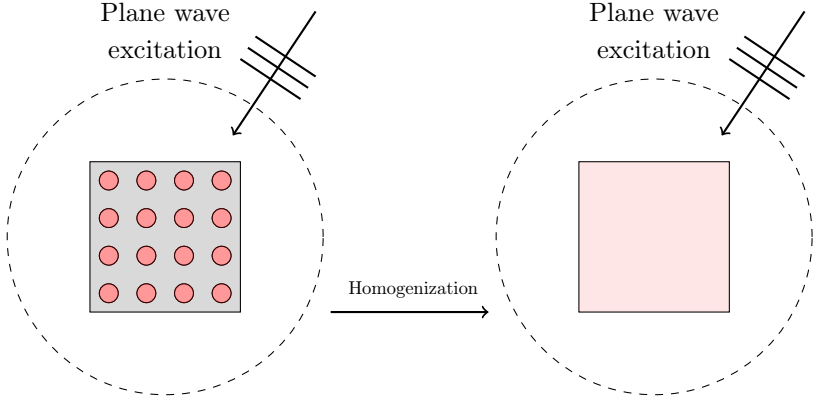

\section{Multiscale approaches and the associated numerical algorithms}\label{sec-4}
Based on the discussion in the previous sections, we have the multiscale approach for solving the problem (\ref{eq:2.7}) consisting of the following steps:

\textbf{Step 1}. Choose a sufficiently large parameter $\lambda>0$, set the coefficients $\gamma_{\eta,\lambda}$ and $\beta^{2}_{\eta,\lambda}$ according to (\ref{eq:2.8}), and extend the equation satisfied by the polarization current into the domain occupied by the dielectric medium.

\vspace{1mm}
\textbf{Step 2}. Solve the equations (\ref{eq:2-10-10})-(\ref{eq:2-10-10-2}) in the reference cell $Y$ to get the cell functions $\theta_i^{\mu}({\bf y})$, $\theta_i^{\varepsilon}({\bf y})$, $\Theta_i^{\gamma}({\bf y})$, $i=1,2,3$, and compute the homogenized coefficients $\widehat{\mu}$, $\widehat{\varepsilon}$, $\widehat{\gamma^{\ast}}$, and $\widehat{\beta^{\ast}}$ given by (\ref{eq:2-10-14-0}).

\vspace{1mm}
\textbf{Step 3}. Solve the homogenized coupled system (\ref{eq:2.18}) in the domain $\Omega$ to get the homogenized solution ($\tilde{\bf E}_{0}$, $\tilde{\bf J}_{0}$).

\vspace{1mm}
\textbf{Step 4}. Compute the multiscale approximate solution (\ref{eq:2-10-15}) by adding some correctors to the homogenized solution.

The proposed multiscale approach reduces the direct solution of the challenging problem (\ref{eq:2.7}) to the solution of several cell problems and a homogenized problem with constant coefficients, which is capable of saving much computational cost especially when there exists a huge number of metallic nanostructures (the parameter $\eta$ is very small).

Numerical results in section~\ref{sec-5} show that the electric field calculated by the multiscale approach is in good agreement with the reference solution (numerical solution of the original system on a very fine mesh) outside the metallic nanostructures. However, it is much less accurate inside the metallic nanostructures.
The main reason for the failure of the multiscale approach inside the metallic nanostructures is that due to the large extension parameter $\lambda$, the imaginary part of the homogenized coefficient $\widehat{\gamma^{\ast}}$ is also very large, making the polarization current rather small and thus the electric field loses the nonlocal information. In fact, we have
\begin{lemma}\label{lem4.1}
Assume that there exists a positive constant $\alpha$ such that the homogenized coefficients $\widehat{\gamma^{\ast}}$ satisfies
\begin{equation}\label{eq:4.0}
{\rm Im}(\widehat{\gamma^{\ast}}) {\bm \xi}\cdot {\bm \xi}\geq \alpha |{\bm \xi}|^{2},\quad \forall {\bm \xi}\in \mathbb{R}^{3}.
\end{equation}
Then the solution of the homogenized coupled system (\ref{eq:2.18}) satisfies
\begin{equation}\label{eq:4.1}
\Vert \tilde{\bf J}_{0} \Vert_{\mathbf{L}^{2}(\Omega_{s})} \leq \frac{C}{\alpha},
\end{equation}
where $C$ is a positive constant independent of $\alpha$.
\end{lemma}

The proof of this lemma is very similar to that of Lemma~\ref{lemma1} and thus we omit it here.

\begin{lemma}\label{lem4.2}
Let $(\tilde{\bf E}_{0}, \tilde{\bf J}_{0})$ be the solution of (\ref{eq:2.18}) and $\mathbf{E}$ be the solution of the Maxwell's equations
 \begin{equation}\label{eq:4.2}
\left\{
\begin{array}{@{}l@{}}
    {\displaystyle \mathbf{curl}\,\big({\mu}_{\rm eff}^{-1}\,\mathbf{curl}\,{{\bf E}}\big)  -\omega^2{\varepsilon}_{\rm eff} \, {{\bf E}}= {\bf 0}, \;\quad \;{\rm in} \;\; \Omega,}\\[3mm]
 {\displaystyle   (\mu_{\rm eff}^{-1}\,\mathbf{curl}\, {\bf E}) \times {\bf n} - {\mathrm {{i}}}\omega({\bf n}\times {\bf E})\times{\bf n} ={\bf g},\;\;\quad {\rm on}\;\; \partial \Omega.}
\end{array}
\right.
\end{equation}
Under the assumption of Lemma~\ref{lem4.1}, we have the following estimate
\begin{equation}\label{eq:4.3}
\Vert \tilde{\mathbf{E}}_{0} - \mathbf{E} \Vert_{\mathbf{H}(\mathbf{curl};\Omega)} \leq \frac{C}{\alpha}.
\end{equation}
\end{lemma}
\begin{proof}
Let $\mathbf{e} = \tilde{\mathbf{E}}_{0} - \mathbf{E}$. By subtracting the first equation of (\ref{eq:2.18}) from (\ref{eq:4.2}), we see that $\mathbf{e}$ satisfies
 \begin{equation}\label{eq:4.4}
\left\{
\begin{array}{@{}l@{}}
    {\displaystyle \mathbf{curl}\,\big({\mu}_{\rm eff}^{-1}\,\mathbf{curl}\,{{\bf e}}\big)  -\omega^2{\varepsilon}_{\rm eff} \, {{\bf e}}= {\rm i}\omega\tilde{\bf J}_{0}, \;\quad \;{\rm in} \;\; \Omega,}\\[3mm]
 {\displaystyle   (\mu_{\rm eff}^{-1}\,\mathbf{curl}\, {\bf e}) \times {\bf n} - {\mathrm {{i}}}\omega({\bf n}\times {\bf e})\times{\bf n} ={\bf 0},\;\;\quad {\rm on}\;\; \partial \Omega.}
\end{array}
\right.
\end{equation}
From Theorem~4.17 of \cite{Monk}, we get the estimate
\begin{equation}\label{eq:4.5}
\Vert \mathbf{e} \Vert_{\mathbf{H}(\mathbf{curl};\Omega)} \leq C\Vert \tilde{\bf J}_{0} \Vert_{\mathbf{L}^{2}(\Omega_{s})},
\end{equation}
which yields the desired estimate (\ref{eq:4.3}) by using (\ref{eq:4.1}).
\qquad \end{proof}

\begin{rem}
The above two lemmas show that as the imaginary part of the homogenized coefficient $\widehat{\gamma^{\ast}}$ becomes larger, the influence of the polarization current on the electric field becomes smaller. If $\alpha$ tends to infinity, the coupled system (\ref{eq:2.18}) will reduce to the pure Maxwell's equations and the electric field will lose the nonlocal information inside the metallic nanostructures. Although we are not able to give an estimate of the lower bound $\alpha$ in terms of the extension parameter $\lambda$, we have tested some practical examples in section~\ref{sec-5} and verified numerically that in these examples, $\alpha$ tends to infinity as $\lambda\rightarrow \infty$.
\end{rem}

Motivated by the numerical results of the previous multiscale approach and Lemma~\ref{lem4.1} and \ref{lem4.2}, we propose the following modified multiscale approach.

\vspace{2mm}
\noindent \textbf{Modified Multiscale Approach}:

\vspace{2mm}
\textbf{Step 1}. Solve the equations (\ref{eq:2-10-10})-(\ref{eq:2-10-10-1}) in the reference cell $Y$ to get the cell functions $\theta_i^{\mu}({\bf y})$, $\theta_i^{\varepsilon}({\bf y})$, and compute the homogenized coefficients $\widehat{\mu}$ and $\widehat{\varepsilon}$ given by (\ref{eq:2-10-14-0}).

\vspace{1mm}
\textbf{Step 2}. Solve the following homogenized Maxwell's equation on the whole domain
 \begin{equation}\label{eq:4.6}
\left\{
\begin{array}{@{}l@{}}
    {\displaystyle \mathbf{curl}\,\big({\mu}_{\rm eff}^{-1}\,\mathbf{curl}\,{{\bf E}}^{0}\big)  -\omega^2{\varepsilon}_{\rm eff} \, {{\bf E}}^{0}= {\bf 0}, \;\quad \;{\rm in} \;\; \Omega,}\\[3mm]
 {\displaystyle   (\mu_{\rm eff}^{-1}\,\mathbf{curl}\, {\bf E}^{0}) \times {\bf n} - {\mathrm {{i}}}\omega({\bf n}\times {\bf E}^{0})\times{\bf n} ={\bf g},\;\;\quad {\rm on}\;\; \partial \Omega.}
\end{array}
\right.
\end{equation}

\vspace{1mm}
\textbf{Step 3}. Compute the multiscale approximate solution ${\bf {E}}^{0}_{\eta}=(I+\mathbf{grad}_{\bf y}\,{\bm \theta^{\varepsilon}})\,{{\bf E}}^{0}$ and then find the modified multiscale approximate solutions ${\bf {E}}^{M}_{\eta}$ for the electric field such that ${\bf {E}}^{M}_{\eta} = {\bf {E}}^{0}_{\eta}$ outside the metallic nanostructures and ${\bf {E}}^{M}_{\eta}$ satisfies
 \begin{equation}\label{eq:4.7}
\left\{
\begin{array}{@{}l@{}}
    {\displaystyle  \mathbf{curl}\, (\mu_{k}^{-1}\mathbf{curl}\, {\bf E}^{M}_{\eta}) - \varepsilon_k \,\omega^{2}\, {\bf E}^{M}_{\eta}-{\mathrm {{i}}}\omega {\bf J} ^{M}_{\eta}={\bf 0}, \;\qquad \quad{\rm in} \;\; \Omega^{k}_{\eta},}\\[2mm]
  {\displaystyle \frac{\omega(\omega+{\mathrm {{i}}}\gamma)}{\omega^{2}_{p}\varepsilon_{0}}{\bf J} ^{M}_{\eta}+ \frac{\beta^{2}}{\omega^{2}_{p}\varepsilon_{0}} \mathbf{grad}\,(\mathbf{div}\, {\bf J}^{M}_{\eta}) -{\mathrm {{i}}}\omega {\bf E}^{M}_{\eta}={\bf 0}, \quad {\rm in}\;\; \Omega^{k}_{\eta},}\\[4mm]
    {\displaystyle{\bf n}\cdot {\bf J} ^{M}_{\eta}= 0,  \;\;\;\;\qquad \qquad {\rm on}\;\; \partial \Omega^{k}_{\eta},}\\[2mm]
    {\displaystyle {\bf n}\times{\bf E} ^{M}_{\eta}= {\bf n}\times{\bf E}_{\eta}^{0},\quad \quad {\rm on}\;\; \partial \Omega^{k}_{\eta},}
\end{array}
\right.
\end{equation}
for $k = 1,\cdots, N$, where $\Omega_{\eta}^{k}$ is the domain occupied by the $k$-th metallic nanostructure, $N$ is the number of the metallic nanostructures, $\mu_{k} = \mu_{\eta}|_{\Omega_{\eta}^{k}}$, and $\varepsilon_{k} = \varepsilon_{\eta}|_{\Omega_{\eta}^{k}}$.

There are two main differences between the modified multiscale approach and the previous one. First, in our modified approach, we only solve the homogenized Maxwell's equation (\ref{eq:4.6}) instead of the homogenized coupled system (\ref{eq:2.18}) on the whole domain. This is based on the observation that the homogenized coupled system (\ref{eq:2.18}) is close to the pure Maxwell's equations if we choose a sufficiently large extension parameter $\lambda$. By doing so, we can reduce much computational cost since we don't need to solve the cell equation (\ref{eq:2-10-10-2}), either. Second, in our modified approach, we modify the multiscale approximate solution ${\bf {E}}^{0}_{\eta}$ inside the metallic nanostructures by solving the original coupled system with boundary conditions given by ${\bf {E}}^{0}_{\eta}$ in each metallic nanostructure. This is based on the observation in numerical experiments that ${\bf {E}}^{0}_{\eta}$ computed by the previous multiscale approach agrees well with the reference solution outside the metallic nanostructures while substantially deviates from the reference solution inside the metallic nanostructures. Note that (\ref{eq:4.7}) is solved in each metallic nanostructure separately, where the coefficients $\mu_{k}$, $\varepsilon_{k}$, $\gamma$ and $\beta$ are constant. Therefore, the computational burden of solving (\ref{eq:4.7}) is much lower than solving the coupled system on the whole domain. In addition, in view of the periodic arrangement of metallic nanostructures within the dielectric medium, we have a fast algorithm for solving (\ref{eq:4.7}) which is described in detail below.

Denoting by $\mathbf{E}^{M,k}_{\eta}= \mathbf{E}^{M}_{\eta}|_{\Omega_{\eta}^{k}}$, $\mathbf{J}^{M,k}_{\eta}= \mathbf{J}^{M}_{\eta}|_{\Omega_{\eta}^{k}}$, and assuming that $\mathbf{E}^{M,k}_{\eta} = \widetilde{\mathbf{E}}^{M,k}_{\eta}+ {\bf {E}}^{0}_{\eta}|_{\Omega_{\eta}^{k}}$, then the problem (\ref{eq:4.7}) is equivalent to finding $\widetilde{\mathbf{E}}^{M,k}_{\eta}$ such that
 \begin{equation}\label{eq:4.8}
\left\{
\begin{array}{@{}l@{}}
    {\displaystyle  \mathbf{curl}\, (\mu_{k}^{-1}\mathbf{curl}\, \widetilde{\bf E}^{M,k}_{\eta}) - \varepsilon_k \,\omega^{2}\, \widetilde{\bf E}^{M,k}_{\eta}-{\mathrm {{i}}}\omega {\bf J} ^{M}_{\eta}=  - \mathbf{curl}\, (\mu_{k}^{-1}\mathbf{curl}\, {\bf E}^{0}_{\eta}) + \varepsilon_k \,\omega^{2}\, {\bf E}^{0}_{\eta}, \quad{\rm in} \;\; \Omega^{k}_{\eta},}\\[2mm]
  {\displaystyle \frac{\omega(\omega+{\mathrm {{i}}}\gamma)}{\omega^{2}_{p}\varepsilon_{0}}{\bf J} ^{M,k}_{\eta}+ \frac{\beta^{2}}{\omega^{2}_{p}\varepsilon_{0}} \mathbf{grad}\,(\mathbf{div}\, {\bf J}^{M,k}_{\eta}) -{\mathrm {{i}}}\omega \widetilde{\bf E}^{M,k}_{\eta}={\mathrm {{i}}}\omega {\bf E}^{0}_{\eta}, \quad {\rm in}\;\; \Omega^{k}_{\eta},}\\[4mm]
    {\displaystyle{\bf n}\cdot {\bf J} ^{M,k}_{\eta}= 0,  \;\;\;\qquad  {\rm on}\;\; \partial \Omega^{k}_{\eta},}\\[2mm]
    {\displaystyle {\bf n}\times \widetilde{\bf E} ^{M,k}_{\eta}= {\bf 0},  \quad \quad {\rm on}\;\; \partial \Omega^{k}_{\eta}.}
\end{array}
\right.
\end{equation}

Assume that ${\bf {E}}^{0}_{\eta}|_{\Omega_{\eta}^{k}} \in \mathbf{H}(\mathbf{curl};\Omega_{\eta}^{k})$ for $k=1,\cdots,N$. The variational formulation of (\ref{eq:4.8}) is: Find $\widetilde{\mathbf{E}}^{M,k}_{\eta}\in \mathbf{H}_{0}(\mathbf{curl};\Omega_{\eta}^{k})$, ${\mathbf{J}}^{M,k}_{\eta}\in \mathbf{H}_{0}(\mathbf{div};\Omega_{\eta}^{k})$ such that
 \begin{equation}\label{eq:4.9}
\left\{
\begin{array}{@{}l@{}}
    {\displaystyle  \mu_{k}^{-1}(\mathbf{curl}\, \widetilde{\bf E}^{M,k}_{\eta},\, \mathbf{curl}\,\mathbf{u}) - \varepsilon_k \,\omega^{2}(\widetilde{\bf E}^{M,k}_{\eta},\,\mathbf{u})-{\mathrm {{i}}}\omega ({\bf J} ^{M,k}_{\eta},\,\mathbf{u})}\\[2mm]
  {\displaystyle \quad = -\mu_{k}^{-1}(\mathbf{curl}\, {\bf E}^{0}_{\eta},\, \mathbf{curl}\,\mathbf{u}) + \varepsilon_k \,\omega^{2}({\bf E}^{0}_{\eta},\,\mathbf{u}), \qquad \forall \mathbf{u} \in \mathbf{H}_{0}(\mathbf{curl};\Omega_{\eta}^{k})}\\[2mm]
  {\displaystyle \frac{\omega(\omega+{\mathrm {{i}}}\gamma)}{\omega^{2}_{p}\varepsilon_{0}}({\bf J} ^{M,k}_{\eta},\, \mathbf{w})- \frac{\beta^{2}}{\omega^{2}_{p}\varepsilon_{0}} (\mathbf{div}\, {\bf J}^{M,k}_{\eta},\, \mathbf{div}\, \mathbf{w}) -{\mathrm {{i}}}\omega (\widetilde{\bf E}^{M,k}_{\eta},\,\mathbf{w})}\\[4mm]
    {\displaystyle \quad ={\mathrm {{i}}}\omega ({\bf E}^{0}_{\eta},\,\mathbf{w}), \qquad \forall \mathbf{w}\in \mathbf{H}_{0}(\mathbf{div};\Omega_{\eta}^{k}).}
\end{array}
\right.
\end{equation}

Let $\mathcal{T}^{k}_{h}$ be a uniform triangulation of $\Omega_{\eta}^{k}$ into tetrahedrons of maximal diameter $h$. We define the lowest order N\'{e}d\'{e}lec $\mathbf{H}_{0}(\mathbf{curl};\Omega_{\eta}^{k})$-conforming and Raviart--Thomas $ \mathbf{H}_{0}(\mathbf{div};\Omega_{\eta}^{k})$-conforming finite element spaces as follows \cite{Monk}:
\begin{equation*}
\begin{array}{lll}
    {\displaystyle  {X}^{k}_{h}=\{{\bf u}_{h} \in \mathbf{H}_{0}(\mathbf{curl};\Omega_{\eta}^{k}):\; {\bf u}_{h}|_{T} =\mathbf{a}_{T}+\mathbf{b}_{T}\times\mathbf{x}\;\;{\rm with}\; \mathbf{a}_{T},\mathbf{b}_{T}\in \mathbb{R}^3,\;\; \forall T\in \mathcal{T}^{k}_{h}\},}\\[2mm]
    {\displaystyle  {Y}^{k}_{h}=\{{\bf u}_{h} \in \mathbf{H}_0(\mathbf{div};\Omega_{\eta}^{k}):\; {\bf u}_{h}|_{T} ={a}_{T}\mathbf{x}+\mathbf{b}_{T}\;\;{\rm with}\; {a}_{T}\in \mathbb{R},\mathbf{b}_{T}\in \mathbb{R}^3,\;\; \forall T\in \mathcal{T}^{k}_{h}\}}.
\end{array}
\end{equation*}

The finite element approximation of (\ref{eq:4.9}) is formulated as follows: find $\widetilde{\mathbf{E}}^{M,k}_{\eta,h}\in {X}^{k}_{h}$, ${\mathbf{J}}^{M,k}_{\eta,h}\in {Y}^{k}_{h}$ such that
 \begin{equation}\label{eq:5.0}
\left\{
\begin{array}{@{}l@{}}
    {\displaystyle  \mu_{k}^{-1}(\mathbf{curl}\, \widetilde{\bf E}^{M,k}_{\eta,h},\, \mathbf{curl}\,\mathbf{u}_{h}) - \varepsilon_k \,\omega^{2}(\widetilde{\bf E}^{M,k}_{\eta,h},\,\mathbf{u}_{h})-{\mathrm {{i}}}\omega ({\bf J} ^{M,k}_{\eta,h},\,\mathbf{u}_{h})}\\[2mm]
  {\displaystyle \quad = -\mu_{k}^{-1}(\mathbf{curl}\, {\bf E}^{0}_{\eta},\, \mathbf{curl}\,\mathbf{u}_{h}) + \varepsilon_k \,\omega^{2}({\bf E}^{0}_{\eta},\,\mathbf{u}_{h}), \qquad \forall \mathbf{u}_{h} \in {X}^{k}_{h},}\\[2mm]
  {\displaystyle \frac{\omega(\omega+{\mathrm {{i}}}\gamma)}{\omega^{2}_{p}\varepsilon_{0}}({\bf J} ^{M,k}_{\eta,h},\, \mathbf{w}_{h})- \frac{\beta^{2}}{\omega^{2}_{p}\varepsilon_{0}} (\mathbf{div}\, {\bf J}^{M,k}_{\eta,h},\, \mathbf{div}\, \mathbf{w}_{h}) -{\mathrm {{i}}}\omega (\widetilde{\bf E}^{M,k}_{\eta,h},\,\mathbf{w}_{h})}\\[4mm]
    {\displaystyle \quad ={\mathrm {{i}}}\omega ({\bf E}^{0}_{\eta},\,\mathbf{w}_{h}), \qquad \forall \mathbf{w}_{h}\in {Y}^{k}_{h}.}
\end{array}
\right.
\end{equation}
The system (\ref{eq:5.0}) can be rewritten as the following algebraic equation
\begin{equation}\label{eq:5.1}
\mathbb{A}_{k}\mathbf{x}_{k} = \mathbf{b}_{k},
\end{equation}
where $\mathbb{A}_{k}$ is the matrix, $\mathbf{x}_{k}$ and $\mathbf{b}_{k}$ are the vectors representing the unknowns and the right-hand term, respectively.

If there is only one metallic nanostructure in the reference cell as shown in Fig~2.1, then the domains $\Omega_{\eta}^{k}$ $(k=1,\cdots,N)$ are essentially the same (translation-invariant). Therefore, we can generate the same mesh $\mathcal{T}_{h}$ for all the domains $\Omega_{\eta}^{k}$, which results in the same matrix $\mathbb{A}$ for the algebraic equation (\ref{eq:5.1}) due to the translation-invariant property of the $\mathbf{curl}$ and $\mathbf{div}$ operators and the fact that the coefficients $\mu_{k}$ and $\varepsilon_{k}$ take the same value in all the domains, respectively. In this case, we need to solve a set of linear equations
\begin{equation}\label{eq:5.2}
\mathbb{A}\mathbf{x}_{k} = \mathbf{b}_{k},\qquad k=1,\cdots,N
\end{equation}
with the same matrix and many different right-hand vectors. To solve this problem, we can perform an $LU$ decomposition of the matrix $\mathbb{A}$ once and then solve the triangular matrices for the different $\mathbf{b}_{k}$. This algorithm can be easily extended to the case that there exist several metallic nanostructures in the reference cell. In this way, we can make the computational burden of solving (\ref{eq:4.7}) controllable and insensitive to the number of the metallic nanostructures.

Finally, we briefly discuss the numerical methods for solving the scalar cell equations (\ref{eq:2-10-10})-(\ref{eq:2-10-10-1}) and the Maxwell's equations (\ref{eq:4.6}). Note that the cell equations are the same as for the scalar elliptic operators $-\mathbf{div}\,\mu_{\eta}\,\mathbf{grad}$ and $-\mathbf{div}\,\varepsilon_{\eta}\,\mathbf{grad}$ and they can be solved numerically by the adaptive finite element method as discussed in \cite{zhang2014multiscale}. As for the Maxwell's equations (\ref{eq:4.6}), a Galerkin finite element method based on the N\'{e}d\'{e}lec elements can be used and the resulting linear system can be solved by a multigrid preconditioned Minres algorithm. For more details, we refer the reader to \cite{chen2007adaptive,Monk}.

\begin{rem}
Although we don't use the technique of extension and the homogenized equation for the polarization current explicitly in the modified multiscale approach, they motivate  and justify the modified approach. In fact, as indicated in Lemma~\ref{lem4.1} and \ref{lem4.2}, it is because the imaginary part of the homogenized coefficient $\widehat{\gamma^{\ast}}$ is sufficiently large due to the large extension parameter $\lambda$ that we can replace the homogenized coupled system (\ref{eq:2.18}) by the homogenized Maxwell's equations (\ref{eq:4.6}) in our modified approach. In addition, the multiscale asymptotic method discussed in section~\ref{sec-3} can be used for the homogenization of composite materials in which the inclusions and the matrix are both metallic nanostructures with different physical parameters.
\end{rem}

\section{Numerical examples}\label{sec-5}
\begin{figure}[t!]
\begin{center}
\tiny{(a)}\includegraphics[width=4.2cm,height=4.2cm]{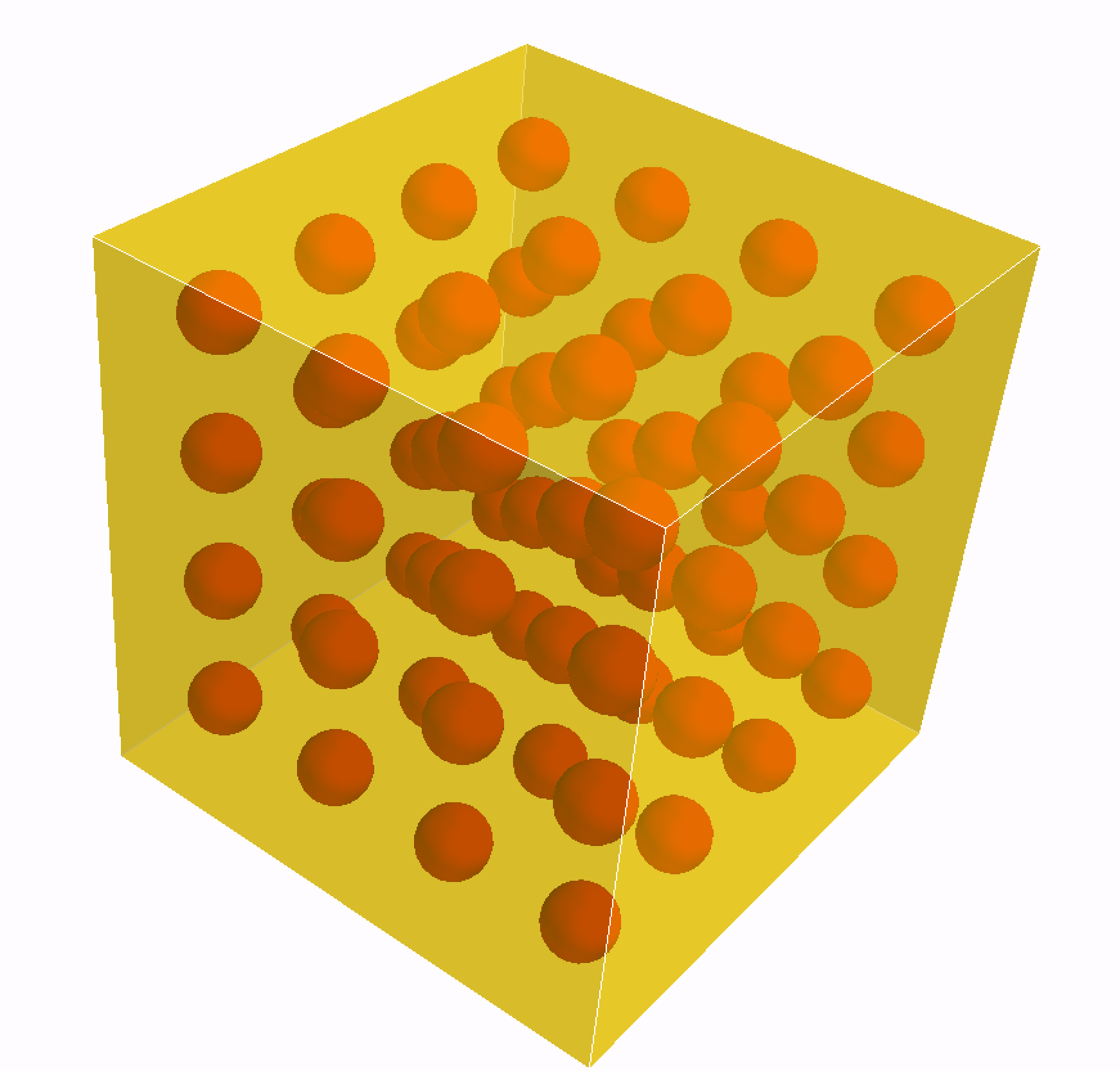}~
\tiny{(b)}\includegraphics[width=4.08cm,height=3.6cm]{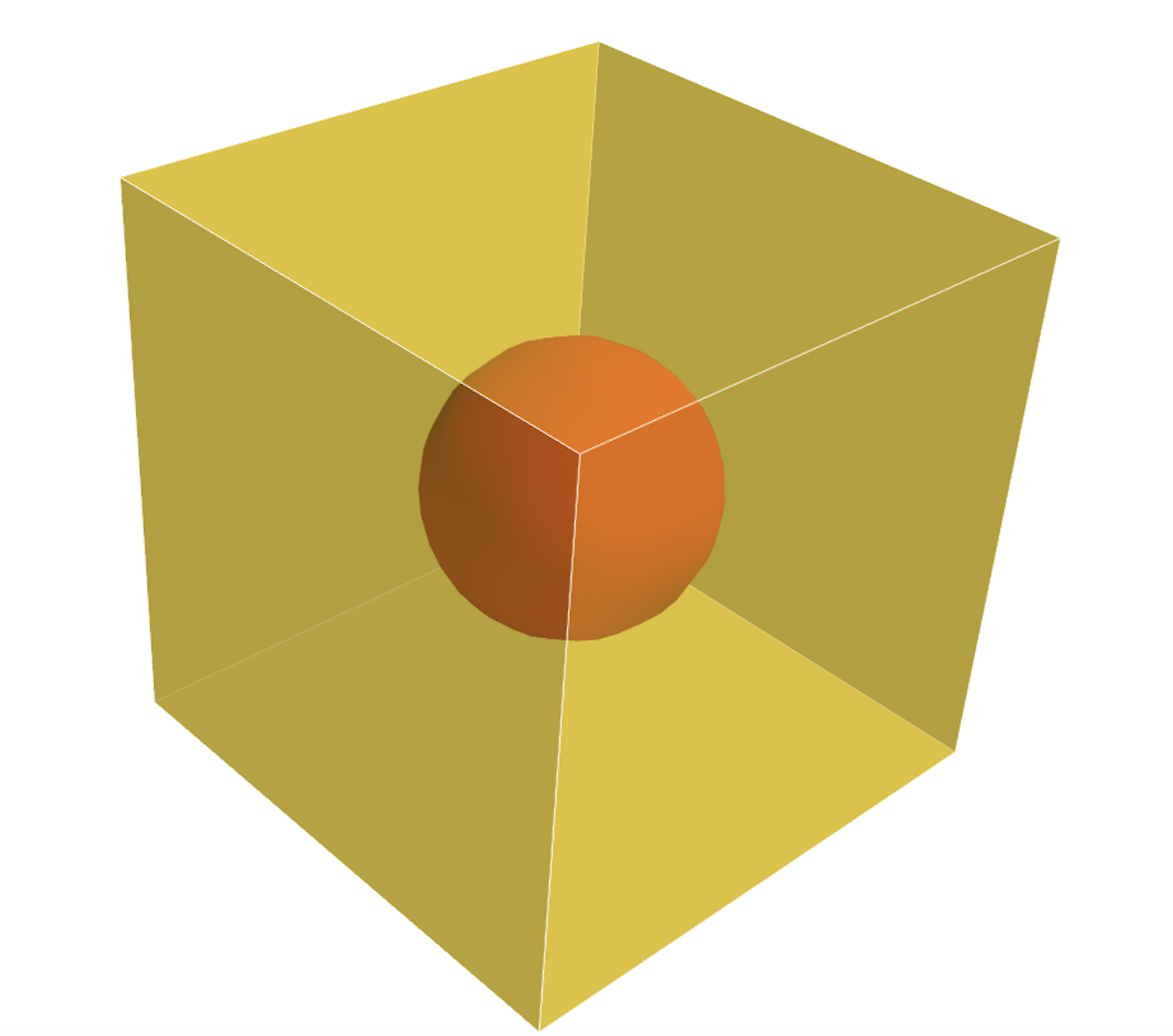}
\caption{{\rm(a)}~periodic gold nanoparticle arrays embedded in a dielectric medium.\/ {\rm(b)}~The rescaled reference cell $Q$.}\label{f4}
\end{center}
\end{figure}

In this section, we present several numerical experiments to validate the effectiveness of our method.
\begin{exam}\label{exam1}
In this example, we consider the scattering of a plane wave by periodic gold nanoparticle arrays embedded in a dielectric medium as shown in Fig.~5.1. There are 4 nanospheres with a radius of 2\,nm in each direction and the inter-particle distance is 1\,nm. The nanoparticle arrays are irradiated by a plane wave $\mathbf{E}^{inc} = {\rm exp}({\mathrm {{i}}}\omega y) \mathbf{e}_{x}$ propagating in the $y-$direction. The Silver--M\"{u}ller absorbing boundary condition is set on the boundary of a sphere of radius 40\,nm to truncate the computational domain. We use the same material parameters for the gold nanoparticle as in \cite{raza2015nonlocal}, which are summarized in Table~\ref{table0}. For the dielectric medium, we consider the following two cases:
\begin{case}\label{case5}
   Silicon dioxide ($SiO_{2}$): $\quad \mu=\mu_{0},\quad \varepsilon=3.9\,\varepsilon_{0}$;
   \vspace{1mm}
\begin{case}\label{case6}
 water ($H_{2}O$) \qquad \quad\;\;: $\quad \mu=\mu_{0},\quad \varepsilon=80\,\varepsilon_{0}$.
\end{case}
\end{case}
\end{exam}

\begin{table}[htbp]
\centering
\caption{Physical parameters of the NHD model.
$\epsilon_{0}$ and $\mu_{0}$ are the electric permittivity and magnetic permeability of free space, respectively.}\label{table0}
\begin{tabular}{ccccc}
  \hline$\omega_{p}$ & $\gamma$& $\beta$ & $\epsilon$ &  $\mu$ \\
 \hline
  $1.37\times 10^{16}$ rad/s & $1.08 \times 10^{14}$ rad/s & $1.08\times 10^{6}$ m/s &$9.5\,\epsilon_{0}$&$\mu_{0}$   \\
  \hline
\end{tabular}
\end{table}

First, we solve the original system (\ref{eq:2.7}) and the extended system (\ref{eq:2.9}) directly on a very fine mesh and test the error between their solutions. Set
  \begin{equation*}
        err=\|{\bf E}_{\eta}-{\bf E}_{\eta,\lambda}\|_{\mathbf{H}_{T}(\mathbf{curl},\Omega)}+ \|{\bf J}_{\eta,\lambda}-{\bf J}_{\eta}\|_{\mathbf{H}(\mathbf{div},\Omega_{s})}
    \end{equation*}
where $({\bf E}_{\eta},{\bf J}_{\eta})$ and $({\bf E}_{\eta,\lambda},{\bf J}_{\eta,\lambda})$ are the numerical solutions of the original system and the extended system, respectively. In Fig.~5.2, we display the errors in logarithmic scale as a function of $\lambda/\gamma$, which allows us to visualize the convergence rates as the slopes of the curves. From Fig.~5.2, we observe that the extension method has the first-order convergence rate, which is faster than the result given in Theorem~\ref{thm:2.1}. A sharper error estimate for the extension method will be investigated in our future work.
\begin{figure}\label{fig:5-1}
\begin{center}
\begin{minipage}{0.4\linewidth}
    {\tiny(a)}\includegraphics[width=5.5cm,height=4.5cm]{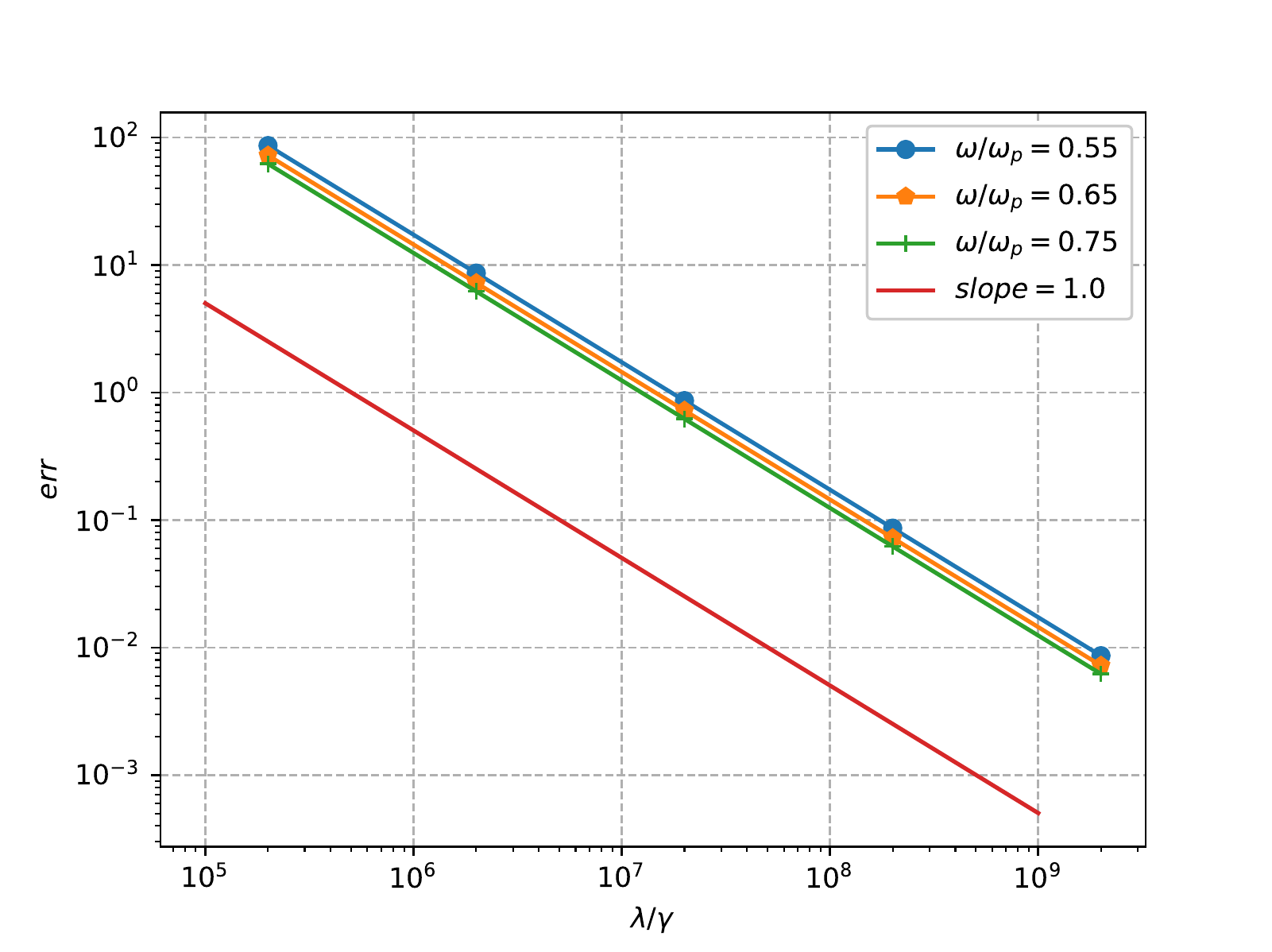}
\end{minipage}
\begin{minipage}{0.4\linewidth}
    {\tiny(b)}\includegraphics[width=5.5cm,height=4.5cm]{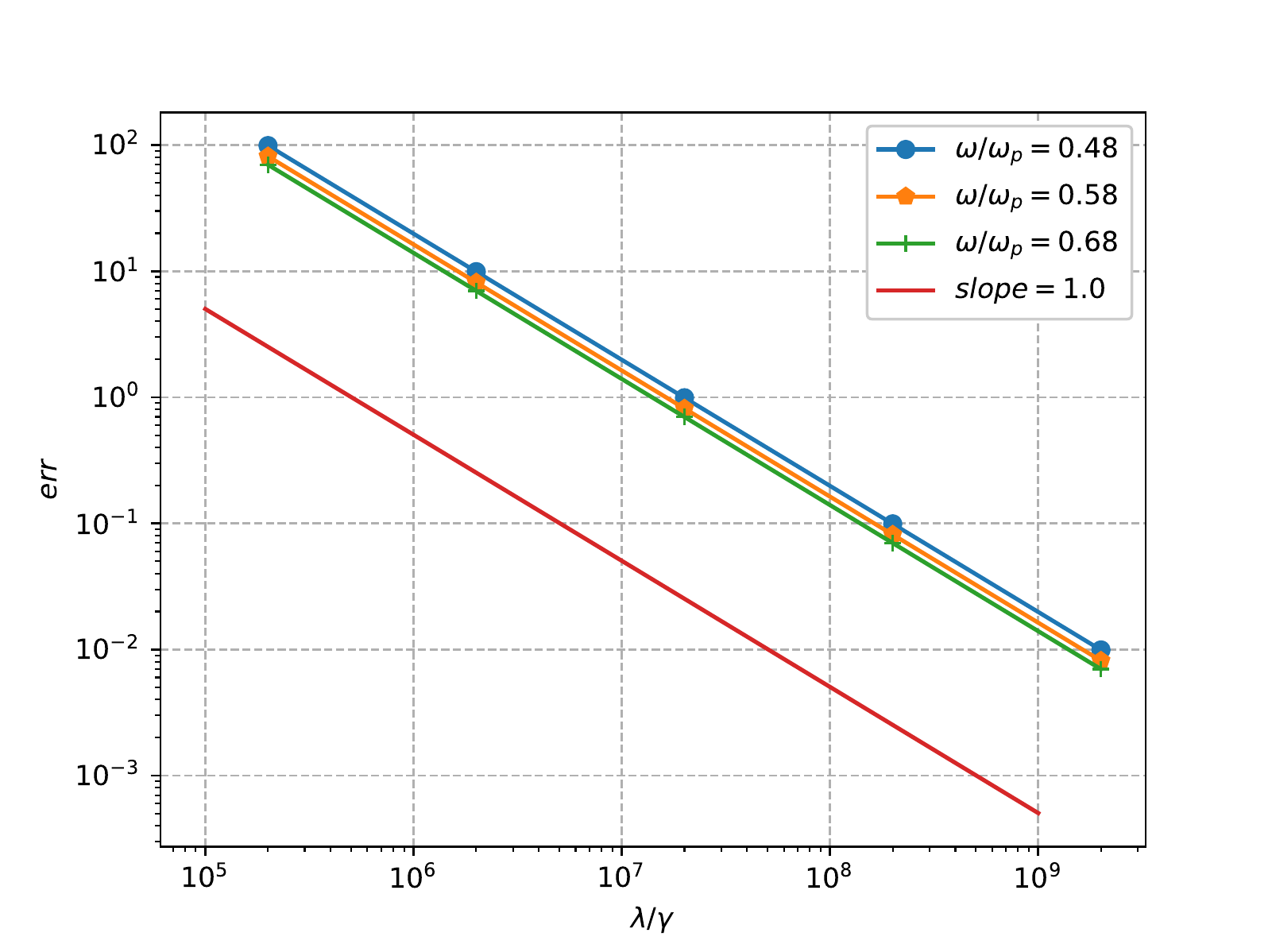}
\end{minipage}
\caption{Convergence rates of the extension (logarithmic scale). {\rm (a)}: Case 5.1; {\rm (b)}: Case 5.2.}
\end{center}
\end{figure}

Next, we give some numerical results to demonstrate the efficiency and accuracy of the proposed multiscale approach. Since in general it is impossible to get the analytic solution $({\bf E}_{\eta},{\bf J}_{\eta})$ of the original problem (\ref{eq:2.7}), in order to show the accuracy of our method, we regard the numerical solution of the original system on a very fine mesh as the reference solution.

\begin{figure}[!htbp]
\centering
{\tiny(a)}\includegraphics[width=12cm,height=5.7cm]{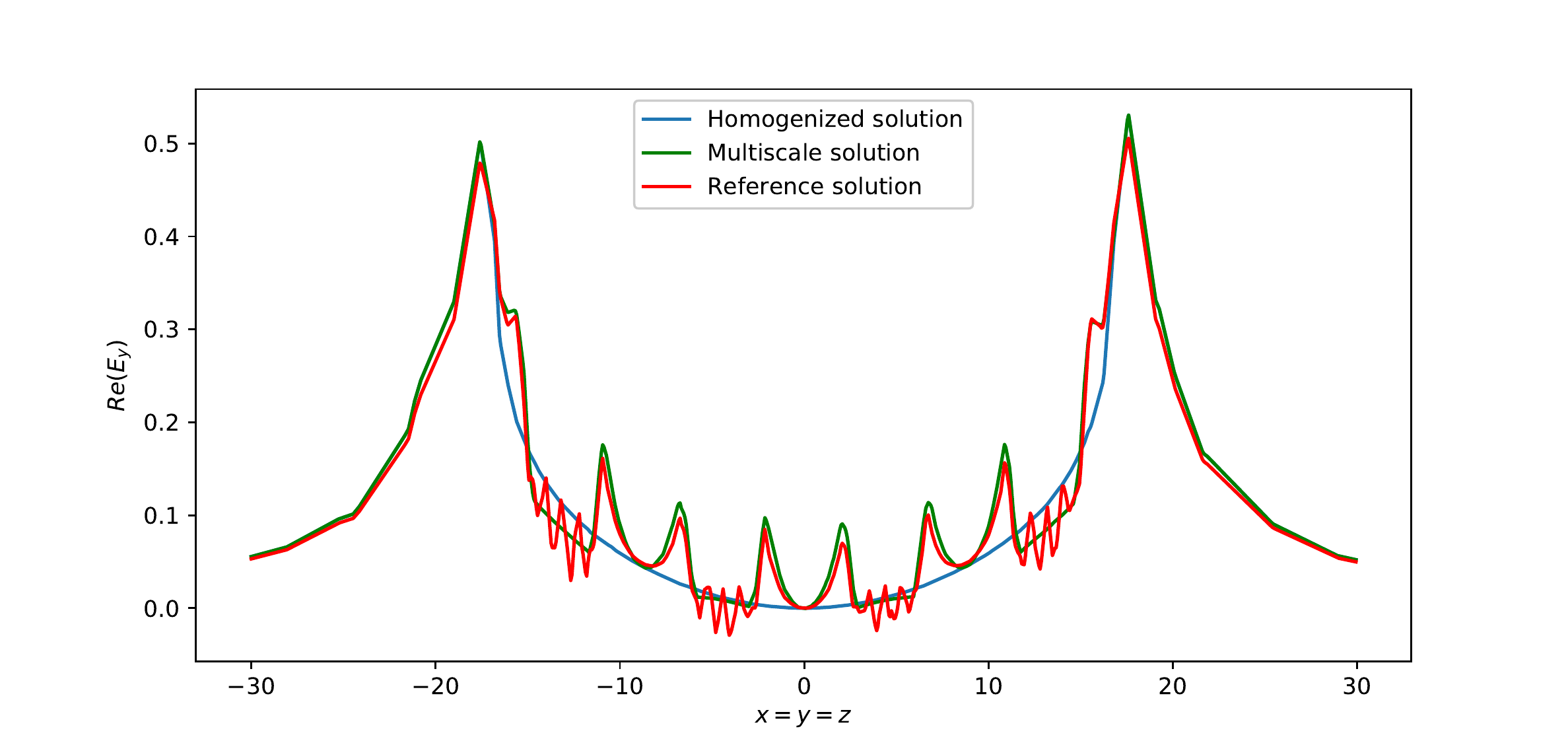}

{\tiny(b)}\includegraphics[width=12cm,height=5.7cm]{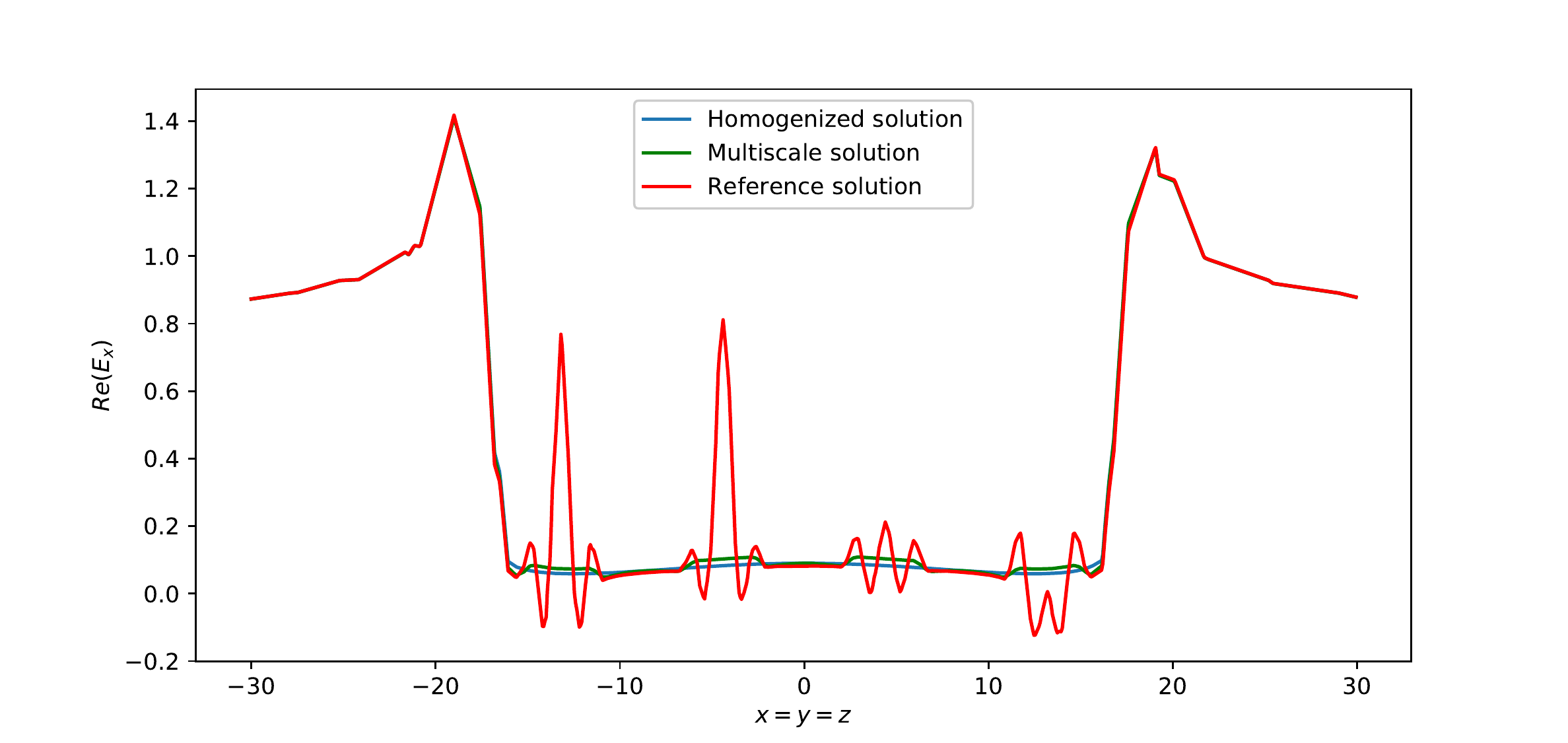}
\caption{Display of the electric field on the line $x=y=z$ calculated by the original multiscale method. (a): The real part of the $y$ component of the electric field in Case 5.1 at $\omega=0.75\omega_{p}$; (b): The real part of the $x$ component of the electric field in Case 5.2 at $\omega=0.48\omega_{p}$.
}\label{fig5-1-2}
\end{figure}

We first solve the problem by the original multiscale approach presented at the beginning of section~\ref{sec-4}. The numerical results for the electric field on the line $x=y=z$ in Case 5.1 and 5.2 are plotted in Fig.~\ref{fig5-1-2} (a)-(b). We clearly see that the multiscale approximate solution agrees well with the reference solution outside the metallic nanostructures while it fails to capture the oscillations of the electric field inside the metallic nanostructures.

To illustrate why the original multiscale approach fails inside the metallic nanostructures, we compute the minimal eigenvalue (denoted by $\alpha$) of the imaginary part of the homogenized coefficient $\widehat{\gamma^{\ast}}$ by solving the cell problem (\ref{eq:2-10-10-2}) numerically. The values of $\alpha/\gamma$ are displayed in Fig.~5.4 as a function of $\lambda/\gamma$, from which we see that the minimal eigenvalue $\alpha$ tends to infinity as the extension parameter $\lambda\rightarrow \infty$. This observation combined with Lemma~\ref{lem4.2} show that the electric field computed by the original multiscale approach with a large extension parameter loses the nonlocal information, leading to the failure of the original multiscale approach.
\begin{figure}\label{fig:5-1-1}
\begin{center}
\begin{minipage}{0.4\linewidth}
    \includegraphics[width=5.5cm,height=4.5cm]{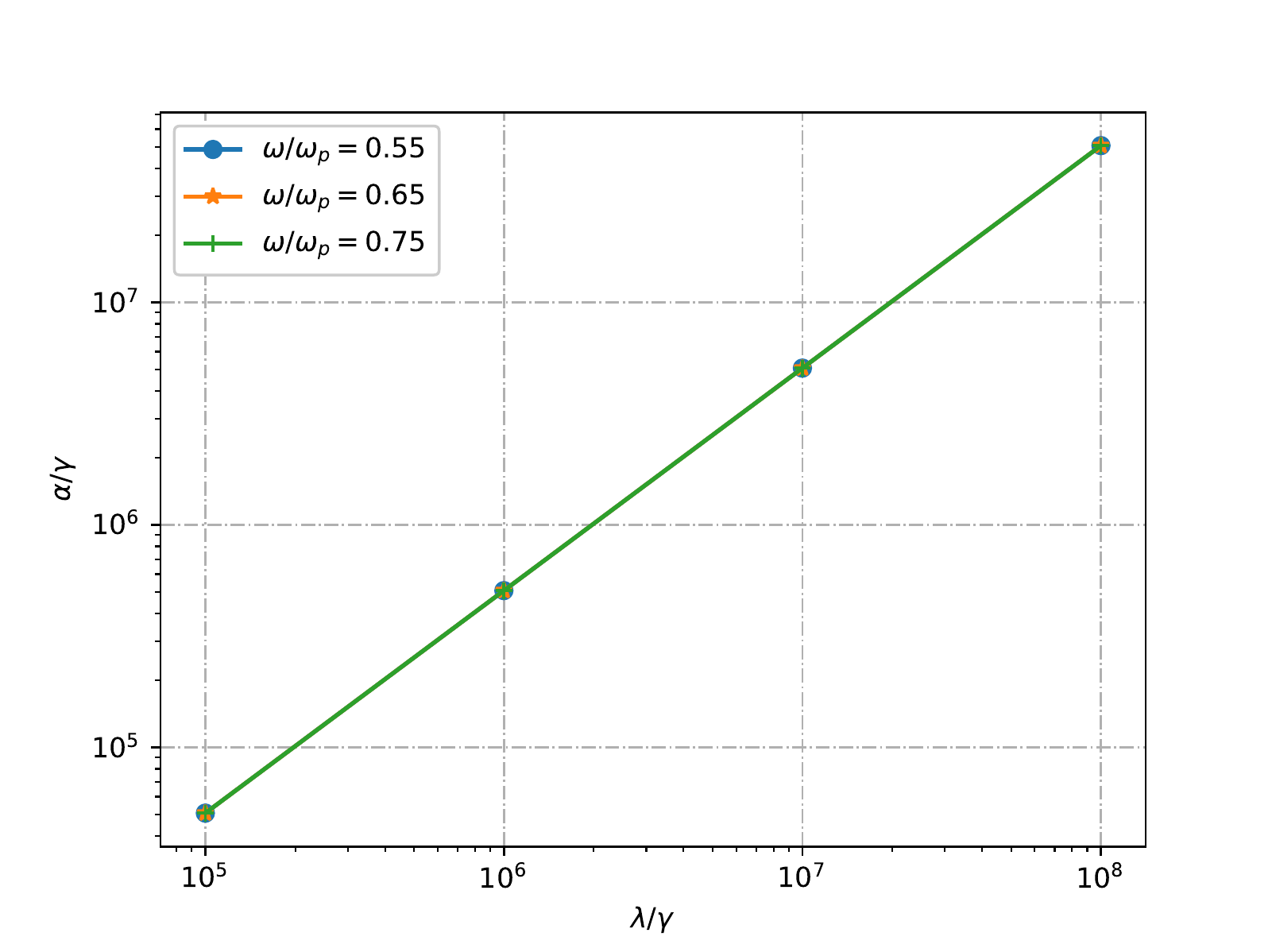}
\end{minipage}
    \caption{The minimal eigenvalue $\alpha$ of $\widehat{\gamma^{\ast}}$ (logarithmic scale).}
\end{center}
\end{figure}

Next we use the modified multiscale approach to solve the problem. The computational costs of solving the problem directly and solving by the modified multiscale approach for Case 5.1 at $\omega=0.75\omega_{p}$ are given in Table~\ref{table4}, which clearly show that our method is capable to reduce the computational cost dramatically.

\begin{table}[t]
\footnotesize
\caption{Comparison of computational costs for Case 5.1 at $\omega=0.75\omega_{p}$.}
\begin{center}
\begin{tabular}{ccccc}\hline
 & original problem & cell problem & homogenized problem & modified problem\\
\hline Elements & 2166403  & 68053 & 514247 & 12999 \\
 Dof & 31032404  & 13749 & 1231616 & 197766 \\
CPU Time (s) &3881.9739 &1.2144 &58.3542 &23.4486\\
\hline
\end{tabular}
\end{center}
\label{table4}
\end{table}

Without confusion, we denote by $({\bf E}_{\eta},{\bf J}_{\eta})$, ${\bf E}^{0}$, ${\bf E}^{0}_{\eta}$, $({\bf E}^{M}_{\eta}, {\bf J}^{M}_{\eta})$ the reference solution of the original problem, the homogenized solution for the electric field by solving the homogenized Maxwell equations (\ref{eq:4.6}) numerically, the multiscale approximate solution for the electric field, and the modified multiscale approximate solution by solving (\ref{eq:4.7}) inside each metallic nanostructure numerically, respectively. The numerical errors of the modified multiscale approach for Case 5.1 and 5.2 are displayed in Table~\ref{table2} and~\ref{table3}, respectively. Here we abbreviate $\|\cdot\|_{\mathbf{H}_{T}({\bf curl};\,\Omega)}$ and $\|\cdot\|_{\mathbf{H}({\bf div};\,\Omega_{s})}$ as $\|\cdot\|_{\Omega}$ and $\|\cdot\|_{\Omega_{s}}$ for brevity.

\begin{table}[t]
\centering
    \caption{Comparison of the numerical errors in Case~\ref{case5}.}\label{table2}
\begin{tabular}{ccccc}
\hline $\omega/ \omega_p$
     &
    $\frac{\|{\bf E}^{0}-\mathbf{E}_{\eta}\|_{\Omega}}{\|\mathbf{E}_{\eta}\|_{\Omega}}$ &
    $\frac{\|{\bf E}^{0}_{\eta}-\mathbf{E}_{\eta}\|_{\Omega}}{\|\mathbf{E}_{\eta}\|_{\Omega}}$&
    $\frac{\|{\bf E}^{M}_{\eta}-\mathbf{E}_{\eta}\|_{\Omega}}{\|\mathbf{E}_{\eta}\|_{\Omega}}$&
$\frac{\|\mathbf{J}^{M}_{\eta}-\mathbf{J}_{\eta}\|_{\Omega_s}}{\|\mathbf{J}_{\eta}\|_{\Omega_s}}$ \\
\noalign{\smallskip}\hline\noalign{\smallskip}
    0.55 & 0.021147 &0.021449&0.017761 &0.184561 \\
    0.65 & 0.019662 &0.015671&0.014059 &0.178703 \\
    0.75 & 0.020467 &0.013447&0.012417 &0.146900 \\
\noalign{\smallskip}\hline
\end{tabular}
\end{table}

\begin{table}[t]
\centering
    \caption{Comparison of the numerical errors in Case~\ref{case6}.}\label{table3}
\begin{tabular}{ccccc}
\hline $\omega/ \omega_p$
     &
     $\frac{\|{\bf E}^{0}-\mathbf{E}_{\eta}\|_{\Omega}}{\|\mathbf{E}_{\eta}\|_{\Omega}}$ &
    $\frac{\|{\bf E}^{0}_{\eta}-\mathbf{E}_{\eta}\|_{\Omega}}{\|\mathbf{E}_{\eta}\|_{\Omega}}$&
    $\frac{\|{\bf E}^{M}_{\eta}-\mathbf{E}_{\eta}\|_{\Omega}}{\|\mathbf{E}_{\eta}\|_{\Omega}}$&
$\frac{\|\mathbf{J}^{M}_{\eta}-\mathbf{J}_{\eta}\|_{\Omega_s}}{\|\mathbf{J}_{\eta}\|_{\Omega_s}}$ \\
\noalign{\smallskip}\hline\noalign{\smallskip}
    0.48 & 0.028744 &0.028645&0.008425 &0.091273 \\
    0.58 & 0.035826 &0.034270&0.009712 &0.046189 \\
    0.68 & 0.034353 &0.024464&0.022433 &0.140880 \\
\noalign{\smallskip}\hline
\end{tabular}
\end{table}

To show the accuracy of the modified multiscale method, in Fig.~\ref{fig5-1-3} we compare the modified multiscale approximate solution with the reference solution for the electric field in the same setting as Fig.~\ref{fig5-1-2}. We find that the modified multiscale approximate solution for the electric field is in good agreement with the reference solution in the whole computational domain.

\begin{figure}[!htbp]
\begin{center}
\tiny{(a)}\includegraphics[width=12cm,height=5.7cm]{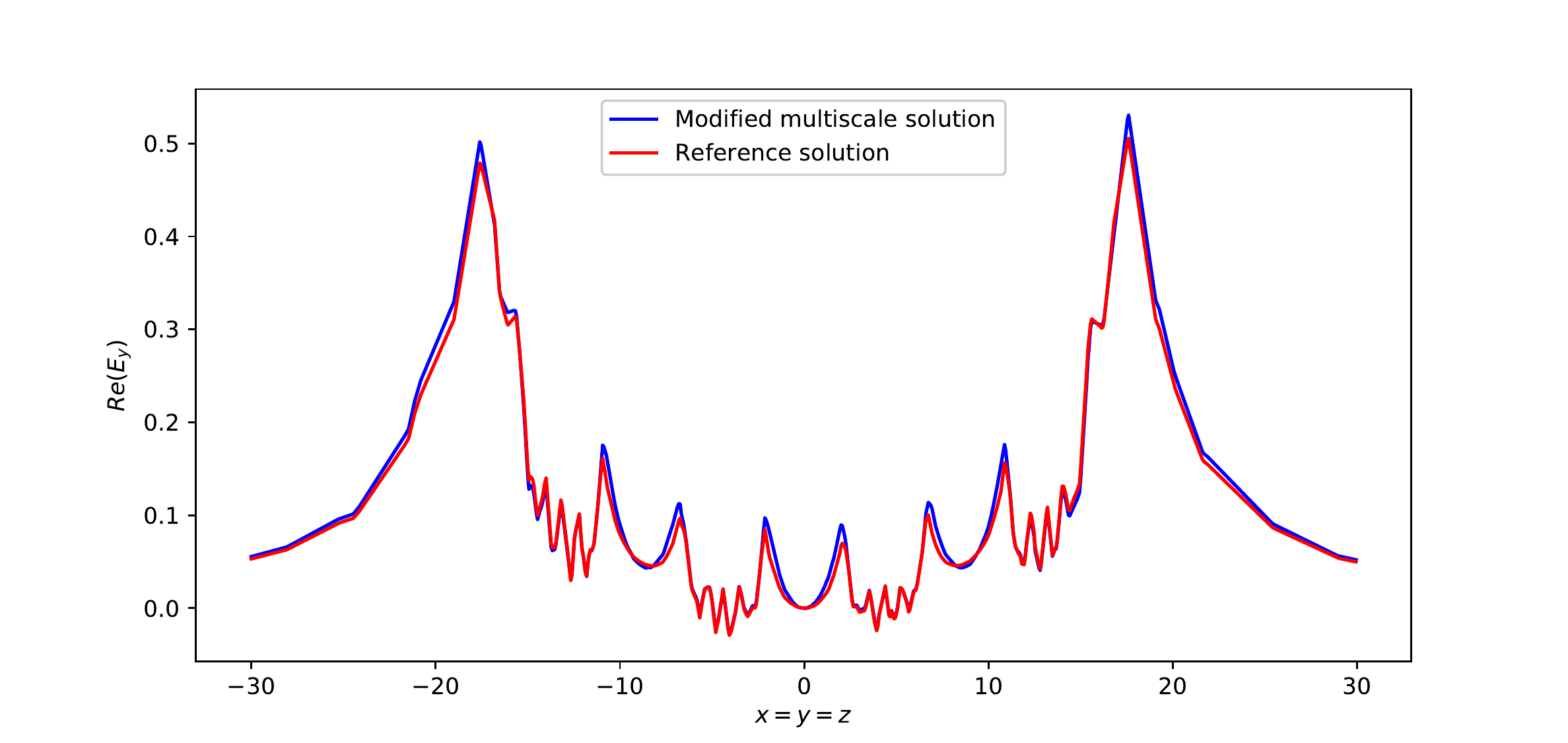}

\tiny{(b)}\includegraphics[width=12cm,height=5.7cm]{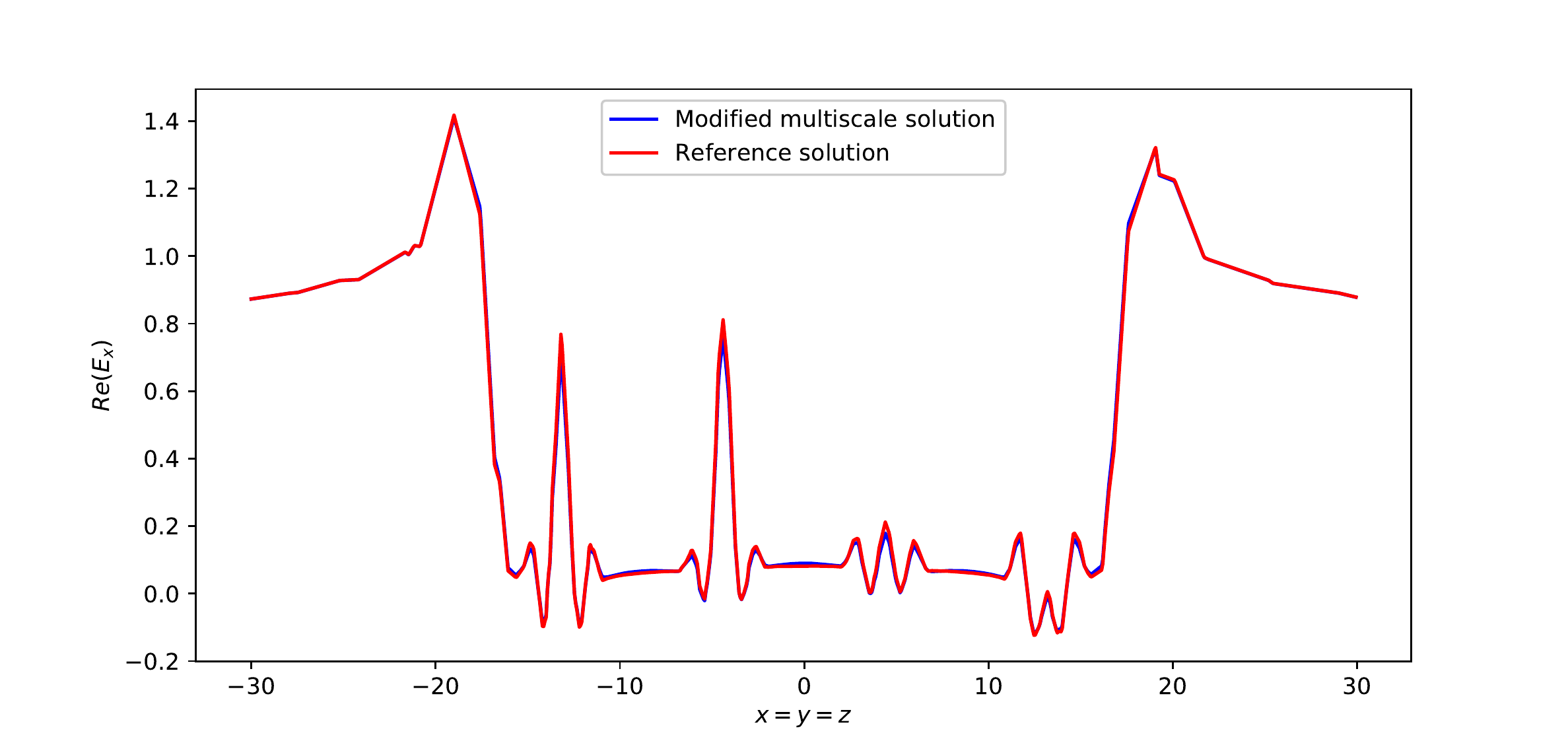}
\caption{Display of the electric field on the line $x=y=z$ calculated by the modified multiscale method. (a): The real part of the $y$ component of the electric field in Case 5.1 at $\omega=0.75\omega_{p}$; (b): The real part of the $x$ component of the electric field in Case 5.2 at $\omega=0.48\omega_{p}$.
}\label{fig5-1-3}
\end{center}
\end{figure}

Fig.~\ref{fig5-1-4} (a)-(d) display the numerical results for the electric field based
on the modified multiscale approach on the intersection $y=1.5$\,nm at $\omega=0.75\omega_p$ in Case 5.1.

In Fig.~\ref{fig5-1-5} (a)-(d), we compare the multiscale approximate solution with the reference solution for the polarization current on the intersection $y=1.5$\,nm in Case 5.1 and 5.2.

\begin{figure}[!htbp]
\centering
{\tiny(a)}\includegraphics[width=6.5cm,height=5cm]{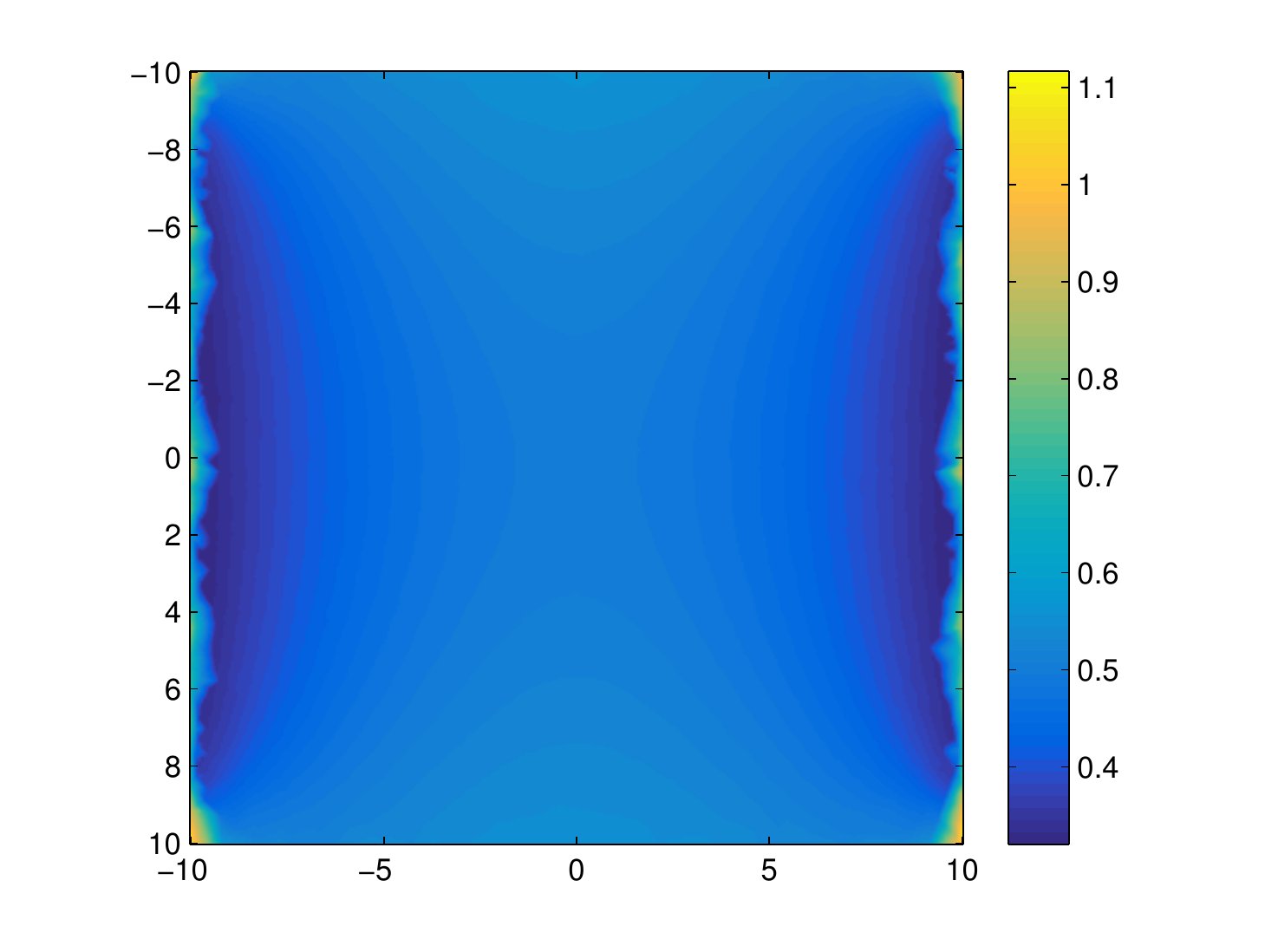}~
{\tiny(b)}\includegraphics[width=6.5cm,height=5cm]{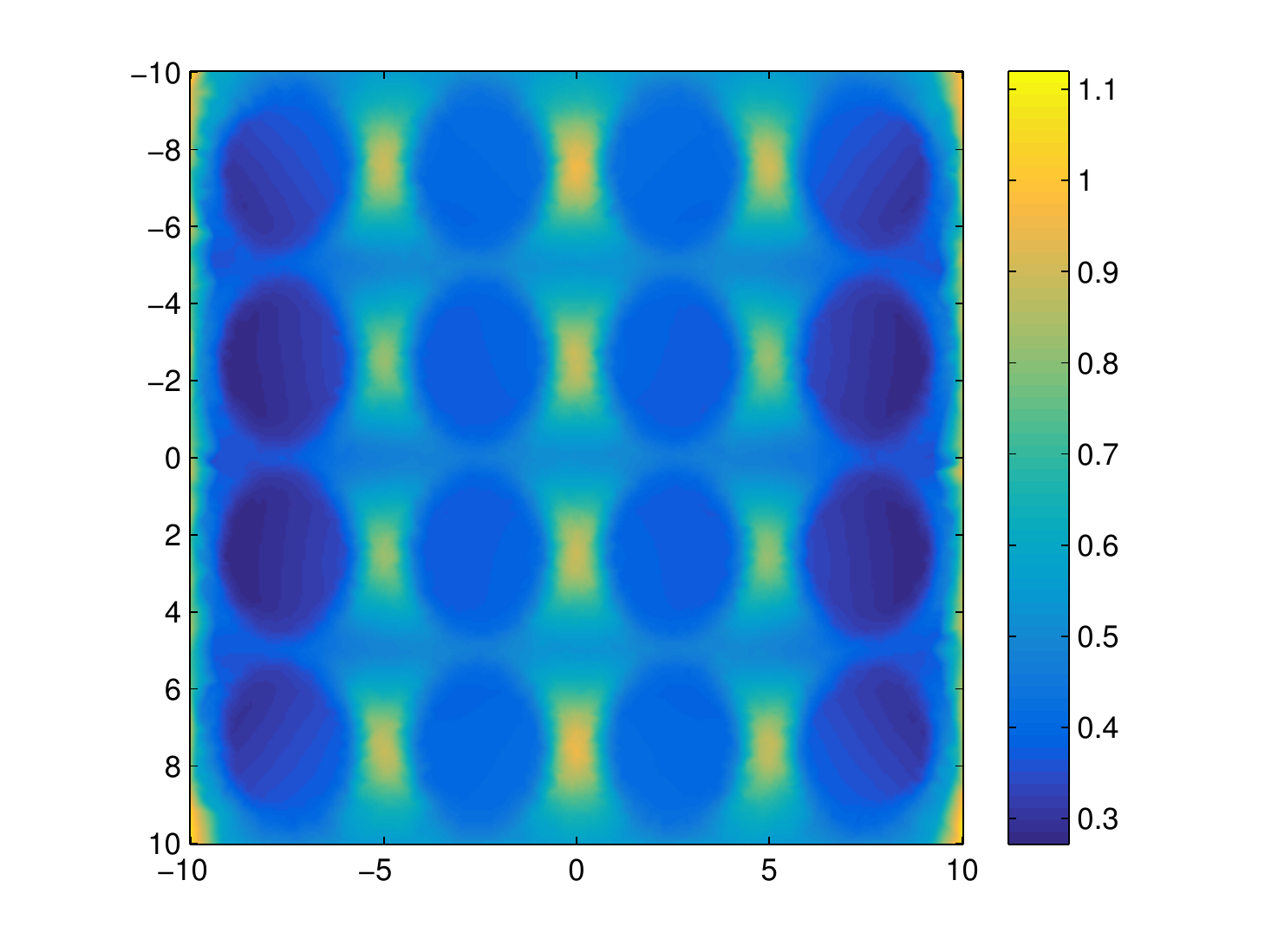}
{\tiny(c)}\includegraphics[width=6.5cm,height=5cm]{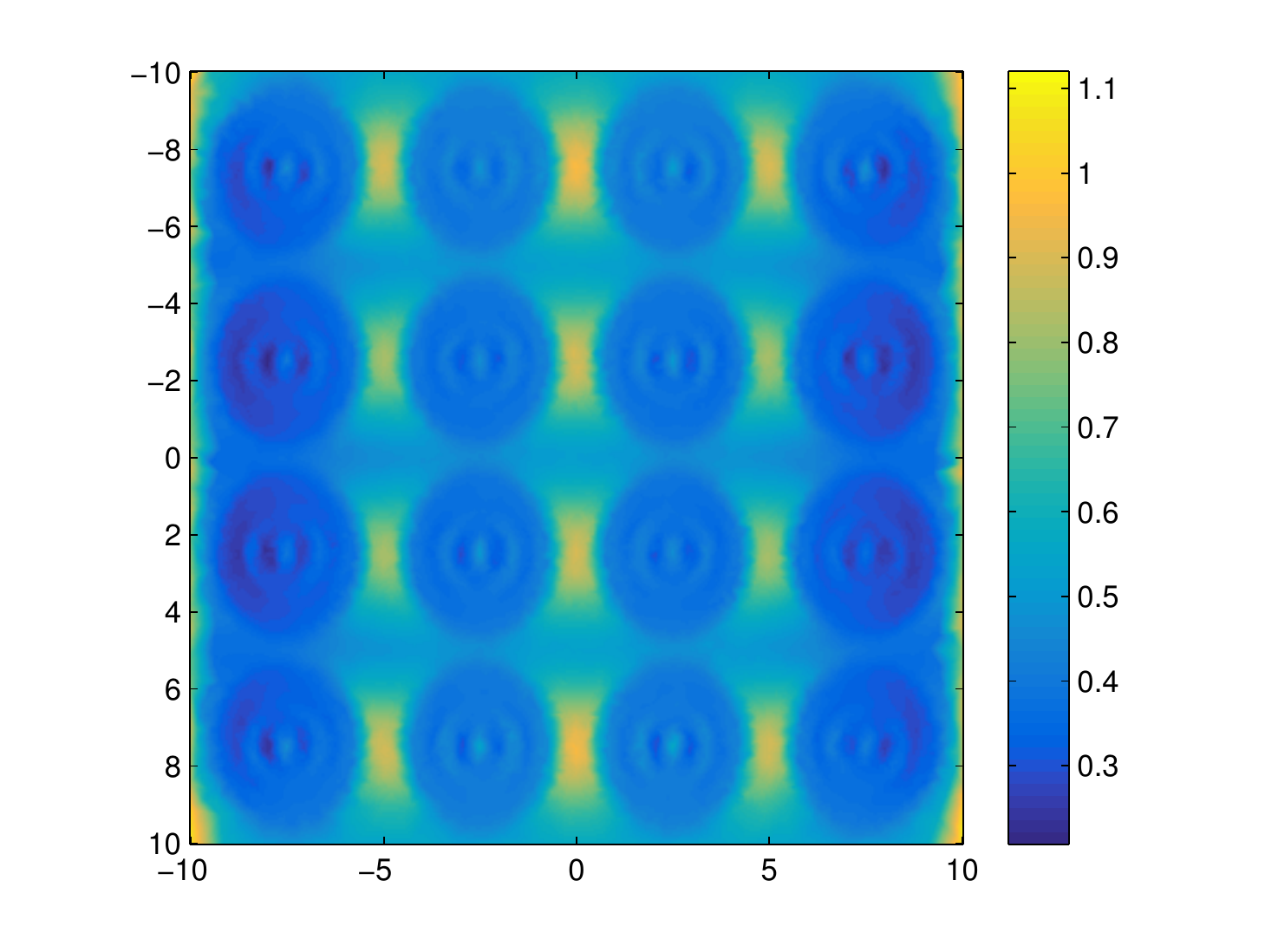}~
{\tiny(d)}\includegraphics[width=6.5cm,height=5cm]{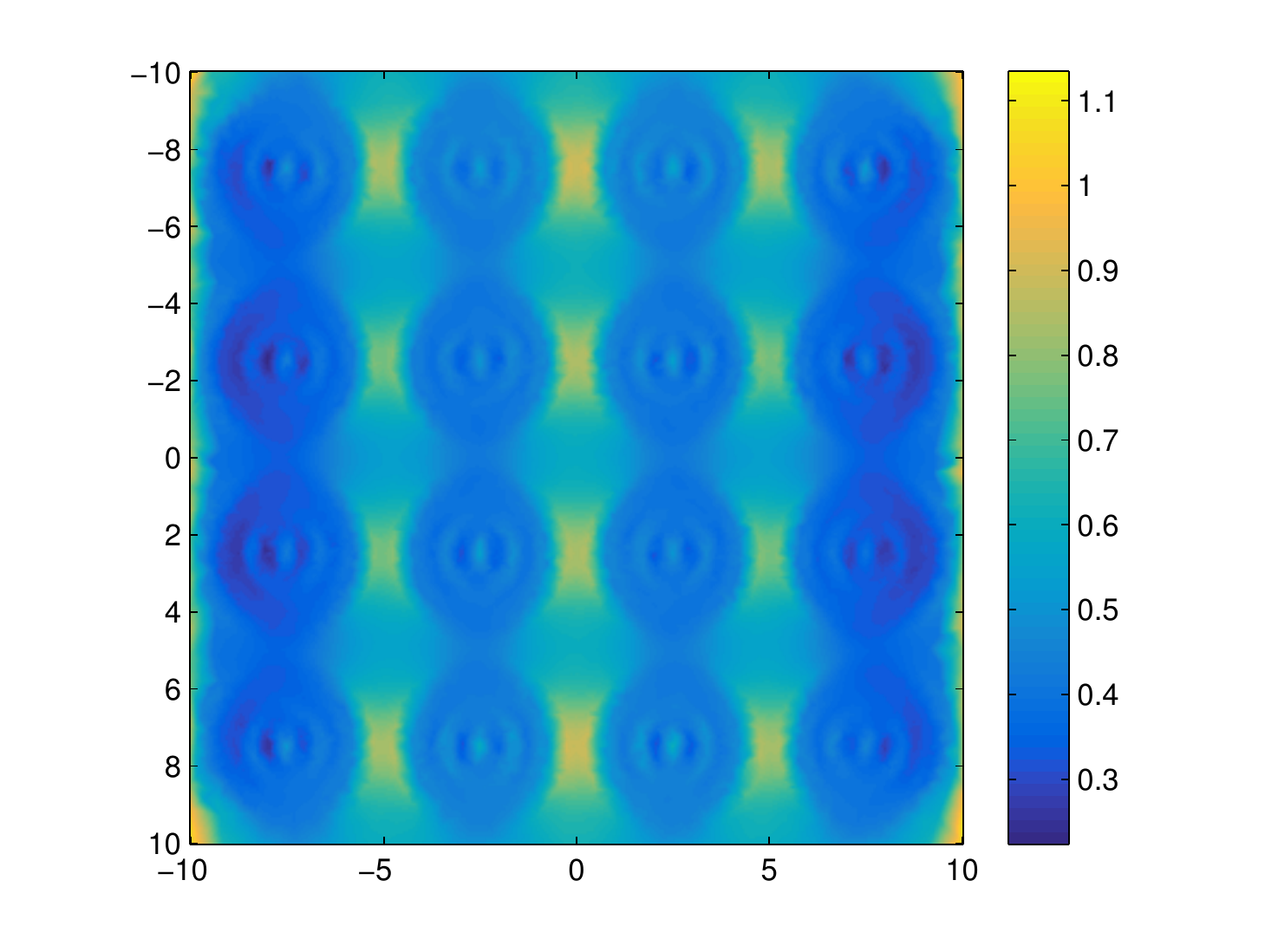}
    \caption{The $x$-component of the electric field $\mathbf{E}$ on the intersection $y=1.5$\,nm at $\omega=0.75\omega_p$ in Case~\ref{case5}:
    {\rm (a)} the homogenized solution $|\mathbf{E}^{0}_{x}|$; {\rm (b)} the multiscale approximate solution $|\mathbf{E}^{0}_{\eta,x}|$;
    {\rm (c)} the modified multiscale approximate solution $|\mathbf{E}^{M}_{\eta,x}|$; {\rm (d)} the reference solution $|{\mathbf{E}}_{\eta,x}|$.
}\label{fig5-1-4}
\end{figure}

\begin{figure}[!htbp]
\centering
{\tiny(a)}\includegraphics[width=6.5cm,height=5cm]{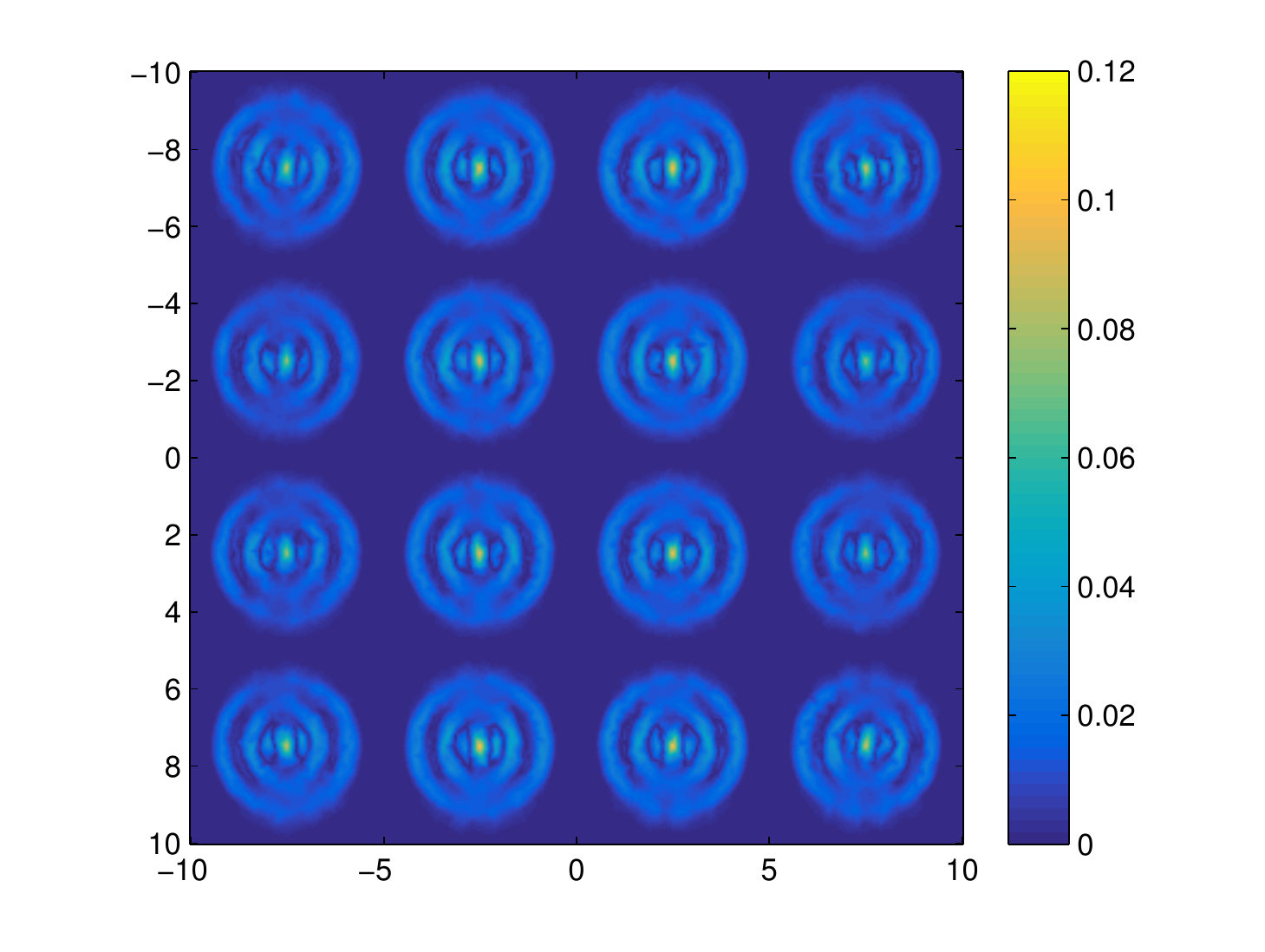}~
{\tiny(b)}\includegraphics[width=6.5cm,height=5cm]{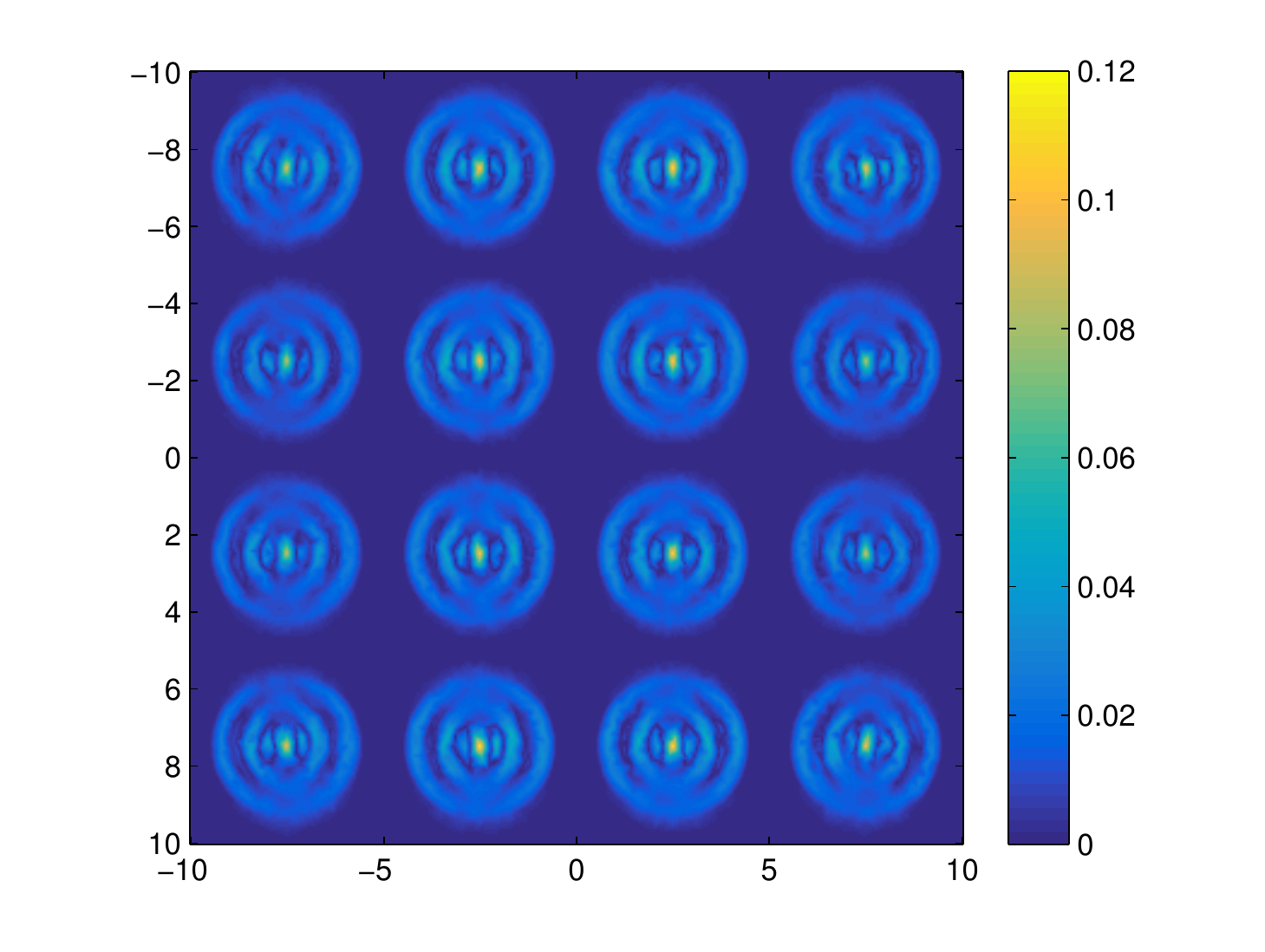}
{\tiny(a)}\includegraphics[width=6.5cm,height=5cm]{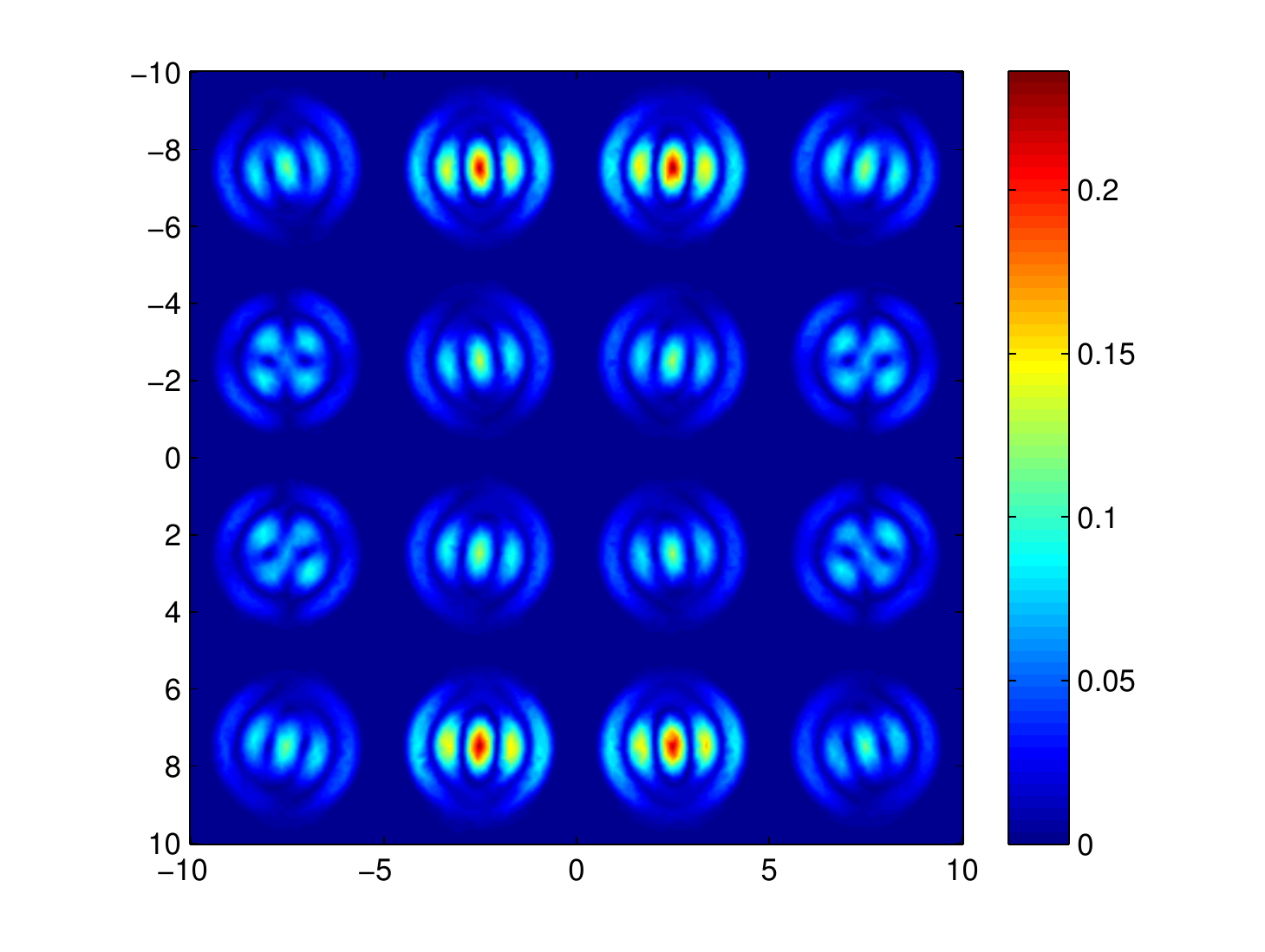}~
{\tiny(b)}\includegraphics[width=6.5cm,height=5cm]{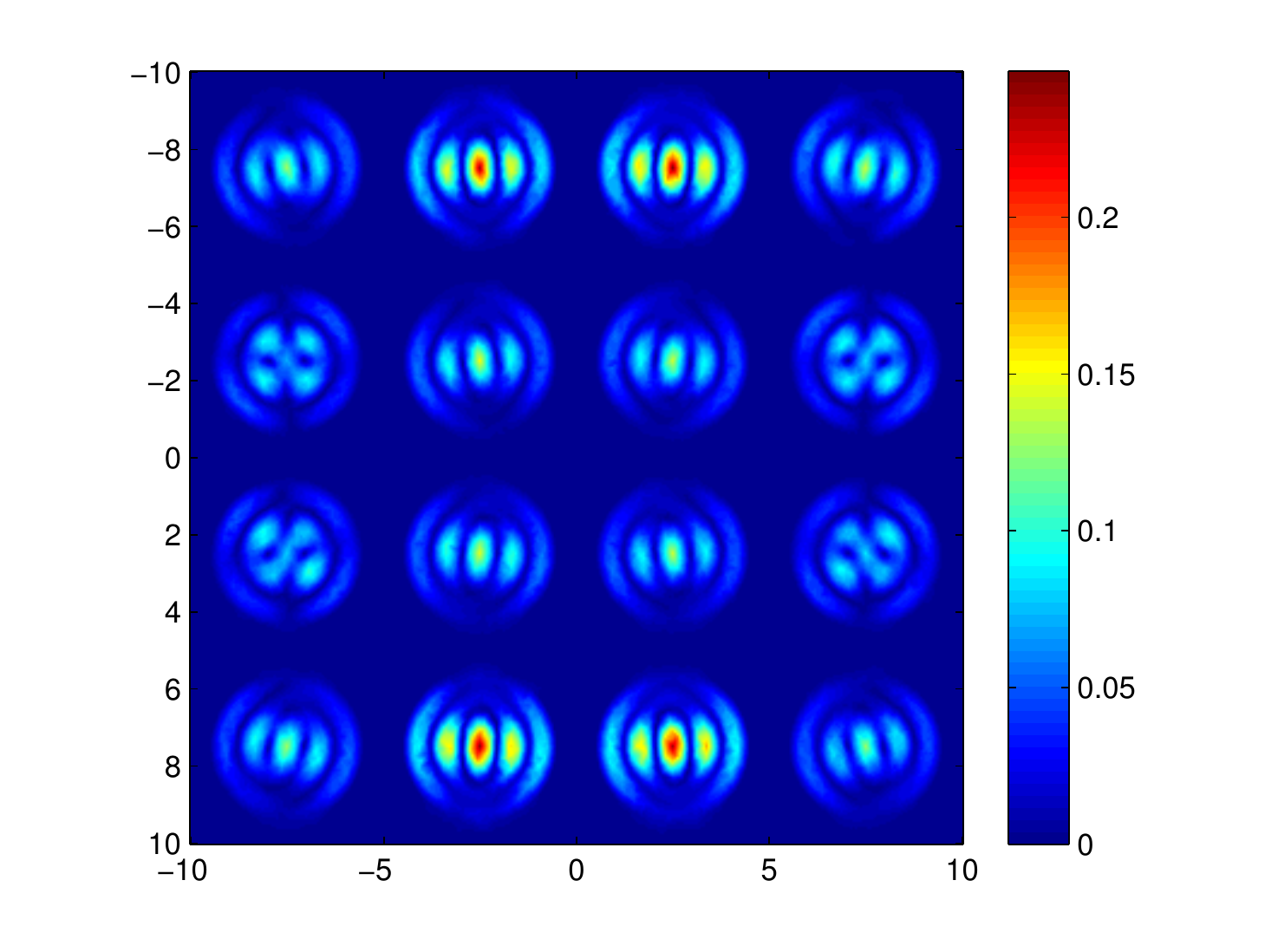}
    \caption{The $x$-component of the polarization current $\mathbf{J}$ on the intersection $y=1.5$\,nm:
    {\rm (a)} the modified multiscale approximate solution $|\mathbf{J}^{M}_{\eta,x}|$ in Case 5.1 at $\omega=0.75\omega_{p}$; {\rm (b)} the reference solution $|\mathbf{J}_{\eta,x}|$ in Case 5.1 at $\omega=0.75\omega_{p}$;   {\rm (c)} the modified multiscale approximate solution $|\mathbf{J}^{M}_{\eta,x}|$ in Case 5.2 at $\omega=0.48\omega_{p}$; {\rm (d)} the reference solution $|\mathbf{J}_{\eta,x}|$ in Case 5.2 at $\omega=0.48\omega_{p}$.}\label{fig5-1-5}
\end{figure}

\section{Conclusions}
We have presented a novel multiscale approach to simulate the nonlocal optical response of metallic nanostructure arrays embedded in a dielectric medium by solving a nonlocal hydrodynamic Drude (NHD) model with rapidly oscillating coefficients. Our approach contains two key steps. First, we solve a homogenized problem on the whole domain, which is much easier to solve than the original system  since the coefficients are constant. Next, we solve the original system in each metallic nanostructure separately to capture the nonlocal information of the electric field. A fast algorithm based on the $LU$ decomposition is proposed to solve the resulting linear system. In this way, we make the numerical results of our approach agree well with the reference solution on the whole domain. More importantly, the computational burden of our approach is much lower than that of solving the original system directly. Numerical experiments are presented to demonstrate the accuracy and efficiency of the proposed multiscale approach.

Based on the work in this paper, in the near future we will study applications of the multiscale approach to the time-dependent NHD model to simulate the transient nonlocal optical response of metallic nanostructure arrays. Furthermore, from the view of theoretical analysis, another focus of future works is to prove the convergence of the homogenization method developed in this paper.

\appendix
\section{Proof of the error bound for the extension} We provide a proof of Theorem~\ref{thm:2.1}. To begin with, we introduce some notations used in the proof. Let ${L}^{p}(\Omega)$ and $\mathbf{L}^{p}(\Omega) = [L^{p}(\Omega)]^{3} $ be the Lebesgue spaces of complex-valued functions and vector-valued functions, respectively. $ L^{2}$ inner-products in $L^{2}(\Omega) $ and $\mathbf{L}^{2}(\Omega)$ are denoted by $(\cdot,\cdot )$ without ambiguity. To avoid confusion, we use $(\cdot,\cdot )_{s}$ and $(\cdot,\cdot )_{s,\eta}$ to denote the $ L^{2}$ inner-products in $L^{2}(\Omega_{s}) $ ($\mathbf{L}^{2}(\Omega_{s})$) and $L^{2}(\Omega_{s,\eta}) $ ($\mathbf{L}^{2}(\Omega_{s,\eta})$), respectively.

We define
\begin{equation}\label{eq2-1}
\begin{array}{lll}
{\displaystyle \mathbf{H}(\mathbf{curl};\Omega)=\{\mathbf{u}\in
\mathbf{L}^{2}(\Omega)\,|\, \,\mathbf{curl}\, \mathbf{u}\in \mathbf{L}^{2}(\Omega)
\}, }\\[2mm]
{\displaystyle \mathbf{H}(\mathbf{div};\Omega)=\{\mathbf{u}\in
\mathbf{L}^{2}(\Omega)\,|\,\, \mathbf{div}\, \mathbf{u}\in L^{2}(\Omega)
\}, }\\[2mm]
\end{array}
\end{equation}
which are equipped with the norms
\begin{equation*}
\begin{array}{lll}
{\displaystyle \|\mathbf{u}\|_{\mathbf{H}(\mathbf{curl};\Omega)}=\|\mathbf{u}\|_{\mathbf{L}^2(\Omega)}+\|\mathbf{curl}\, \mathbf{u}\|_{\mathbf{L}^2(\Omega)}, }\\[2mm]
{\displaystyle \|\mathbf{u}\|_{\mathbf{H}(\mathbf{div};\Omega)}=\|\mathbf{u}\|_{\mathbf{L}^2(\Omega)}+\|\mathbf{div}\, \mathbf{u}\|_{L^2(\Omega)}. }
\end{array}
\end{equation*}
In addition,
\begin{equation}\label{eq2-2}
\begin{array}{lll}
{\displaystyle \mathbf{H}_{T}(\mathbf{curl};\Omega)=\{\mathbf{u}\in
\mathbf{H}(\mathbf{curl};\Omega)\, |  \,\, \mathbf{u}_{T} = ({\bf n}\times {\bf u})\times{\bf n} \in {\bf L}^{2}(\partial \Omega)  \,\, \,{\rm on}\, \,\partial \Omega\}, }\\[2mm]
{\displaystyle \mathbf{H}_{0}(\mathbf{curl};\Omega)=\{\mathbf{u}\in
\mathbf{H}(\mathbf{curl};\Omega)\, |  \,\, \mathbf{u}\times \mathbf{n}={\bf 0} \,\, \,{\rm on}\, \,\partial \Omega\}, }\\[2mm]
{\displaystyle \mathbf{H}_{0}(\mathbf{div};\Omega)=\{\mathbf{u}\in
\mathbf{H}(\mathbf{div};\Omega)\, |  \,\, \mathbf{u}\cdot\mathbf{n}=0 \,\, \,{\rm on}\, \,\partial \Omega\}. }\\[2mm]
\end{array}
\end{equation}
Functions in $\mathbf{H}_{T}(\mathbf{curl};\Omega)$ are equipped with the norm
\begin{equation}
\|\mathbf{u}\|_{\mathbf{H}_{T}(\mathbf{curl};\Omega)}=\|\mathbf{u}\|_{\mathbf{L}^2(\Omega)}+\|\mathbf{curl}\, \mathbf{u}\|_{\mathbf{L}^2(\Omega)} + \|\mathbf{u}_{T}\|_{\mathbf{L}^2(\partial \Omega)}.
\end{equation}
For the sake of convenience, we denote by
\begin{equation}\label{eq2-3}
\begin{array}{lll}
{\displaystyle X(\Omega) =\mathbf{H}_{T}(\mathbf{curl};\Omega),\;\qquad \|\mathbf{u}\|_{X(\Omega)}= \|\mathbf{u}\|_{\mathbf{H}_T(\mathbf{curl};\Omega)} , }\\[2mm]
{\displaystyle Y(\Omega_s) = \mathbf{H}_{0}(\mathbf{div};\Omega_s), \;\qquad \|\mathbf{v}\|_{Y(\Omega_s)} = \|\mathbf{v}\|_{\mathbf{H}(\mathbf{div};\Omega_s)}, }\\[2mm]
{\displaystyle Y(\Omega_{s,\eta}) = \mathbf{H}_{0}(\mathbf{div};\Omega_{s,\eta}), \quad \|\mathbf{v}\|_{Y(\Omega_{s,\eta})} = \|\mathbf{v}\|_{\mathbf{H}(\mathbf{div};\Omega_{s,\eta})}.}
\end{array}
\end{equation}

The weak formulation of the original system (\ref{eq:2.7}) reads as follows: Find $(\mathbf{E}_{\eta},{\bf J}_{\eta}) \in X(\Omega)\times Y(\Omega_{s,\eta})$, such that the equations
 \begin{equation}\label{eq2-4}
\left\{
\begin{array}{@{}l@{}}
    {\displaystyle  \big(\tilde{\mu}_{\eta}^{-1}\mathbf{curl}\,{\bf E}_{\eta}, \,\mathbf{curl}\,{\bf u}) - \omega^{2}(\tilde{\varepsilon}_\eta{\bf E}_{\eta}, \,{\bf u}) - {\mathrm {{i}}}\omega \langle {\bf E}_{\eta,T}, {\bf u}_{T}\rangle = \langle {\bf g}, {\bf u}_{T}\rangle + {\mathrm {{i}}}\omega\big({\bf J}_{\eta}, \,{\bf u}\big)_{s,\eta}}\\[2mm]
    {\displaystyle \big(\beta^{2}\mathbf{div}\, {\bf J}_{\eta},\,\mathbf{div}\, {\bf v}\big)_{s,\eta}-(\omega(\omega+{\mathrm {{i}}}\gamma){\bf J}_{\eta},\,{\bf v})_{s,\eta} = -{\mathrm {{i}}}\omega\omega^{2}_{p}\varepsilon_{0} \big({\bf E}_{\eta},\,{\bf v}\big)_{s,\eta}}
\end{array}
\right.
\end{equation}
hold for each $({\bf u}, {\bf v}) \in X(\Omega)\times Y(\Omega_{s,\eta})$, where $\tilde{\mu}_{\eta}$ and $\tilde{\varepsilon}_\eta$ are given by
\begin{equation}
\tilde{{\mu}}_{\eta} =
\left\{
\begin{array}{ll}
  {{\mu}_{\eta}},\ &{\rm in}\,\,\; \Omega_{s} \\[2mm]
{\mu_0} ,\ &{\rm in}\,\, \;\Omega/{\Omega}_{s},
\end{array}
\right.
\qquad
\tilde{{\varepsilon}}_{\eta} =
\left\{
\begin{array}{ll}
   {{\varepsilon}_{\eta}},\ &{\rm in}\,\,\; \Omega_{s} \\[2mm]
  {\varepsilon_0} ,\ &{\rm in}\,\,\; \Omega/{\Omega}_{s},
\end{array}
\right.
\end{equation}
${\bf u}_{T} = ({\bf n}\times {\bf u})\times {\bf n}$ with ${\bf n}$ being the unit outward normal to $\partial \Omega$, and $\langle\cdot,\cdot\rangle$ denotes the $L^{2}$ inner product in $\mathbf{L}^{2}(\partial \Omega)$. Note that here we use $\mathbf{E}_{\lambda}$ to denote the electric field on the whole domain for convenience.

Similarly, we can give the weak formulation for the extended system (\ref{eq:2.9}). Find $(\mathbf{E}_{\eta,\lambda},{\bf J}_{\eta,\lambda}) \in X(\Omega)\times Y(\Omega_{s})$, such that the equations
 \begin{equation}\label{eq2-4-0}
\left\{
\begin{array}{@{}l@{}}
    {\displaystyle  \big(\tilde{\mu}_{\eta}^{-1}\mathbf{curl}\,{\bf E}_{\eta,\lambda}, \,\mathbf{curl}\,{\bf u}) - \omega^{2}(\tilde{\varepsilon}_\eta{\bf E}_{\eta,\lambda}, \,{\bf u}) - {\mathrm {{i}}}\omega \langle {\bf E}_{\eta,\lambda,T}, {\bf u}_{T}\rangle = \langle {\bf g}, {\bf u}_{T}\rangle + {\mathrm {{i}}}\omega\big({\bf J}_{\eta,\lambda}, \,{\bf u}\big)_{s}}\\[2mm]
    {\displaystyle \big(\beta^{2}_{\eta,\lambda}\mathbf{div}\, {\bf J}_{\eta,\lambda},\,\mathbf{div}\, {\bf v}\big)_{s}-(\omega(\omega+{\mathrm {{i}}}\gamma_{\eta,\lambda}){\bf J}_{\eta,\lambda},\,{\bf v})_{s} = -{\mathrm {{i}}}\omega\omega^{2}_{p}\varepsilon_{0} \big({\bf E}_{\eta,\lambda},\,{\bf v}\big)_{s}}
\end{array}
\right.
\end{equation}
hold for each $({\bf u}, {\bf v}) \in X(\Omega)\times Y(\Omega_{s})$. In the same spirit, here $\mathbf{E}_{\eta,\lambda}$ denotes the electric field on the whole domain.

The following lemma gives the well-posedness of the equations (\ref{eq2-4}) and (\ref{eq2-4-0}).
\begin{lemma}{\rm \cite[{\rm Theorem 3.1}]{ma2019mathematical}}
Let $\Omega$, $\Omega_s$, and $\Omega_{s,\eta}$ be the bounded, simply-connected, Lipschitz domains in $\mathbb{R}^{3}$ with $\bar{\Omega}_{s}\subset \Omega$ and $\bar{\Omega}_{s,\eta}\subset \Omega_{s}$. Assumes that $\mathbf{g}\in \mathbf{L}^{2}(\partial \Omega)$. There exist a unique solution $({\bf E}_{\eta},\,{\bf J}_{\eta}) \in X(\Omega)\times Y(\Omega_{s,\eta})$ to the equations (\ref{eq2-4}) satisfying
\begin{equation}\label{eq2-7}
\Vert {\bf E}_{\eta} \Vert_{X(\Omega)} + \Vert {\bf J}_{\eta} \Vert_{Y(\Omega_{s,\eta})} \leq C \Vert {\bf g}\Vert_{{\bf L}^{2}(\partial \Omega)},
\end{equation}
where $C$ is a positive constant. Similarly, there exist a unique solution $({\bf E}_{\eta,\lambda},\,{\bf J}_{\eta,\lambda}) \in X(\Omega)\times Y(\Omega_{s})$ to the equations (\ref{eq2-4-0}) satisfying
\begin{equation}\label{eq2-7-0}
\Vert {\bf E}_{\eta,\lambda} \Vert_{X(\Omega)} + \Vert {\bf J}_{\eta,\lambda} \Vert_{Y(\Omega_{s})} \leq C \Vert {\bf g}\Vert_{{\bf L}^{2}(\partial \Omega)}.
\end{equation}
\end{lemma}

In order to estimate the error between $({\bf E}_{\eta},{\bf{J}}_{\eta})$ and $({\bf E}_{\eta,\lambda},{\bf{J}}_{\eta,\lambda})$, we need a decay estimate for the extended polarization current ${\bf J}_{\eta,\lambda}$ in $\Omega_{s}/\Omega_{s,\eta}$ with respect to $\lambda$.
\begin{lemma}\label{lemma1}
Let $({\bf E}_{\eta,\lambda}, {\bf J}_{\eta,\lambda})$ be the solution of the extended system (\ref{eq2-4-0}). There exists a constant $C>0$ independent of $\lambda$ such that
    \begin{equation}\label{eq2-7-1}
        \Vert {\bf J}_{\eta,\lambda}\Vert_{\mathbf{L}^2(\Omega_{s}/\Omega_{s,\eta})}\leq C/{\lambda}^{1/2},\,\quad \Vert \mathbf{div}\, {\bf J}_{\eta,\lambda}\Vert_{\mathbf{L}^2(\Omega_{s}/\Omega_{s,\eta})}\leq C/{\lambda}^{1/2}.
   \end{equation}
\end{lemma}
\begin{proof}
By choosing the test function $({\bf u},{\bf v}) =({\bf E}_{\eta,\lambda}, {\bf J}_{\eta,\lambda})$ in (\ref{eq2-4-0}), and taking the imaginary part of the two equations of (\ref{eq2-4-0}), we obtain
    \begin{equation}\label{eq2-8}
\begin{array}{@{}l@{}}
   {\displaystyle  \omega\|{\bf E}_{\eta,\lambda,T}\|^2_{\mathbf{L}^2(\partial \Omega)}=-\omega\,{\rm Re}\big({\bf J}_{\eta,\lambda},\,{\bf E}_{\eta,\lambda}\big)_{s}- {\rm Im}\langle {\bf g}, {\bf E}_{\eta,\lambda,T}\rangle,}\\[2mm]
  {\displaystyle \big(\gamma_{\eta,\lambda}{\bf J}_{\eta,\lambda},\, {\bf J}_{\eta,\lambda}\big)_{s} = \omega^{2}_{p}\varepsilon_{0}\,{\rm Re}\big({\bf E}_{\eta,\lambda},\,{\bf J}_{\eta,\lambda}\big)_{s},}
\end{array}
\end{equation}
which implies that
\begin{equation}\label{eq2-8-0}
\begin{array}{lll}
{\displaystyle \big(\gamma_{\eta,\lambda}{\bf J}_{\eta,\lambda},\, {\bf J}_{\eta,\lambda}\big)_{s} + \omega^{2}_{p}\varepsilon_{0}\|{\bf E}_{\eta,\lambda,T}\|^2_{\mathbf{L}^2(\partial \Omega)} \leq \frac{\omega^{2}_{p}\varepsilon_{0}}{\omega}|\langle {\bf g}, {\bf E}_{\eta,\lambda,T}\rangle| }\\[2mm]
{\displaystyle \qquad \leq \frac{\omega^{2}_{p}\varepsilon_{0}}{2}\|{\bf E}_{\eta,\lambda,T}\|^2_{\mathbf{L}^2(\partial \Omega)} + C\|{\bf g}\|^2_{\mathbf{L}^2(\partial \Omega)}.}
\end{array}
\end{equation}
Since $\gamma_{\eta,\lambda}|_{\Omega_{s}/\Omega_{s,\eta}} = \lambda$, from (\ref{eq2-8-0}) we deduce
\begin{equation}\label{eq2-8-1}
\lambda\Vert {\bf J}_{\eta,\lambda}\Vert^{2}_{\mathbf{L}^2(\Omega_{s}/\Omega_{s,\eta})}\leq  \big(\gamma_{\eta,\lambda}{\bf J}_{\eta,\lambda},\, {\bf J}_{\eta,\lambda}\big)_{s} \leq C\|{\bf g}\|^2_{\mathbf{L}^2(\partial \Omega)},
\end{equation}
which gives the first inequality in (\ref{eq2-7-1}). By applying Theorem~4.17 of \cite{Monk} to the first equation of (\ref{eq2-4-0}), we further have
    \begin{equation}\label{eq2-8-2}
\Vert {\bf E}_{\eta,\lambda} \Vert_{X(\Omega)} \leq C(\|{\bf g}\|_{\mathbf{L}^2(\partial \Omega)}+\Vert {\bf J}_{\eta,\lambda}\Vert_{\mathbf{L}^2(\Omega_{s})})\leq C\|{\bf g}\|_{\mathbf{L}^2(\partial \Omega).}
    \end{equation}

Next we choose ${\bf v}={\bf J}_{\eta,\lambda}$ in the second equation of (\ref{eq2-4-0}) and take the real part of the equation to get
\begin{equation}\label{eq2-8-3}
    \Vert \beta_{\eta,\lambda} \mathbf{div}\,{\bf J}_{\eta,\lambda} \Vert^{2}_{{\bf L}^{2}(\Omega_{s})} = \omega^{2}\Vert {\bf J}_{\eta,\lambda} \Vert^{2}_{{\bf L}^{2}(\Omega_{s})}+ \omega\omega_{p}^{2}\varepsilon_{0} {\rm Im}\,\big({\bf E}_{\eta,\lambda},\,{\bf J}_{\eta,\lambda}\big)_{s}\leq C\|{\bf g}\|^2_{\mathbf{L}^2(\partial \Omega)},
\end{equation}
where we have used (\ref{eq2-8-2}). Then the second inequality in (\ref{eq2-7-1}) follows immediately from (\ref{eq2-8-3}) and the fact that $\beta_{\eta,\lambda}|_{\Omega_{s}/\Omega_{s,\eta}} = \lambda^{\frac12}$. \qquad  \end{proof}

Next we construct a function $\mathbf{w} \in {\bf H}(\rm {\bf div},\Omega_{s,\eta})$ that has the same normal trace on $\partial \Omega_{s,\eta}$ as ${\bf J}_{\eta,\lambda}$.
\begin{lemma}\label{lemma2}
Let $({\bf E}_{\eta,\lambda}, {\bf J}_{\eta,\lambda})$ be the solution of the extended system (\ref{eq2-4-0}). Then there exists a $\mathbf{w} \in {\bf H}(\rm {\bf div},\Omega_{s,\eta})$ such that $\mathbf{n} \cdot \mathbf{w} = \mathbf{n} \cdot {\bf J}_{\eta,\lambda}$ on $\partial \Omega_{s,\eta}$, and
\begin{equation}
 \Vert{\bf w}\Vert_{{\bf H}(\rm {\bf div},\Omega_{s,\eta})}\leq C/\lambda^{1/2},
\end{equation}
where $C$ is a positive constant independent of $\lambda$.
\end{lemma}
\begin{proof}
  By Lemma \ref{lemma1} and \cite[Theorem 3.24]{Monk}, we find
        \begin{equation}\label{eq:2-4-3}
            \Vert \mathbf{n}\cdot {\bf J}_{\eta,\lambda} \Vert_{{ H}^{-\frac12}(\partial\Omega_{s,\eta})} \leq C\Vert{\bf J}_{\eta,\lambda}\Vert_{{\bf H}({\rm {\bf div}},\Omega_{s}/\Omega_{s,\eta})}\leq C/\lambda^{1/2}.
        \end{equation}
        Let $\phi\in H^1(\Omega_{s,\eta})$ be the solution of the Neumann problem
        \begin{equation}\label{eq:2-4-4}
            (\mathbf{grad}\, \phi,\mathbf{grad}\, \psi)_{s,\eta}+(\phi,\psi)_{s,\eta}=<\mathbf{n}\cdot {\bf J}_{\eta,\lambda},\,\psi>_{{ H}^{-\frac12}(\partial\Omega_{s,\eta}),{ H}^{\frac12}(\partial\Omega_{s,\eta})}
        \end{equation}
        for all $\psi\in H^{1}(\Omega_{s,\eta})$. Taking $\phi$ as the test function in (\ref{eq:2-4-4}) gives
        \begin{equation*}
        \begin{array}{lll}
        {\displaystyle  \Vert \phi\Vert_{{H}^{1}(\Omega_{s,\eta})}^{2} \leq \,| <\mathbf{n}\cdot {\bf J}_{\eta,\lambda},\,\phi>_{{ H}^{-\frac12}(\partial\Omega_{s,\eta}),{ H}^{\frac12}(\partial\Omega_{s,\eta})} | }\\[2mm]
        {\displaystyle \quad \leq \Vert \mathbf{n}\cdot {\bf J}_{\eta,\lambda} \Vert_{{ H}^{-\frac12}(\partial\Omega_{s,\eta})}\Vert \phi\Vert_{{ H}^{\frac12}(\partial\Omega_{s,\eta})}\leq C\Vert \mathbf{n}\cdot {\bf J}_{\eta,\lambda}\Vert_{{ H}^{-\frac12}(\partial\Omega_{s,\eta})}\Vert \phi\Vert_{{ H}^{1}(\Omega_{s,\eta})},}
            \end{array}
   \end{equation*}
   which yields
        \begin{equation}\label{eq:2-4-5}
            \Vert \phi\Vert_{{ H}^{1}(\Omega_{s,\eta})}\leq C/\lambda^{1/2}
        \end{equation}
        by a use of (\ref{eq:2-4-4}).

        Let ${\bf w}=\mathbf{grad}\,\phi\in {\bf L}^2(\Omega_{s,\eta})$. Taking $\psi \in \mathcal{C}_0^{\infty}(\Omega_{s,\eta})$ in (\ref{eq:2-4-4}), we see that
        \begin{equation}
            ({\bf w},\mathbf{grad}\, \psi)_{s,\eta}+(\phi,\psi)_{s,\eta}=0
        \end{equation}
        for all $\psi\in \mathcal{C}_0^{\infty}(\Omega_{s,\eta})$, which implies that $\mathbf{div}\, {\bf w}=\phi\in L^2(\Omega_{s,\eta})$. Thus ${\bf w}\in {\bf H}(\rm {\bf div},\Omega_{s,\eta})$ satisfies $\mathbf{n} \cdot \mathbf{w} = \mathbf{n} \cdot {\bf J}_{\eta,\lambda}$ on $\partial \Omega_{s,\eta}$, and
        \begin{equation}
            \Vert{\bf w}\Vert_{{\bf H}(\rm {\bf div},\Omega_{s,\eta})}=\Vert \phi\Vert_{{ H}^{1}(\Omega_{s,\eta})}\leq C/\lambda^{1/2},
        \end{equation}
by using (\ref{eq:2-4-5}). \qquad \end{proof}

Now we turn to the proof of Theorem~\ref{thm:2.1}.

Let ${\bf e}={\bf E}_{\eta,\lambda}-{\bf E}_{\eta}$ and ${\bf j}={\bf J}_{\eta,\lambda}-{\bf J}_{\eta}-{\bf w}$. Note that $\mathbf{n}\cdot\mathbf{j} = 0$ on $\partial \Omega_{s,\eta}$ and thus $\mathbf{j}\in Y(\Omega_{s,\eta})$. By subtracting (\ref{eq2-4}) from (\ref{eq2-4-0}) and recalling that $\gamma_{\eta,\lambda}|_{\Omega_{s,\eta}} = \gamma$ and $\beta^{2}_{\eta,\lambda}|_{\Omega_{s,\eta}} = \beta^{2}$, we see that $({\bf e},\,{\bf j})$ satisfies
        \begin{equation}\label{eq:2-4-2}
\left\{
\begin{array}{@{}l@{}}
    {\displaystyle  \big(\tilde{\mu}_{\eta}^{-1} \mathbf{curl}\,{\bf e}, \,\mathbf{curl}\,{\bf u}) - \omega^{2}(\tilde{\varepsilon}_{\eta}{\bf e}, \,{\bf u}) - {\mathrm {{i}}}\omega \langle {\bf e}_{T}, {\bf u}_{T}\rangle =  {\mathrm {{i}}}\omega\big({\bf j}, \,{\bf u}\big)_{s,\eta} }\\[2mm]
    {\displaystyle \qquad \qquad \qquad \qquad \qquad +\,{\mathrm {{i}}}\omega\big({\bf {\bf{w}}}, \,{\bf u}\big)_{s,\eta}+\mathrm{i}\omega\big(\mathbf{J}_{\eta,\lambda},\, \mathbf{u}\big)_{\Omega_{s}/\Omega_{s,\eta}},}\\[3mm]
    {\displaystyle \beta^{2} \big(\mathbf{div}\, {\bf j},\,\mathbf{div}\, {\bf v}\big)_{s,\eta}-\omega(\omega+{\mathrm {{i}}}\gamma)({\bf j},\,{\bf v})_{s,\eta} = -{\mathrm {{i}}}\omega \omega_{p}^{2}\varepsilon_{0}\big({\bf e},\,{\bf v}\big)_{s,\eta}}\\[2mm]
    {\displaystyle \quad \quad \quad\quad \quad \quad \quad -\beta^{2} \big(\mathbf{div}\, {\bf w},\,\mathbf{div}\, {\bf v}\big)_{s,\eta}+\omega(\omega+{\mathrm {{i}}}\gamma)({\bf w},\,{\bf v})_{s,\eta}}
\end{array}
\right.
\end{equation}
for all $({\bf u},{\bf v})\in X(\Omega)\times Y(\Omega_{s,\eta})$. We will prove that $({\bf e},\,{\bf j})$ satisfies the estimate
\begin{equation}\label{eq:2-4-6}
\Vert {\bf e} \Vert_{X(\Omega)}+\Vert {\bf j}\Vert_{Y(\Omega_{s,\eta})}\leq C/{\lambda}^{1/2},
\end{equation}
where $C>0$ is independent of $\lambda$. Theorem~\ref{thm:2.1} follows from (\ref{eq:2-4-6}) and Lemma~\ref{lemma2}.

To prove (\ref{eq:2-4-6}), we first choose $(\mathbf{u},\,\mathbf{v})=(\mathbf{e},\,\mathbf{j})$ in (\ref{eq:2-4-2}) and take the imaginary part of the equations to get
\begin{equation}\label{eq:2-4-7}
\begin{array}{lll}
{\displaystyle \Vert \mathbf{e}_{T}\Vert^{2}_{\mathbf{L}^{2}(\partial \Omega)} = -{\rm Re} \,(\mathbf{j},\,\mathbf{e})_{s,\eta}-{\rm Re} \,(\mathbf{w},\,\mathbf{e})_{s,\eta}-\mathrm{Re}\,(\mathbf{J}_{\eta,\lambda},\, \mathbf{e})_{\Omega_{s}/\Omega_{s,\eta}},} \\[2mm]
{\displaystyle \gamma\Vert \mathbf{j}\Vert^{2}_{\mathbf{L}^{2}(\Omega_{s,\eta})} = \omega^{2}_{p}\varepsilon_{0}\,\mathrm{Re}\,(\mathbf{e},\,\mathbf{j})_{s,\eta}- \gamma\,\mathrm{Re}\,(\mathbf{w},\,\mathbf{j})_{s,\eta}.}
\end{array}
\end{equation}
Combining the two equations in (\ref{eq:2-4-7}) and using Lemma~\ref{lemma1} and~\ref{lemma2}, we obtain
\begin{equation*}
\begin{array}{lll}
{\displaystyle \omega^{2}_{p}\varepsilon_{0} \Vert \mathbf{e}_{T}\Vert^{2}_{\mathbf{L}^{2}(\partial \Omega)} + \gamma\Vert \mathbf{j}\Vert^{2}_{\mathbf{L}^{2}(\Omega_{s,\eta})}}\\[2mm]
{\displaystyle \leq \omega^{2}_{p}\varepsilon_{0} |(\mathbf{w},\,\mathbf{e})_{s,\eta}| +\omega^{2}_{p}\varepsilon_{0}|(\mathbf{J}_{\eta,\lambda},\, \mathbf{e})_{\Omega_{s}/\Omega_{s,\eta}}| +\gamma |(\mathbf{w},\,\mathbf{j})_{s,\eta}| }\\[2mm]
{\displaystyle  \leq C\big(\Vert \mathbf{w}\Vert_{\mathbf{L}^{2}(\Omega_{s,\eta})} + \Vert \mathbf{J}_{\eta,\lambda}\Vert_{\mathbf{L}^{2}(\Omega_{s}/\Omega_{s,\eta})} \big)\Vert \mathbf{e}\Vert_{\mathbf{L}^{2}(\Omega)}  + \frac{\gamma}{2} \Vert \mathbf{j}\Vert^{2}_{\mathbf{L}^{2}(\Omega_{s,\eta})}+ \frac{\gamma}{2} \Vert \mathbf{w}\Vert^{2}_{\mathbf{L}^{2}(\Omega_{s,\eta})}}\\[3mm]
{\displaystyle \leq {C}{\lambda^{-\frac12}}\Vert \mathbf{e}\Vert_{\mathbf{L}^{2}(\Omega)}  +  \frac{\gamma}{2} \Vert \mathbf{j}\Vert^{2}_{\mathbf{L}^{2}(\Omega_{s,\eta})}+ {C}{\lambda^{-1}},}
\end{array}
\end{equation*}
which implies that
\begin{equation}\label{eq:2-4-8}
\Vert \mathbf{j}\Vert_{\mathbf{L}^{2}(\Omega_{s,\eta})}\leq {C}{\lambda^{-\frac14}}\Vert \mathbf{e}\Vert^{\frac12}_{\mathbf{L}^{2}(\Omega)} +{C}{\lambda^{-\frac12}}.
\end{equation}
Now we consider the first equation of (\ref{eq:2-4-2}) separately in which $\mathbf{j}$ and $\mathbf{w}$ are considered as given sources. By employing Theorem~4.17 of \cite{Monk} again, we find
\begin{equation*}
\Vert {\bf e} \Vert_{X(\Omega)} \leq C\big(\Vert \mathbf{j}\Vert_{\mathbf{L}^{2}(\Omega_{s,\eta})} + \Vert \mathbf{w}\Vert_{\mathbf{L}^{2}(\Omega_{s,\eta})} + \Vert \mathbf{J}_{\eta,\lambda}\Vert_{\mathbf{L}^{2}(\Omega_{s}/\Omega_{s,\eta})} \big),
\end{equation*}
and consequently, with a use of (\ref{eq:2-4-8}), Lemma~\ref{lemma1}, and Lemma~\ref{lemma2}, we come to
\begin{equation*}
\Vert {\bf e} \Vert_{X(\Omega)} \leq {C}{\lambda^{-\frac14}}\Vert \mathbf{e}\Vert^{\frac12}_{\mathbf{L}^{2}(\Omega)} +{C}{\lambda^{-\frac12}}\leq \frac{1}{2} \Vert \mathbf{e}\Vert_{\mathbf{L}^{2}(\Omega)} + {C}{\lambda^{-\frac12}},
\end{equation*}
which yields that
\begin{equation}\label{eq:2-4-9}
\Vert {\bf e} \Vert_{X(\Omega)} \leq {C}{\lambda^{-\frac12}}.
\end{equation}
Next we consider the second equation of (\ref{eq:2-4-2}) separately with given sources $\mathbf{e}$ and $\mathbf{w}$. By applying a similar argument as in the proofs of Theorem~4.17 in \cite{Monk} and Theorem~3.1 in \cite{ma2019mathematical}, we get the following estimate
\begin{equation*}
\Vert {\bf j}\Vert_{Y(\Omega_{s,\eta})} \leq C\big(\Vert \mathbf{e}\Vert_{\mathbf{L}^{2}(\Omega_{s,\eta})} + \Vert{\bf w}\Vert_{{\bf H}(\rm {\bf div},\Omega_{s,\eta})} \big).
\end{equation*}
With Lemma~\ref{lemma2} and (\ref{eq:2-4-9}), it follows that
\begin{equation}\label{eq:2-4-10}
\Vert {\bf j}\Vert_{Y(\Omega_{s,\eta})}\leq C{\lambda^{-\frac12}}.
\end{equation}
Combining (\ref{eq:2-4-9}) and (\ref{eq:2-4-10}) gives (\ref{eq:2-4-6}) and we complete the proof of Theorem~\ref{thm:2.1}.

\bibliographystyle{gbt-7714-2015-numerical}
  \bibliography{ref}

\end{document}